\documentclass[reqno,11pt]{amsart}

\usepackage{amsmath,amsfonts,amssymb,amsthm,graphicx,color}
\usepackage{mathrsfs}
\usepackage{yhmath}

 \allowdisplaybreaks
\numberwithin{equation}{section}

\font\script=rsfs10 at 12pt
\def\eps{\varepsilon}
\def\A{{\mathcal A}}
\def\B{{\mathcal B}}
\def\C{{\mathcal C}}
\def\SS{{\mathcal S}}
\def\G{{\mathscr G}}
\def\Q{\mathcal Q}
\def\H{{\mbox{\script H}\,\,}}
\def\L{{\mbox{\script L}\,}}
\def\comp{\subset\subset}
\def\R{{\mathbb R}}
\def\RR{{\mathcal R}}
\def\T{{\mathcal T}}
\def\D{{\mathcal D}}

\def\F{{\mathcal F}}
\def\N{\mathbb N}
\def\bal{\begin{aligned}}
\def\eal{\end{aligned}}
\def\res{\mathop{\hbox{\vrule height 7pt width .5pt depth 0pt \vrule height .5pt width 6pt depth 0pt}}\nolimits}
\def\proofof#1{\begin{proof}[Proof of #1]}
\def\step#1#2{\par\noindent{\underline{\it Step~#1.}}\emph{ #2}\\}
\def\case#1#2{\par\noindent{\underline{\it Case~#1.}}\emph{ #2}\\}
\def\part#1#2{\par\noindent{\underline{\it Part~#1.}}\emph{ #2}\\}

\def\S{\mathbb S}
\def\arc#1{\wideparen{#1}}
\def\deg{{\rm deg}}

\def\u#1{\hbox{\boldmath $#1$}} 

\DeclareMathOperator{\diam}{\rm diam}
\DeclareMathOperator{\id}{\rm Id}
\DeclareMathOperator{\dist}{\rm dist}
\DeclareMathOperator{\osc}{\rm osc}

\def\XXint#1#2#3{{\setbox0=\hbox{$#1{#2#3}{\int}$} \vcenter{\vspace{-1pt}\hbox{$#2#3$}}\kern-.5\wd0}}

\theoremstyle{plain}
\newtheorem{lemma}{Lemma}[section]
\newtheorem{prop}[lemma]{Proposition}
\newtheorem{theorem}[lemma]{Theorem}

\newtheorem{corol}[lemma]{Corollary}
\newtheorem{defin}[lemma]{Definition}
\newtheorem{remark}[lemma]{Remark}

\def\de0#1{\rule[3pt]{#1}{0.4pt} \hspace{-0.1pt} \rule[3.05pt]{0.05pt}{0.4pt} \hspace{-0.1pt} \rule[3.1pt]{0.05pt}{0.4pt} \hspace{-0.1pt} \rule[3.15pt]{0.05pt}{0.4pt} \hspace{-0.1pt} \rule[3.2pt]{0.05pt}{0.4pt} \hspace{-0.1pt} \rule[3.25pt]{0.05pt}{0.4pt} \hspace{-0.1pt} \rule[3.3pt]{0.05pt}{0.4pt} \hspace{-0.1pt} \rule[3.35pt]{0.05pt}{0.4pt} \hspace{-0.1pt} \rule[3.4pt]{0.05pt}{0.4pt} \hspace{-0.1pt} \rule[3.45pt]{0.05pt}{0.4pt} \hspace{-0.1pt} \rule[3.5pt]{0.05pt}{0.4pt} \hspace{-0.1pt} \rule[3.55pt]{0.05pt}{0.4pt} \hspace{-0.1pt} \rule[3.6pt]{0.05pt}{0.4pt} \hspace{-0.1pt} \rule[3.65pt]{0.05pt}{0.4pt} \hspace{-0.1pt} \rule[3.7pt]{0.05pt}{0.4pt} \hspace{-0.1pt} \rule[3.75pt]{0.05pt}{0.4pt} \hspace{-0.1pt} \rule[3.8pt]{0.05pt}{0.4pt} \hspace{-0.1pt} \rule[3.85pt]{0.05pt}{0.4pt} \hspace{-0.1pt} \rule[3.9pt]{0.05pt}{0.4pt} \hspace{-0.1pt} \rule[3.95pt]{0.05pt}{0.4pt} \hspace{-0.1pt} \rule[4.0pt]{0.05pt}{0.4pt} \hspace{-0.1pt} \rule[4.05pt]{0.05pt}{0.4pt} \hspace{-0.1pt} \rule[4.1pt]{0.05pt}{0.4pt} \hspace{-0.1pt} \rule[4.15pt]{0.05pt}{0.4pt} \hspace{-0.1pt} \rule[4.2pt]{0.05pt}{0.4pt} \hspace{-0.1pt} \rule[4.25pt]{0.05pt}{0.4pt} \hspace{-0.1pt} \rule[4.3pt]{0.05pt}{0.4pt} \hspace{-0.1pt} \rule[4.35pt]{0.05pt}{0.4pt} \hspace{-0.1pt} \rule[4.4pt]{0.05pt}{0.4pt} \hspace{-0.1pt} \rule[4.45pt]{0.05pt}{0.4pt} \hspace{-0.1pt} \rule[4.5pt]{0.05pt}{0.4pt} \hspace{-0.1pt} \rule[4.55pt]{0.05pt}{0.4pt} \hspace{-0.1pt} \rule[4.6pt]{0.05pt}{0.4pt} \hspace{-0.1pt} \rule[4.65pt]{0.05pt}{0.4pt} \hspace{-0.1pt} \rule[4.7pt]{0.05pt}{0.4pt} \hspace{-0.1pt} \rule[4.75pt]{0.05pt}{0.4pt} \hspace{-0.1pt} \rule[4.8pt]{0.05pt}{0.4pt} \hspace{-0.1pt} \rule[4.85pt]{0.05pt}{0.4pt} \hspace{-0.1pt} \rule[4.9pt]{0.05pt}{0.4pt} \hspace{-0.1pt} \rule[4.95pt]{0.05pt}{0.4pt} \hspace{-0.1pt} \rule[5.0pt]{0.05pt}{0.4pt} \hspace{-0.1pt} \rule[5.05pt]{0.05pt}{0.4pt} \hspace{-0.1pt} \rule[5.1pt]{0.05pt}{0.4pt} \hspace{-0.1pt} \rule[5.15pt]{0.05pt}{0.4pt} \hspace{-0.1pt} \rule[5.2pt]{0.05pt}{0.4pt} \hspace{-0.1pt} \rule[5.25pt]{0.05pt}{0.4pt} \hspace{-0.1pt} \rule[5.3pt]{0.05pt}{0.4pt} \hspace{-0.1pt} \rule[5.35pt]{0.05pt}{0.4pt} \hspace{-0.1pt} \rule[5.4pt]{0.05pt}{0.4pt} \hspace{-0.1pt} \rule[5.45pt]{0.05pt}{0.4pt} \hspace{-0.1pt} \rule[5.5pt]{0.05pt}{0.4pt} \hspace{-0.1pt} \rule[5.55pt]{0.05pt}{0.4pt} \hspace{-0.1pt} \rule[5.6pt]{0.05pt}{0.4pt} \hspace{-0.1pt} \rule[5.65pt]{0.05pt}{0.4pt} \hspace{-0.1pt} \rule[5.7pt]{0.05pt}{0.4pt} \hspace{-0.1pt} \rule[5.75pt]{0.05pt}{0.4pt} \hspace{-0.1pt} \rule[5.8pt]{0.05pt}{0.4pt} \hspace{-0.1pt} \rule[5.85pt]{0.05pt}{0.4pt} \hspace{-0.1pt} \rule[5.9pt]{0.05pt}{0.4pt} \hspace{-0.1pt} \rule[5.95pt]{0.05pt}{0.4pt} \hspace{-0.1pt} \rule[6.0pt]{0.05pt}{0.4pt}}	
\def\deb{\mathop{\de0{16pt} ~}\limits}		

\newcounter{mt}
\def\maintheorem#1#2#3{\par \medskip \noindent {\bf Theorem~\mref{#1}}~(#2).~{\it #3}\par}
\def\mref#1{\Alph{#1}}
\def\maintheoremdeclaration#1{\stepcounter{mt}\newcounter{#1}\setcounter{#1}{\arabic{mt}}}

\maintheoremdeclaration{Caratt}
\maintheoremdeclaration{WeakStrong}
\maintheoremdeclaration{NonSconn}

\begin{document}

\title{The closure of planar diffeomorphisms in Sobolev spaces}

\author{G. De Philippis}
\author{A. Pratelli}

\begin{abstract}
We characterize the (sequentially) weak and strong closure of planar diffeomorphisms in the Sobolev topology and we show that they always coincide. We also provide some sufficient condition for a planar map to be approximable by diffeomorphisms in terms of the connectedness of its counter-images, in the spirit of Young's characterisation of monotone functions. We finally show that the closure of diffeomorphisms in the Sobolev topology is strictly contained in the class \(INV\) introduced by M\"uller and Spector.
\end{abstract}

\maketitle

\section{Introduction}

\subsection{General overview}
In a series of papers~\cite{Y,Y1,Y2} J.W.T. Youngs characterised the closure in the uniform topology of the set of homeomorphisms between two \(2\)-dimensional topological manifolds, by showing that it coincides with the class of monotone maps, i.e. those contiuous maps for which the counter-image of a point is a  connected set. 

By a trivial gluing argument the same characterisation holds for the uniform closure of the set of homeomorphisms from the two dimensional unit square \(\Q:=[0,1]^2\) into itself which are equal to the identity on \(\partial \Q\) and actually the result can be extended to more general domains, see for instance~\cite[Section 2]{IO}.

In recent years, mainly motivated by applications in non-linear elasticity and in geometric function theory, a great deal has been devoted in understanding which is the closure of the class of diffeomorphisms in the weak and/or strong Sobolev topology. In view of applications to non-linear elasticity, it is of particular interest  to understand these relations in the case of maps in \(W^{1,p}\) for \(p\in (1,2)\). Indeed   this allows for discontinuous maps which, in principle, may be used to model cavitations, see~\cite{Ball}. In this respect, note that the map \(x\mapsto x/|x|\) can be easily obtained as strong \(W^{1,p}\) limit of diffeomorphisms if \(p\in[1,2)\).

To state the known results and for further use let us first introduce a few notations. We will mainly restrict to functions mapping the closed unit square \(\Q=[0,1]^2\) into itself (see however Remark~\ref{rmk:changeofvariables}). For the sake of clarity we will denote by \(\Q\) the ``domain'' unit square and by \(\u\Q=[0,1]^2\) the ``target'' unit square. In general we will denote by capital letters points and sets in \(\Q\) and by bold capital letter points  and sets in \(\u \Q\). Hence  \(u:\Q\to \u \Q\), and \(A,B,\dots\in\Q\) while \(\u A,\u B,\dots\in\u\Q\). We also set
\begin{gather*}
\mathcal H_{\id}(\Q)=\Big\{u\in{\rm C}^0(\Q;\u\Q)\textrm{ homeomorphism such that \(u=\id\) on \(\partial \Q\)}\Big\}\,,\\
\D_{\id} (\Q)=\Big\{u\in{\rm C}^\infty(\Q;\u\Q)\textrm{ diffeomorphism such that \(u=\id\) on \(\partial \Q\)}\Big\}\,,\\
\mathcal M_{\id} (\Q)=\Big\{u\in{\rm C}^0(\Q;\u\Q)\textrm{ monotone map such that \(u=\id\) on \(\partial \Q\)}\Big\}\,.
\end{gather*}
Recall that a continuous map \(u:\Q\to \u \Q\) is said to be {\em monotone} if for every \(\u P\in \u \Q\) the set \(u^{-1}(\u P)\subseteq \Q\) is connected. Moreover, for any subset \(\F\subseteq W^{1,p}(\Q;\u\Q)\) we denote by 
\[
\overline{\F}^{\,\textrm{s}-W^{1,p}}\qquad\textrm{and}\qquad \overline{\F}^{\,\textrm{ws}-W^{1,p}}
\]
the closure of \(\F\) in the {\em strong} topology and its {\em weak sequential} closure, i.e., the smallest set containing \(\F\) which is sequentially closed\footnote{Note that in general the weak sequential closure of \(\F\) is strictly smaller than the closure of \(\F\) in the weak topology and strictly larger than the set of all weak limit points of sequences in \(\F\), see for instance~\cite[Volume 1, Section 4.1]{GMS} for a discussion of this fact.}. With these notations we can state the following theorem which summarises the work of several authors.
\begin{theorem}\label{thm:summ}
The following statements hold:
\begin{itemize}
\item[(i)] Let \(u\in W^{1,p}(\Q)\cap \mathcal H_{\id}(Q) \) and \(p\in[1,+\infty)\). Then there exists a sequence \(\{u_j\}\subseteq \mathcal D_{\id}(\Q)\) such that \(\|u_j-u\|_{W^{1,p}(\Q)}\to 0\).
\item[(ii)] Let \(u\in W^{1,p}(\Q)\cap \mathcal M_{\id}(Q)\) and \(p\in(1,+\infty)\). Then there exists a sequence \(\{u_j\}\subseteq \mathcal D_{\id}(\Q)\) such that \(\|u_j-u\|_{W^{1,p}(\Q)}\to 0\).
\item[(iii)] For \(p\in [2,+\infty)\), 
\[
\overline{\mathcal D_{\id}(\Q)}^{\,\mathrm{ws}-W^{1,p}}=\overline{\mathcal D_{\id}(\Q)}^{\,\mathrm{s}-W^{1,p}}=\mathcal M_{\id}(\Q)\cap W^{1,p}(\Q)\,.
\]
\end{itemize}
\end{theorem}
Point~(i) has been proved in~\cite{IKO,IKO2} for \(p\in(1,+\infty)\) and~\cite{HP} for \(p=1\). See also~\cite{C} for the case of Orlicz-Sobolev spaces and~\cite{DP} for the case of bi-Lipschitz maps. Point~(ii) as been proved in~\cite{IO} and point~(iii) in~\cite{IO2}. See also  the forthcoming paper~\cite{CORT} where  point~(ii) is extended   to the case $p=1$.

We remark that for the sake of simplicity we have stated the above theorem for the case of maps from $\Q$ to $\u\Q$ coinciding  with the identity on the boundary. However  the quoted results actually work for more general boundary conditions and domains,  see in particular~\cite{IO}. Note however that the results easily extend at least to the case of domains that are bi-Lipschitz equivalent to \(\Q\) and of bi-Lipschitz boundary data, see Remark~\ref{rmk:changeofvariables}.

\subsection{Main results}
The first main result of this paper is a characterisation of \(\overline{\mathcal D_{\id}(\Q)}^{\,\textrm{s}-W^{1,p}}\) for {\em any} \(p\in [1+\infty)\). As a consequence we will obtain the second main result, which states that
\[
\overline{\mathcal D_{\id}(\Q)}^{\,\textrm{s}-W^{1,p}}=\overline{\mathcal D_{\id}(\Q)}^{\,\textrm{ws}-W^{1,p}},
\]
thus extending point~(iii) of Theorem~\ref{thm:summ} to \(p\in[1,2)\), see Theorem~\mref{WeakStrong} below. To state our theorem we first need to give the following definition. 

\begin{defin}[No-crossing condition]\label{defnocross}
Let $u:\Q\to\u\Q$ be a Borel map. We say that \emph{$u$ satisfies the no-crossing condition} if there is a function \(\tilde u:\Q\to\u\Q \) such that \(\tilde u=u\) \(\L^2\) almost everywhere and such that the following three conditions are satisfied:
\begin{itemize}
\item[(i)] \(\tilde u=\id\) on \(\partial \Q\).
\item[(ii)] There exists a $\H^1$-negligible  set $\mathcal S_{\tilde u}\subseteq\Q$ such that $\tilde u$ is continuous at any point of $\Q\setminus\mathcal S_{\tilde u}$.
\item[(iii)] Let $N\in\N$, and for any $1\leq i \leq N$ let $\gamma_i:[0,1]\to \Q\setminus\mathcal S_{\tilde u}$ be a Lipschitz, injective curve, such that for any $i\neq j$ the intersection $\gamma_i\cap\gamma_j$ is either empty, or a single point $P_{ij}=\gamma_i(t_i)=\gamma_j(t_j)$. In the second case, assume also that  both $\gamma_i'(t_i)$ and $\gamma_j'(t_j)$ exist and are not parallel, and $P_{ij}\neq P_{kl}$ whenever $\{i,\,j\} \neq \{k,\,l\}$ and both the points are defined. Then, setting $\Gamma=\cup_{i=1}^N\gamma_i([0,1])$, for every \(\eps>0\) there exists a continuous and injective map $\psi_\eps:\Gamma\to\u\Q$, coinciding with the identity on $\partial\Q\cap\Gamma$ and such that $\|{\tilde u}-\psi_\eps\|_{L^\infty(\Gamma)}<\eps$.
\end{itemize}
We collect the maps satisfying the no crossing condition in the set 
\[
\mathcal{NC}_{\id} (\Q)=\Big\{ u: \Q\to \u\Q \textrm{ such that \(u\) satisfies the no-crossing condion}\Big\}\,.
\]
\end{defin}
Roughly speaking, a map satisfies the no-crossing condition if every time that we restrict it to a family of possibly intersecting curves, it is possible to make it injective up to a small error in \(L^\infty\). We have then the following result.
\maintheorem{Caratt}{Characterisation of the closure of diffeomorfisms}{Let $\Q=\u\Q=[0,1]^2$, let $p\in [1,+\infty)$ and let $u:\Q\to\u\Q$ be a function. Then, there exists a sequence of diffeomorphisms $\{u_j\}\subseteq \mathcal D_{\id} (\Q)$ strongly converging in $W^{1,p}$ to $u$, if and only if $u\in\mathcal{NC}_{\id} (\Q)\cap W^{1,p}(\Q) $. Moreover, the sequence $\{u_j\}\subseteq \mathcal D_{\id} (\Q)$ can be chosen such that:
\begin{itemize}
\item[(i)] If \(p\in [1,2]\), for every \(j\) there exists a compact set \(K_j\) such that:
\[
\H_{\infty}^{2-p}(K_j)\le 2^{-j}\qquad\textrm{and}\qquad \|u-u_{j}\|_{L^\infty(\Q\setminus K_j)}\le 2^{-j}.
\]
In particular, defining $K:=\bigcap_{\ell\ge 1} \bigcup_{j\ge \ell} K_j$, one has \(\H^{2-p}(K)=0\) and \(u_j(x)\to u(x)\) for all \(x\in \Q\setminus K\).
\item[(ii)] If \(u\in C^{0}(\Q)\) (which is always the case if \(p\geq 2\)), then \( \|u-u_{j}\|_{L^\infty(\Q)}\le 2^{-j}\). 
\end{itemize}
}

\begin{remark}
Note that in the above theorem we recover the (known) fact that \(W^{1,2}\)-limits of diffeomorphisms are continuous. By Theorem~\mref{WeakStrong} below, the same is true for weak limits.
\end{remark}

Since $\overline{\mathcal D_{\id}(\Q)}^{\,\textrm{ws}-W^{1,p}}\subseteq \mathcal{NC}_{\id} (\Q)$, see Section~\ref{sec:DimB}, from Theorem~\mref{Caratt} we deduce the following.

\maintheorem{WeakStrong}{Equality of weak and strong closure of diffeomorfisms}{Let $\Q=\u\Q=[0,1]^2$. Then, for every $p\in [1,+\infty)$ one has
\[
\overline{\mathcal D_{\id}(\Q)}^{\,\mathrm{s}-W^{1,p}}=\overline{\mathcal D_{\id}(\Q)}^{\,\mathrm{ws}-W^{1,p}}\,.
\]
}

\begin{remark} Note that in the above Theorem we are speaking about {\em weak} convergence. In particular, when \(p=1\), this requires  the {\em equi-integrability} of the sequence of gradients.
\end{remark}

By the results of Young, a monotone map satisfies the no-crossing condition (hence a corollary of Theorem~\mref{Caratt} is that Theorem~\ref{thm:summ}~(ii) is true also for $p=1$). Since the monotonicity is expressed in terms of the connectedness of \(u^{-1}(\u P)\), it is interesting to understand if the same condition (suitably extended to Sobolev maps) can characterise the weak and/or strong closure of difffeomorphisms. As we shall see with the counterexamples of Section~\ref{sec:Ex} this is actually not the case. However, we can show that a stronger condition is actually sufficient to be a  limit of diffeomorphisms.

To understand the condition in Theorem~\mref{NonSconn} below, two remarks are in order. First of all, note that for a {\em continuous monotone} map \(u:\Q\to \u \Q\), it is not hard to show that if $\u C\subseteq \u\Q$ is a closed, connected set which does not disconnect $\u\Q$, i.e. such that $\u\Q\setminus \u C$ is connected, then $u^{-1}(\u C)\subseteq \Q$ is a closed, connected set which does not disconnect $\Q$, see Lemma~\ref{condThC}.

Second, by the ideas of Sverak and M\"uller--Spector,~\cite{MS,Sve}, to any Sobolev map \(u\in W^{1,p}\) we can associate a uniquely defined {\em multivalued map} \(u:\Q\to 2^{\u \Q}\), see Definition~\ref{def:multifunction}. This allows to define the counter-image of subsets of $\u\Q$ as well, as those points whose image intersects the set, that is, for every $\u C\subseteq\u\Q$ we let $u^{-1}(\u C)=\big\{P\in \Q: u(P)\cap \u C \neq \emptyset\big\}$. Our sufficient condition reads then as follows.

\maintheorem{NonSconn}{Sufficiency of the non-disconnecting property}{Let $\Q=\u\Q=[0,1]^2$, let $1\leq p<\infty$, and let $u:\Q\to\u\Q$ be a $W^{1,p}$ function, coinciding with the identity on the boundary. Suppose that, whenever $\u C\subseteq \u\Q$ is a closed, connected set which does not disconnect $\u\Q$, then $u^{-1}(\u C)\subseteq \Q$ is a closed, connected set which does not disconnect $\Q$. Then, $u$ satisfies the no-crossing condition.}

\begin{remark}\label{rmk:changeofvariables}
All the above results are easily seen to be invariant by a bi-Lipschitz change of variables. Hence, they can be immediately extended to the class of maps \(u:\Omega_1\to \Omega_2\) such that \(u=\Phi\) on \(\partial \Omega_1\), where $\Omega_1$ is a bounded, bi-Lipschitz, simply connected domain, and \(\Phi:\Omega_1\to \Omega_2\) is a bi-Lipschitz function. Moreover, the  same technique of the proofs applies verbatim to show that the same results are true if the set \(\Omega_1\) is bi-Lipschitz equivalent to a standard multi-connected domain (i.e. a disk with finitely many compactly contained and disjoint disks removed) and the boundary data extends to a bi-Lipschitz map from $\Omega_1$ to $\Omega_2$. It would be of interest to understand how to extend the above results in the case of boundary data which extends to a monotone mapping. It does not seem that our techniques straightforwardly apply to this case.\end{remark}

\subsection{M\"uller and Spector \(INV\) condition and counterexamples} 
As we already mentioned, in the study of vectorial variational problems, there has been a great interest in understanding weak notions of invertibility which are closed by weak Sobolev limit and which can include discontinuous map.

In~\cite{MS}, M\"uller and Spector introduced a condition which they called \(INV\) (compare also with the weak diffeomorphism introduced by Giaquinta, Modica and Soucek~\cite[Volume 2, Chapter 2]{GMS}) and which since then has revealed to be a reasonable substitute of invertibility, in particular \(INV\) maps with almost everywhere positive Jacobian share some of reasonable properties one would expect from smooth diffeomorphisms. These made them a natural class of maps on which to settle vectorial variational problems which would be naturally settled on the class of diffeomorphims.

Refering the reader to Section~\ref{sec:INV} for the precise definition and some of the main properties of \(INV\) functions, let us just say here that the \(INV\) condition does not allow the interpenetration of matter, even after the formation of cavitation.

Since it is not hard to see that the \(INV\) condition is closed by (sequentially) \(W^{1,p}\) weak limits and since, trivially, diffeomorphisms satisfy this condition, we deduce 
\[
\overline{\mathcal D_{\id}(\Q)}^{\,\textrm{s}-W^{1,p}}\subseteq \overline{\mathcal D_{\id}(\Q)}^{\,\textrm{ws}-W^{1,p}}\subseteq INV(\Q),
\] 
see Corollary~\ref{cor:diff}.

A natural conjecture (to which also the authors have believed for a while) is then that the closure of diffeomorphisms actually coincides with \(INV\) maps. This is however not the case, not even for \(INV\) maps with almost everywhere positive Jacobian (a class of maps which enjoy better properties, as mentioned above). Indeed, in Section~\ref{terzo} we will construct a map which satisfies the \(INV\) condition, has positive Jacobian almost everywhere and cannot be approximated by diffeomorphisms. Furthermore, we will show that the \(INV\) condition is in general not invariant by reparametrization of the domain, while this invariance holds true under the additional assumption of positive Jacobian almost everywhere, as proved in~\cite[Section 9]{MS}).
 
Finally, again in Section~\ref{sec:Ex}, we will present an example showing that in Theorem~\mref{NonSconn} it is not enough to require that the counter-image of any point is a connected set (in particular, not requiring that it does not disconnect) as in the case of continuous maps (see Theorem~\ref{thm:summ}~(ii)).

We conclude the introduction by mentioning that we actually do not know if in Theorem~\mref{NonSconn} the requirement that $u^{-1}(\u C)$ is a connected set that does not disconnect $\Q$ for every connected set $\u C$ which does not disconnect $\u\Q$ can be weakened by restricting oneself only to the case when $\u C$ is a point, see also Remark~\ref{onlyptsgsq}. Keep in mind that, for continuous maps, these two conditions are easily seen to be equivalent, and in fact they coincide with the monotonicity, i.e., to the mere requirement that the counter-image of a point is connected, see Lemma~\ref{condThC}.

\subsection{Organisation of the paper and structure of the proof}

The paper is organised as follows. In Section~\ref{sec:prelimin} we present some technical definitions and their main properties that we are going to use, in particular we will speak about Sobolev functions, degree, INV condition and Lebesgue squares. In Sections~\ref{sec:A} and~\ref{sec:C} we prove Theorems~\mref{Caratt}, \mref{WeakStrong} and~\mref{NonSconn}, and in Section~\ref{sec:Ex} we provide some examples and counterexamples.

\medskip

We now briefly explain the main ideas behind the proofs of the main results :The main difficulty in the proof of the sufficiency part of Theorem~\mref{Caratt} is that the maps we are dealing with are both discontinuous and non-injective. When dealing with continuous maps, the idea of Young to solve the issue of non injectivity was to perform a fine partition on the target domain \(\u\Q\) with small squares \(\u Q_i\). The map is then approximated separately on each \(u^{-1}(\u Q_i)\) by a bijective map. This same strategy is performed in the case of Sobolev functions in~\cite{IKO,IKO2} thanks to the \(p\)-Laplacian version of the Rad\'o-Knesser theorem proved in~\cite{AS}. On the other hand, the strategy performed in~\cite{HP} is based on a suitable decomposition of the domain and to several extension lemmas which allow to extend an injective boundary data to a diffeomorphism in the interior with a suitable control on the Sobolev norms.
 
The proof of Theorem~\mref{Caratt} mixes both strategies. We start by constructing a ``good'' grid \(\G\) on the domain \(\Q\) such that \(u\in W^{1,p}(\G)\) and there are good estimates on the Sobolev norm of the function when restricted to its edges, see Lemma~\ref{startinggrid}. Moreover, the grid can be chosen such that on a good percentage of the squares the function is very close to an affine one. The aim is now to construct a good injective approximation of \(u\res \G\) and then to extend it to a piecewise linear homeomorphism on each square of the grid. To this end we start by constructing a grid \(\widetilde\G\) on the target domain \(\u \Q\), with side length much smaller then the one of \(\G\) and such that the image of \(u(\G)\) intersects \(\widetilde \G\) transversally, see Lemma~\ref{existencearrival}.

Once this second grid has been constructed, we consider \(\G\cap u^{-1}(\widetilde \G) \), which is a finite collection of points dividing each side of \(\G\) in finitely many intervals. The idea will be then to substitute on each of these segments the function with an affine  one. If we do so by keeping the values at the points of \(\G\cap u^{-1}(\widetilde \G) \), this will preserve the good control on the Sobolev norms, but the resulting function might be non-injective; indeed, there can be two or more points in \(\G\) whose image is the same point in \(\widetilde \G\). To solve this issue, we will make use of the no-crossing condition in order to slightly separate the images of the points in $\G\cap u^{-1}(\widetilde\G)$, so to get a piecewise linear and injective approximation on $\G$ without destroying the control on the Sobolev norms, see Proposition~\ref{injectivelinearization}. At this point, we will use known extension results to obtain a piecewise affine homeomorphism which approximates \(u\). The smooth approximation is then obtained thanks to the results in~\cite{MCP}. Theorem~\mref{WeakStrong} will then easily follow from Theorem~\mref{Caratt}.
 
Concerning Theorem~\mref{NonSconn}, the initial strategy is the same as the one described above for the proof of Theorem~\mref{Caratt}. However, we cannot rely anymore on the no-crossing condition in order to define an injective function on the grid. Hence, we need another rule to separate the images of the points in \(\G\cap u^{-1}(\widetilde\G)\). It is here that a key role is played by the assumption that \(u^{-1}(\u P)\) is a connected set that does not disconnect $\u\Q$ for every $\u P\in\u\Q$. Indeed, this allows to define the relation of \emph{right} and \emph{left connectibility} among the points in \(u^{-1}(\u P)\cap \G\) and consequently to ``blow-up'' \(\u P\in u(\G)\cap \widetilde\G\), see Section~\ref{sec:lr}. Verifying that this rule is actually well defined and provides a injective map is the main content of Chapter~\ref{sec:C}.

\section{Some preliminary definitions and facts\label{sec:prelimin}}
In this Section we collect some preliminary definitions and some technical facts, mainly known. In particular, in Section~\ref{sec:INV} we discuss the definition of degree and the INV condition, while in Section~\ref{sec:extension} we present the definition of ``Lebesgue squares'' and we recall some extension results.

\subsection{The definition of degree and the INV condition}\label{sec:INV}
Recall that, given a map \(u\in W^{1,1}(\Q)\), the limit
\[
\lim_{r\to 0} \frac{1}{\pi r^2} \int_{B(P,r)} u(x)\,dx
\]
exists for all \(P\) outside a \(\H^1\)-negligible set. We define the \emph{precise representative} of \(u\) as the map \(u^*(P)=(u^*_1(P),u^*_2(P))\) given by
\[
u_i^*(P)=\limsup_{r\to 0} \frac{1}{\pi r^2} \int_{B(P,r)} u_i(x)\,dx\,.
\]
Note that, by the \(1\)-dimensional Sobolev embedding, \(u^*\) is continuous when restricted to ``almost every'' curve. In particular we have the following simple remark.

\begin{remark}\label{rmk:degree}
Let \(u\in W^{1,1}(\Q)\), \(D\comp\Q\) be a domain with Lipschitz  boundary, and for \(\eps\) small let 
\[
D_{\eps}=\{x: {\rm sdist}(x,\partial D)\le \eps\}\comp Q
\]
where \({\rm sdist}\) denotes the signed distance from \(\partial D\). Then, \(u^*\res \partial D_\eps\) is continuous for almost every small $\eps$.
\end{remark}

Let us start with the definition of degree for a continuous planar map.

\begin{defin}
Let $\alpha:\S^1\to\R^2$ be a continuous curve and let \(u:\R^2\to\R^2\) be a function such that \(u\circ \alpha\in C^0(\S^1)\). Then for any \(\u P\in\R^2\setminus u(\alpha(\S^1))\) we define \(\deg(\u P,u\circ\alpha)\) as the winding number of the curve \(u\circ \alpha\) with respect to \(\u P\). 
\end{defin}

By Remark~\ref{rmk:degree}, it is thus possible to define the degree of a \(W^{1,1}\) function when restricted to ``almost every'' curve. If \(\alpha\) parametrises the boundary of a domain \(\partial D\) on which \(u^*\) is continuous, we will sometimes write \(\deg(\u P,u^*\res \partial D)\) to denote the degree of \(\u P\) with respect to \(u^*\circ \alpha\). 

We now show how, it is possible to associate to every function \(u\in W^{1,1}(\Q)\) a multivalued function defined at {\em every} point as follows, see~\cite{MS,Sve}.


\begin{defin}\label{def:multifunction}
Let $u:\Q\to\u\Q$ be a $W^{1,1}$ function. For every Lipschitz open set \(D\subseteq \Q\) such that \(u^*\res \partial D\) is continuous, we set
\[
F_u(D)=\big\{ \u P\in\u\Q: \deg(\u P,u^*\res \partial D)\ne 0\big\}\cup u^*(\partial D)\,.
\]
Given \(P\in \Q\), for and $\eps>0$ we denote by
\[
\mathcal D_{P,\eps}=\big\{ D\subseteq \Q: D\textrm{ open},\ P\in D\subseteq B(P,\eps) \textrm{ and \(u^*\res \partial D\) is continuous}\big\}\,,
\]
where $B(P,\eps)$ is the ball centered in $P$ with radius $\eps$, and we finally define the set
\[
u_{\rm multi}(P)=\Big\{\u P\in\u\Q:\, \exists\, \eps_n\searrow 0,\, D_n\in \mathcal D_{P,\eps_n},\, \u P_n\in F_u(D_n),\, \u P_n\to \u P\Big\}\,.
\] 
Notice that $u_{\rm multi}(P):\Q\to 2^{\u\Q}$ is a pointwise defined multivalued map; moreover, for every $P\in \Q$ the set $u_{\rm multi}(P)$ is by definition a non-empty closed set. With a small abuse of notation, we will often refer to $u_{\rm multi}$ simply as $u$ and, accordingly, for any \(\u C\subseteq \u\Q\) we define 
\[
u^{-1} (\u C)=\big\{P\in \Q: u_{\rm multi}(P)\cap \u C\neq\emptyset\big\}\,.
\]
\end{defin}
Notice that the above definition is consistent; in particular, whenever $P$ is a continuity point for $u^*$ the set $u_{\rm multi}(P)$ reduces to the sole point $u^*(P)$. Moreover, it is easy to show that the same happens if \(P\) is a Lebesgue point for both \(u\) and \(Du\). In particular, the multivalued function $u_{\rm multi}$ and the precise representative \(u^*\) coincide almost everywhere.

We can now report the definition of \(INV\) map as introduced by M\"uller and Spector in~\cite{MS}. see also~\cite{CD}.

\begin{defin}[INV and INV$^+$ conditions]\label{defINV}
Let $u:\Q\to\u\Q$ be a $W^{1,1}$ function. 
One says that \emph{$u$ satisfies the \(INV\) condition} if for every \(P\in {\rm Int} (\Q)\) and almost every \(r<\dist(P,\partial \Q)\) such that $u^*$ is continuous on $\partial B(P,r)$ we have
\begin{itemize}
\item[(i)]\(u(Q)\in F_u(B(P,r))\) for a.e. \(Q\in B(P,r)\);
\item[(ii)] \(u(Q)\in \u\Q\setminus F_u(B(P,r))\cup u^*(\partial B(P,r))\) for a.e. \(Q\in \Q\setminus B(P,r)\).
\end{itemize}
One says that \emph{\(u\) satisfies the \(INV^+\) condition} if \(u\circ \Phi\) satisfies the \(INV\) condition for every bi-Lipschitz map \(\Phi:\Q\to \Q\). We collect all \(INV\) (resp., \(INV^+\)) maps into the set \(INV(\Q)\) (resp., \(INV^+(\Q)\)).
\end{defin}

Roughly speaking, a map satisfies condition \(INV\) if for most balls the image of what is inside (resp., outside) remains inside or on the boundary (resp., outside or on the boundary). Condition \(INV^+\) requires this for more general sets instead than just for balls. 

Clearly, condition \(INV^+\) is invariant by re-parametrization of the domain, while it is less clear whether the same holds true for condition \(INV\). This is actually the case if \(\det D u>0\) a.e., see~\cite[Section 9]{MS}, but as we will show in Section~\ref{primo} it is false in general.

Two are the main properties of \(INV\) maps that we are going to need in the sequel: that they are closed under weak convergence, and that they are continuous outside a small set. Both these facts have been established in~\cite{MS}, but since we need them in a slightly different setting we provide here their proofs.

\begin{lemma}\label{lm:continuity}
Let \(u\in W^{1,p}(\Q)\cap INV(\Q) \), \(1\leq p\leq 2\). Then, there exists a set \(\SS_u\subseteq \Q\) such that \(\H^{2-p}(\SS_u)=0\) and \(u\res (\Q\setminus \SS_u)\) is continuous.
\end{lemma}

\begin{proof}
We note that, according to the \(INV\) condition, given \(P\in \Q\), then for almost every \(r\) small
\begin{equation}\label{e:osc}
\osc_{B(P,r)} u\le \osc_{\partial B(P,r)} u\le \int_{\partial B(P,r) } |Du|.
\end{equation}
Here by \(\osc_{B(P,r)} u\) we mean the oscillation of \(u\) intended as a multivalued function, cf. Defintion~\ref{def:multifunction}; that is,
\[
\osc_{B(P,r)} u=\diam \bigcup_{Q\in B(P,r)} u(Q).
\]
By integrating~\eqref{e:osc} with respect to \(r\in (\rho,2\rho)\) and using the monotonicity of \(r\to \osc_{B(P,r)} u\) we get
\begin{equation}\label{eq:osc2}
\osc_{B(P,\rho )} u\le \frac{1}{\rho} \int_{ B(P,2\rho) } |Du|\leq (4\pi)^{1/p'}\rho^{1-2/p} \Big(\int_{ B(P,2\rho) } |Du|^p\Big)^{1/p} .
\end{equation}
In particular,
\[
\begin{split}
\big\{P: \textrm{\(u\) is not continuous at \(P\)}\big\}& =\big\{P: \liminf_{r\to 0}\osc_{B(P,r)} u>0\big\}
\\
&\subseteq \Bigg\{P: \limsup_{r \to 0} r^{p-2}\int_{B(P,r)} |Du|^p>0\Bigg\}.
\end{split}
\]
Since the latter set is \(\H^{2-p}\)-negligible (see for instance~\cite[Section 2.4.3]{EG}), we conclude the proof.
\end{proof}

\begin{lemma}\label{lm:closure}
Let \(\{u_k\}\subseteq W^{1,p}(\Q)\cap INV(\Q) \) (resp., \(\{u_k\}\subseteq W^{1,p}(\Q)\cap INV^+(\Q) \)), be such that \(u_k\deb u\) in \(W^{1,p}\). Then, \(u\in W^{1,p}(\Q)\cap INV(\Q)\) (resp., \(u\in W^{1,p}(\Q)\cap INV^+(\Q) \)). 
\end{lemma}

\begin{proof} We will prove the Lemma only in the case \(p=1\) and for the \(INV\) condition, the other cases being analogous. Recall that in this case, by the Dunford-Pettis Theorem, there exists a super-linear function \(g\) such that
\[
\sup_{k} \int_{\Q} g(|Du_k|)<+\infty.
\]
 We need to show that that for every \(P\in \Q\) and almost every radius \(r\)
\[
u(Q)\in F_u(B(P,r)) \qquad\textrm{for almost all \(Q\in B(P,r)\)}.
\] 
To see this, note that for almost every \(r>0\) there is a (not re-labaled) subsequence (possibly depending on \(r\)) such that 
\[
\limsup_{k} \int_{\partial B(P,r)} g(|Du_k|)<+\infty\textrm{ and }\|u_k^*-u^*\|_{L^\infty(\partial B(P,r))} \to 0 
\]
In particular, \(u^*\res \partial B(P,r)\) is continuous. Let now \(Q\in B(P,r)\) be such that \(u^*(Q)\) exists as the limit of averages and such that \(u^*_k(Q)\to u^*(Q)\). If \(u^*(Q)\in u^*(\partial B(P,r))\) there is nothing to prove, hence we can assume that \(u^*(Q)\notin u^*(\partial B(P,r))\) and thus, for \(k\) large, \(u^*_k(Q)\notin u_k^*(\partial B(P,r))\). Since \(u_k\) satisfies the \(INV\) condition and \(Q\in B(P,r)\) we have that
\[
\deg(u_k^*(Q), u^*_k\res \partial B(P,r))\ne 0.
\]
By the continuity of the degree we deduce that \(\deg(u^*(Q), u^*\res \partial B(P,r))\ne 0\). By the same argument one shows that for almost all \(Q\in \Q\setminus B(P,r)\) either \(u^*(Q)\in u^*(\partial B(P,r))\) or \(\deg(u^*(Q), u^*\res \partial B(P,r))=0\).
\end{proof}

Note that since every diffeomorphism is clearly an \(INV^+\) map we have the following result.
\begin{corol}\label{cor:diff}
One has
\[
\overline{\mathcal D_{\id}(\Q)}^{\,{\rm ws}-W^{1,p}}\subseteq INV^+(\Q)\,.
\]
\end{corol}

We now show that a map which equals the identity on \(\partial \Q\) and for which the counter-image of any point is connected satisfies condition \(INV^+\).

\begin{lemma}\label{ThCINV}
Let $u:\Q\to\u\Q$ be a $W^{1,p}$ function such that $u={\rm Id}$ on $\partial\Q$ and for every $\u P\in{\rm Int}(\u\Q)$ the set $u^{-1}(\u P)$ is connected. Then $u$ satisfies the INV$^+$ condition. Moreover, whenever $\gamma:\S^1\to \Q$ is an injective, Lipschitz continuous, counterclockwise oriented curve on which $u^*$ is continuous,  one has $\deg(\u P,u^*\circ\gamma)\in \{0,1\}$ for every $\u P\in\u\Q\setminus u^*(\gamma)$.
\end{lemma}
\begin{proof}
Up to a re-parametrization (i.e., by pre-composing \(u\) with a bi-Lipschitz map \(\Phi\)) it is enough to show that the conditions in Definition~\ref{defINV} are satisfied for a fixed square \(S\subseteq\Q\) with sides parallel to $\partial\Q$ and such that \(u^*\res S\) is continuous.

\step{I}{Internal points have internal counter-images.}
Let \(\u P\in \u\Q\) be such that \(\deg(\u P, u^*\res \partial S)\ne 0\). We want to show that there exists \(P\in S\) such \(\u P\in u(P)\). To do so, let us divide divide \(S\) in four rectangles \(S^i\), \(i=1,\dots,4\), such that \(u^*\res \partial S^i\) is continuous; notice that, since \(u^*\) is continuous on most of the vertical and horizontal segments, we can assume that \(\diam(S^i)\le 3\diam(S)/4\). Then, either \(\u P\in u^*(\partial S^i)\) for some \(i\) or \(\deg(\u P, u^*\res \partial S^i)\) is well defined for all \(i=1,\dots,4\). In the first case we have that \(\u P\in u(P)\) for some \(P\in S\), so we are already done, while in the latter case we have, by the additivity of the degree, that
\[
0\ne \deg(\u P, u^*\res \partial S)=\sum_{i=1}^4 \deg(\u P, u^*\res \partial S^i).
\]
In particular \(\deg(\u P, u^*\res \partial S^i)\ne 0\) for some \(i\). Iterating this argument, either we prove our claim, or we find a sequence of encapsulated rectangles \(S_k\) with \(\diam(S_k)\to 0\) such that \(0\ne \deg(\u P, u^*\res \partial S_k)\). If \(P=\cap _k S_k\), one has then that \(\u P\in u(P)\). 
%
%
%
\step{II}{External points have external counter-images.}
Let \(\u P\in \u\Q\) be such that \(\deg(\u P, u^*\res \partial S)=0\), so in particular $\u P\notin u^*(\partial S)$: we want to show that there exists a point \(P\in \Q\setminus S\) such that \(P\in u^{-1}(\u P)\). To this end note that, up to a bi-Lipschitz change of variable, we can assume that \(S\) is a dyadic square, i.e., of the form \((i2^{-k}, (i+1)2^{-k})\times(j2^{-k},(j+1)2^{-k})\) for some $k\in\N$ and $0\leq i,\,j<2^k$, and that \(u^*\) is continuous on the boundary of every dyadic square with side length \(2^{-k}\). Let \(S_m\), for $1\leq m\leq 2^{2k}$, be the collection of these squares and let us assume that they are numbered so that \(S_1=S\). Now, either \(\u P\in \bigcup_{m} u^*(\partial S_m)\), and in this case we are done, or by the additivity of the degree:
\[
1=\deg(\u P, u^*\res \partial \Q)=\sum_{m} \deg(\u P, u^*\res \partial S_m)\,,
\]
where the first equality follows form the fact that \(u=\id\) on \(\partial \Q\). Since \(\deg(\u P, u^*\res \partial S_1)=0\) by assumption, there exists \(m\geq 2\) such that \(\deg(\u P, u^*\res \partial S_m)\neq 0\). By the previous step there exists \(P\in S_m \subseteq \Q\setminus S\) such that \(P\in u^{-1}(\u P)\), and this concludes the proof of the claim.
\step{III}{Conclusion.} The conclusion is now simple. Indeed, let \(Q\in S\) be such that \(u_{\rm multi}(Q)=u^*(Q)=\u Q\) is uni-valued, and assume by contradiction that 
\[
\deg(\u Q, u^*\res S)=0\,,
\]
so that in particular \(\u Q\notin u^*(\partial S)\). By the second step, there exists \(P\in \Q\setminus \overline {S}\) such that \(P\in u^{-1}(\u Q)\). Then, $u^{-1}(\u Q)$ contains $Q\in S$, $P\in \Q\setminus \overline S$, and no point in $\partial S$, and this is a contradiction with the fact that $u^{-1}(\u Q)$ is connected by assumption; thus, the first condition in Definition~\ref{defINV} is satisfied. In the very same way one checks the second one.
 
%
%
%
%
%
%

It only remains to show the last claim, namely, that the degree can only assume the values \(0\) and \(1\). To prove this, let \(S\) be the domain bounded by \(\gamma\); notice that, since we can extend $u$ as the identity in an external neighborhood of $\Q$, it is not restrictive to assume that the image of $\gamma$ lies in the interior of $\Q$. Hence, up to a bi-Lipschitz change of coordinates, we can assume that \(S\) is a dyadic square as in Step II.

Let \(\u P\in \Q\setminus u^*(\partial S)\) be such that \(\deg(\u P,u^*\res \partial S)\ne 0\): we need to show that \(\deg(\u P,u^*\res \partial S)=1\). Note that, by Step 1, there exists a point \(P\in S\) such that \(P\in u^{-1}(\u P)\). This and the fact that \(u^{-1}(\u P)\) is a connected set disjoint from \(\partial S\) imply that \(u^{-1}(\u P)\subseteq S\). Let now \(S_m\) be the partition of \(\Q\) in dyadic cube as in Step II, and assume again that $S=S_1$. Again by Step I, \(\deg(\u P, u^*\res S_m)\) is well defined and equal to \(0\) for every \(m\ge 2\) since otherwise \(u^{-1}(\u P)\cap \overline{S}_m\ne \emptyset \), a contradiction. By the additivity of the degree and since \(u\) equals the identity on $\partial\Q$ we get
\[
1=\deg(\u P, u^*\res \partial \Q)=\sum_{m} \deg(\u P, u^*\res \partial S_m)=\deg(\u P, u^*\res \partial S_1)\,,
\]
so the proof is concluded.
%
%
%
%
%
\end{proof}

\begin{remark}\label{contconn}
Notice that, thanks to steps~I and~II of the above proof, if $u$ is a function for which $u^{-1}(\u P)$ is connected for every $\u P\in\u\Q$ then, whenever $D$ and $E$ are two Lipschitz sets such that $u^*$ is continuous on $\partial D$ and $\partial E$, if $D\subseteq E$ then $F_u(D)\subseteq F_u(E)$ by Definition~\ref{def:multifunction}. As a consequence, under this assumption the definition of $u_{\rm multi}(P)$ can be simplified as
\[
u_{\rm multi}(P) = \bigcap\nolimits_{\,\D_P} F_u(D)\,,
\]
where $\D_P=\cup_{\eps>0} \D_{P,\eps}$. It is actually possible to show that, if \(u\in INV^+\), then for every $\u P\in\u\Q$ the set \(u^{-1}(\u P)\) is necessarily connected, thus the condition of the above Lemma is actually a characterisation of \(INV^+\), see~\cite[Proposition 5.14]{BMH}. In particular, by Corollary~\ref{cor:diff}, connectedness of any \(u^{-1}(\u P)\) is a necessary condition to be approximable by diffeomorphisms. However, as we will show in Section~\ref{secondo}, it is not sufficient.
\end{remark}

\subsection{Definition of Lebesgue squares and extension results}\label{sec:extension}

In this short section, we recall the definition and some properties of the ``Lebesgue squares'', which have been already used several times in the last years to get approximation results, see for instance~\cite{DP,HP,C,R}. Moreover, we list three extension results, which will be used in the sequel.

\begin{defin}[Lebesgue squares]
Let \(u\in W^{1,p}(\Q)\). We say that a square \(S\subseteq \Q\) is a \emph{Lebesgue square with matrix \(M\) and constant \(\kappa\)} if
\[
\frac{1}{|S|}\,\int_{S} |Du-M|^p\le \kappa \,.
\] 
\end{defin}


\begin{lemma}\label{moltiLeb}
Let \(u\in W^{1,p}(\Q)\), let $\A\subseteq\R^{2\times 2}$ be a set of matrixes, and let \(\kappa>0\) be fixed. For every $K\geq 1$ let $\{S^K_m\}$, for $1\leq m\leq K^2$, be the standard partition of $\Q$ in $K^2$ squares of side $1/K$, and let us call \({\rm Leb_K}(\kappa, \A)\) the set of squares of this partition which are Lebesgue square with some matrix \(M\in \A\) and constant \(\kappa\). If $K$ is big enough, depending on $u,\, \kappa$ and $\A$, then
\[
\sum_{S^K_m\notin {\rm Leb_K}(\kappa, \A)} |S^K_m|\leq 2 |\{x\in\Q:\, Du(x)\notin \A\}|\,.
\]
\end{lemma}
\begin{proof}
Let $x\in\Q$ be a Lebesgue point for $Du$ such that $Du(x)\in\A$. By definition of Lebesgue point, there exists $\bar r=\bar r(x,\kappa)$ such that, whenever $r<\bar r$, one has
\[
\int_{B(x,r)} |Du-Du(x)|^p \leq \frac \kappa 2\, r^2\,.
\]
Assume now that $x$ is contained in some square $\SS$ with side $\ell$ for some $\ell \sqrt 2< \bar r$; then, since $\SS\subseteq B(x,\ell\sqrt 2)$, we obtain
\[
\frac 1{|\SS|}\, \int_\SS |Du-Du(x)|^p \leq \frac 1{\ell^2} \int_{B(x,\ell \sqrt 2)} |Du-Du(x)|^p \leq \kappa \,,
\]
that is, $\SS$ is a Lebesgue square with matrix $Du(x)\in\A$ and constant $\kappa$. As a consequence, if a square $S^K_m$ of the partition is not contained in ${\rm Leb_K}(\kappa, \A)$, then the whole square is contained in the set of points $x$ which are either not Lebesgue points for $Du$, or Lebesgue points for $Du$ with $Du(x)\notin A$, or Lebesgue points for $Du$ with $Du(x)\in\A$ but $\bar r(x,\kappa)\leq \sqrt 2/K$. Since the set of points for which the first condition holds is negligible, and the set of those for which the third one holds has mass which goes to $0$ when $K\to \infty$, the lemma is proved.
\end{proof}

\begin{lemma}\label{linearsmall}
Let $\alpha\in W^{1,1}(PQ;\R^2)\) where $PQ$ is a segment, and let $\beta$ be the linear function such that \(\beta(P)=\alpha(P)\) and \(\beta(Q)=\alpha(Q)\). Then, for every matrix $M\in \R^{2\times 2}$,
\[
\int_{PQ} |\beta'-M\tau|^p\leq \int_{PQ} |\alpha'-M\tau|^p\,,
\]
where $\tau$ is the unit vector in direction of $PQ$.
\end{lemma}
\begin{proof}
This is just a trivial application of Jensen inequality.
\end{proof}

We present now three extension results. The first two have been proved in the case $p=1$ in~\cite[Theorem 2.1, Theorem 3.1]{HP}, and in the case $p>1$ in~\cite[Theorem 1.1]{R}, \cite[Theorem 2.1, Theorem 3.7]{C}. Actually, the second one was proved with a slightly different claim, hence we add a very short proof to show that our claim follows from the original one.
\begin{prop}\label{extgener}
For every $p\geq 1$ there exists a purely geometrical constant $H_1=H_1(p)\geq 1$ such that the following holds. For every square $\RR\subseteq\R^2$ with side $\ell$, and for every injective, piecewise linear function $\varphi:\partial\RR\to\R^2$, there exists a finitely piecewise affine homeomorphism $\omega:\RR\to\R^2$, coinciding with $\varphi$ on $\partial\RR$, such that
\[
\int_\RR |D\omega|^p \leq H_1 \ell \int_{\partial\RR} |\varphi'|^p\,.
\]
\end{prop}
\begin{prop}\label{extdet0}
For every $p\geq 1$ there exists a purely geometrical constant $H_2=H_2(p)$ such that the following holds. For every non-zero $2\times 2$ matrix $M$ with $\det M=0$, every square $\RR\subseteq\R^2$ with side $\ell$, and every injective, piecewise linear function $\varphi:\partial\RR\to\R^2$, there exists a finitely piecewise affine homeomorphism $\omega:\RR\to\R^2$, coinciding with $\varphi$ on $\partial\RR$, such that
\[
\int_\RR |D\omega - M|^p \leq H_2 \ell \int_{\partial\RR} |\varphi' - M\tau|^p\,,
\]
where $\tau$ is the tangential vector on $\partial\RR$.
\end{prop}
\begin{proof}
In~\cite[Theorem 3.7]{C}, this result was proved (actually, in a quite more general situation) with the additional assumption that
\[
\int_{\partial\RR} |\varphi'-M\tau|^p < \delta_{\rm max} \ell |M|^p\,,
\]
for some geometrical constant $\delta_{\rm max}$ only depending on $p$. However, if this assumption is not satisfied, we can define the extension $\omega$ from Proposition~\ref{extgener}, and we get
\[\begin{split}
\int_\RR |D\omega - M|^p &\leq 2^p \int_\RR |D\omega|^p + 2^p |M|^p \ell^2
\leq 2^p H_1\ell \int_{\partial\RR} |\varphi'|^p + 2^p |M|^p \ell^2\\
&\leq 4^p H_1\ell \int_{\partial\RR} |\varphi'-M\tau|^p +(4^{p+2} H_1+2^p) |M|^p \ell^2\\
&\leq \bigg(4^p H_1+\frac{(4^{p+2} H_1+2^p)}{\delta_{\rm max}}\bigg)\,\ell\,\int_{\partial\RR} |\varphi' - M\tau|^p\,,
\end{split}\]
so that the result is still true, up to possibly enlarge the constant $H_2$.
\end{proof}

\begin{prop}\label{extdetsigma}
For every $p\geq 1$ and $\sigma<1$ there exists a constant $H_3=H_3(\sigma,p)$ such that the following holds. For every square $\RR\subseteq\R^2$ with side $\ell$, every matrix $M\in\R^{2\times 2}$ with $|M|,\, \det M \in (\sigma,1/\sigma)$, and every injective, piecewise linear function $\varphi:\partial\RR\to\R^2$, there exists a finitely piecewise affine homeomorphisms $\omega:\RR\to\R^2$, coinciding with $\varphi$ on $\partial\RR$, such that
\begin{equation}\label{itsla}
\int_\RR |D\omega - M|^p \leq H_3 \ell^{2-\frac 1p} \bigg(\int_{\partial\RR} |\varphi' - M\tau|^p\bigg)^{\frac 1p}\,,
\end{equation}
where $\tau$ is the tangent vector $\partial\RR$.
\end{prop}
\begin{proof}
By a change of variables, we can assume that $M$ is the identity matrix; it is actually not difficult to see how $H_3$ depends on $\sigma$, but we are not interested in the precise dependance. Moreover, it is enough to consider the case when $\RR=\Q=[0,1]^2$ is the unit square, and we will do so for simplicity. By the same argument as in Proposition~\ref{extdet0}, we can assume that
\[
\delta = \bigg(\int_{\partial \Q} |\varphi'- \tau|^p\bigg)^{\frac 1p}\,,
\]
is as small as we wish, up to increase the constant $H_3$.\par

Let us first show that $\|\varphi-Id\|_{L^\infty(\partial\Q)}\leq 2\delta$, assuming up to a translation that $\varphi(O)=(O)$, where $O=(0,0)$. In fact, for any $0\leq t\leq 1$, calling $P=(t,0)$ we have
\[
|\varphi(P)-P| \leq \int_{s=0}^t |\varphi'(t)-{\rm e}_1| \leq \|\varphi'-\tau\|_{L^1([0,1]\times \{0\})}\leq \delta\,,
\]
and arguing similarly on the other sides of $\partial\Q$ we obtain the $L^\infty$ estimate.\par

Let now $C=C(p,\sigma)$ be a large, purely geometric constant, and let $N\approx 1/(C\delta)$ be a large integer. First of all, we set $\omega$ as the identity on the square $[1/N,1-1/N]^2$, and of course $\omega=\varphi$ on $\partial\Q$. Let now $P_i=\u P_i=(i/N,1/N)$ for some $1\leq i\leq N-1$, and let us call $\u P'_i$ a  point on $\varphi([0,1]\times \{0\})$ minimizing the distance with $\u P_i$. Notice that, since $\|\varphi-{\rm Id}\|_{L^\infty}<2\delta$ and $C$ is a large constant, then the distance between $\u P'_i$ and $(i/N,0)$ is much smaller than $1/N$, and  in turn also the distance between $P'_i=\varphi^{-1}(\u P'_i)$ and $(i/N,0)$. We define in the analogous way the points $Q_i=\u Q_i=(i/N,1-1/N)$ with $\u Q_i'\in \varphi([0,1]\times \{1\})$ and $Q_i'=\varphi^{-1}(\u Q_i)$ near $(i/N,1)$, the points $R_i=\u R_i=(1/N,i/N)$ with the points $\u R_i'$ and $R_i'$ near $(0,i/N)$ and the points $S_i=\u S_i=(1-1/N,i/N)$ with $\u S_i'$ and $S_i'$ near $(1,1/N)$. Note that for \(i\ne j\) we always have that \(|\u P_j-\u P_j|=1/N(1+O(\delta))\).

Let us consider each segment $P_iP_i'$: since $\omega$ is already defined on both points, we extend $\omega$ on the whole segment $P_iP_i'$ as a linear function; we do the same on all the segments $Q_iQ_i'$, $R_iR_i'$ and $S_iS_i'$. Notice that by construction $\omega$ is injective, where it is already defined.\par

Observe now that the region $\Q\setminus [1/N,1-1/N]^2$ is subdivided in $4N-4$ quadrilaterals, on the boundary of each of which $\omega$ is already defined; moreover, each quadrilateral is the $2$-bi-Lipschitz copy of a square of side $1/N$: as a consequence, by Proposition~\ref{extgener} we can extend $\omega$ on the interior of any such quadrilateral, call it $\RR$, so that
\[
\int_\RR |D\omega|^p \leq \frac{4H_1}N \int_{\partial\RR} |D_\tau \omega|^p\,.
\]
Moreover, of course
\[
\int_{\partial\RR\cap\partial\Q} |D_\tau \omega|^p = \int_{\partial\RR\cap\partial\Q} |\varphi'|^p\,.
\]
On the other hand, on each of the (either two or three) sides of $\partial\RR$ which are not contained in $\partial\Q$, the function $\omega$ is a linear function which maps a segment of length almost precisely $1/N$ onto a segment of length almost precisely $1/N$; therefore, $\int_{\partial\RR\setminus\partial\Q} |D_\tau \omega|^p \leq \frac 4N$. Summarizing,
\[
\int_\RR |D\omega|^p \leq \frac{4H_1}N \bigg( \int_{\partial\RR\cap\partial\Q} |\varphi'|^p + \frac 4N\bigg)\,,
\]
which adding on the $4N-4$ quadrilaterals gives
\[
\int_{\Q\setminus [1/N,1-1/N]^2} |D\omega|^p \leq \frac {4H_1}N \int_{\partial\Q} |\varphi'|^p + \frac{64H_1}N
\leq \frac {2^{p+2}H_1}N \big(\delta^p + 4\big) + \frac{64H_1}N
\leq \frac{2^{p+6}H_1}N\,.
\]
As a consequence, we can calculate
\[\begin{split}
\int_\Q |D\omega - M |^p &= \int_{Q\setminus [1/N,1-1/N]^2} |D\omega - M |^p
\leq 2^p\int_{Q\setminus [1/N,1-1/N]^2} |D\omega|^p + |M|^p\\
&\leq 2^p \bigg(\frac{2^{p+6}H_1}N+\frac 4{N\sigma^p}\bigg)
\leq \frac{2^{2p+8} H_1 C}{\sigma^p} \, \delta=: H_3(p,\sigma) \,\delta\,,
\end{split}\]
which is~(\ref{itsla}) for the case $M={\rm Id}$ and $\ell=1$ that we are considering.
\end{proof}

It is easy to guess that a stronger version of last extension result should actually be true, namely, without the $p$-th root and with $\ell$ in place of $\ell^{2-\frac 1p}$ in~(\ref{itsla}). However, this weaker result is enough for our needs.

\section{Proof of Theorem~\mref{Caratt} and~\mref{WeakStrong}\label{sec:A}}
In this Section we will prove Theorems  Theorem~\mref{Caratt} and~\mref{WeakStrong}. Before going in the details of the proof we will make a couple of simplifying assumption that will be in force for all the Section:
 
\begin{itemize}
\item[-] We will always work with the representatives defined in Definition~\ref{defnocross}.
\item[-] We will always assume that the functions coincide with the identity in a neighbourhood of the boundary (we can do this without loss of generality, up to extend all the maps as the identity outside \(\Q\)).

\end{itemize}

\subsection{The two grids}

This short section is devoted to define suitable grids $\G$ and $\widetilde\G$ on the squares $\Q$ and $\u\Q$.
\begin{defin}[Admissible curve]\label{admcur}
Let $X$ be either $[0,1]$ or $\S^1$, and let $\gamma: X\to\Q$ be a piecewise ${\rm C}^1$, injective curve. We say that \emph{$\gamma$ is an admissible curve} if $\gamma(t)$ is a Lebesgue point for $Du$ for $\H^1$-a.e. $t\in X$, and $u\circ \gamma$ belongs to $W^{1,p}(X)$.
\end{defin}
\begin{defin}
Let $K\geq 1$ be an integer. Then, we call $\RR_{i,j}=[i/K,(i+1)/K]\times [j/K,(j+1)/K]\subseteq \Q$, for all $0\leq i,\,j<K$, the elements of the standard partition of $\Q$ in $K^2$ squares. Moreover, we call $\G=\G(K):= \bigcup_{0\leq i,\, j<K} \RR_{i,j}\subseteq\Q$.
\end{defin}

\begin{lemma}\label{startinggrid}
Let $u\in W^{1,p}(\Q;\u\Q)$, with $u={\rm Id}$ in a small neighborhood of $\partial\Q$, and let $\delta\ll 1$. Then, for every integer $K>1/\delta$ and any choice of matrices $M_{i,j}\in\R^{2\times 2}$, with $0\leq i,\,j <K$, there exists a finitely piecewise affine, $(1+4\delta)$-biLipschitz homeomorphism $\Phi:\Q\to\Q$ coinciding with the identity on $\partial\Q$, such that $\|\Phi-{\rm Id}\|_{L^\infty(\Q)}<2\delta/K$ and that, for every square $\RR_{i,j}\in\G(K)$, the curve $\partial(\Phi(\RR_{i,j}))$ is an admissible curve and one has
\begin{equation}\label{tocon}
\int_{\Phi(\partial\RR_{i,j})} |D u -M_{i,j}|^p \, d\H^1 \leq \frac{204K} \delta \int_{\RR_{i,j}^+} |Du-M_{i,j}|^p \, d\H^2\,.
\end{equation}
Here $\RR_{i,j}^+$ the union of the squares $\RR_{i',j'}$ with $\max\{|i-i'|,\, |j-j'|\}\leq 1$. In addition, for every $0\leq i,\,j<K$, $\Phi$ is affine on each of the two triangles in which the square $\RR_{i,j}$ is divided by the northeast-southwest diagonal. Finally, the points $\Phi(i/K,j/K)$ are all continuity points for $u$.
\end{lemma}
\begin{proof}
The proof is a simple variation of the argument of Lemmas~4.9 and~4.13 in~\cite{HP}, see also step~II in the proof of~\cite[Theorem~4.1]{C}. Let us fix $\delta>0$, so small that $u={\rm Id}$ in a $2\delta$-neighborhood of $\partial\Q$, and an integer $K>1/\delta$. For every choice of $1\leq l,\, m\leq K-1$, we consider the following short segment near $A=(l/K,m/K)$, with a slope of $45^\circ$,
\[
I_A= \bigg\{\bigg(\frac{l+t}K , \frac{m+t+\eta}K \bigg),\, -\delta < t <\delta \bigg\}\,.
\]
The constant $\eta\ll \delta$ is chosen in such a way that $\H^1$-a.e. point of each $I_A$ is a continuity point of $u$.\par

Let us now consider two adjacent points $A=(l/K,m/K)$ and $B=(l'/K,m'/K)$, where by ``adjacent'' we mean that $|l-l'|+|m-m'|=1$. Let us call $\RR^+(A,B)$ the union of the six squares $\RR_{i,j}$ having either $A$ or $B$ or both as vertices, and let $M_1$ and $M_2$ be the matrices corresponding to the two squares having $AB$ as a side. Defining then the set
\[
\Gamma(A,B) = \bigg\{ (x,y)\in I_A\times I_B:\, \exists \,n\in\{1,\,2\}, \int_{xy} |Du - M_n|^p > \frac{51K}\delta \int_{\RR^+(A,B)} |Du-M_n|^p \bigg\}\,,
\]
an immediate change of variable argument ensures that
\[
\H^2\big(\Gamma(A,B)\big)<\frac 1 {25} \H^1(I_A)\times \H^1(I_B)\,.
\]
In particular
\[
\H^1 \bigg( \bigg\{ x\in I_A:\, \H^1\bigg(\bigg\{y\in I_B:\, (x,y)\in \Gamma(A,B)\bigg\}\bigg)> \frac{\H^1(I_B)}5\bigg\}\bigg) < \frac{\H^1(I_A)}5\,.
\]
An easy recursion argument, identical to the one done in Lemma~4.13 of~\cite{HP}, ensures that it is possible to pick a point $V_A\in I_A$ for each $A=(l/K,m/K)$ in such a way that, whenever $AB$ is a side of some square $\RR_{i,j}$, the pair $(V_A,V_B)$ does not belong to $\Gamma(A,B)$, almost every point in the segment $V_AV_B$ is a Lebesgue point for $Du$, and $V_A$ and $V_B$ are continuity points of $u$. We let then $\Phi:\Q\to\Q$ be the finitely piecewise affine bijection which sends each vertex $A$ in the corresponding point $V_A$, and which is affine on each triangle obtained subdividing any of the squares with the northeast-southwest diagonal. The facts that $\|\Phi-{\rm Id}\|_{L^\infty} < 2\delta/K$, that $\Phi$ is $(1+4\delta)$-biLipschitz, and that the points $\Phi(i/K,j/K)$ are all continuity points for $u$, are true by construction. Moreover, since  $Du$ is an $L^p$ function in the segment $V_AV_B$ for each pair such that $AB$ is the side of some square, so in particular each segment $V_AV_B=\Phi(AB)$ is an admissible curve, and then so is each $\partial(\Phi(\RR_{i,j}))$. 

To conclude, we have then only to establish~(\ref{tocon}).To do so, let us fix some square $\RR_{i,j}$, and let $AB$ be one of its sides. Since $M_{i,j}$ is one of the two matrices $M_1$ and $M_2$ corresponding to the side $AB$, and $\RR^+(A,B)\subseteq \RR_{i,j}^+$, the fact that $(V_A,V_B)\notin\Gamma(A,B)$ implies
\[
\int_{V_AV_B} |Du - M_{i,j}|^p \leq \frac{51K}\delta \int_{\RR^+(A,B)} |Du-M_{i,j}|^p \leq \frac{51K}\delta \int_{\RR^+_{i,j}} |Du-M_{i,j}|^p\,.
\]
Since $\Phi(\partial\RR_{i,j})$ is the union of four sides of the form $V_AV_B$, the last estimate yields~(\ref{tocon}).
\end{proof}

Relying on this lemma, we give the following definition.
\begin{defin}[Good starting grid]\label{goodstarting}
Let $K$ be any large integer. The set $\G=\G(K)$ is called a \emph{good starting grid} if every vertex of each square $\RR_{i,j}$ is a continuity point for $u$, and the boundary of any square $\RR_{i,j}$ is an admissible curve in the sense of Definition~\ref{admcur}.
\end{defin}

Concerning the arrival grid, what we need is the following.
\begin{defin}[Good arrival grid]\label{goodarrival}
Let $\G\subseteq\Q$ be a good starting grid. Let $\eta \ll 1/K$ and let $0=x_0< x_1 < x_2 < \cdots < x_N = 1$ and $0=y_0 <y_1 < y_2 < \cdots < y_M = 1$ be given coordinates, such that $\eta < x_{n+1}-x_n < 2\eta$ and $\eta < y_{m+1}-y_m < 2\eta$ for every $0\leq n<N$ and every $0\leq m <M$. The set $\widetilde\G\subseteq\u\Q$ given by
\begin{equation}\label{defwidegxnyn}
\widetilde\G = \bigcup_{0 \leq n \leq N} x_n \times [0,1] \cup \bigcup_{0 \leq m \leq M} [0,1]\times y_m
\end{equation}
is called a \emph{good arrival grid associated with $\G$ and with side-length $\eta$} if
\begin{enumerate}
\item $u^{-1}(\widetilde\G)\cap \G$ is a finite subset of $\G$;
\item for any $P\in u^{-1}(\widetilde\G)\cap \G\setminus \partial\Q$, $u(P)$ is not a vertex of some rectangle of $\widetilde\G$, $P$ is not a vertex of some square of $\G$ and is a Lebesgue point for $Du$ and, calling $\tau$ the direction of the open side of $\G$ containing $P$, the tangential derivative $D_\tau u(P)$ has a non-zero component in the direction orthogonal to the side of $\widetilde\G$ containing $u(P)$;
\item for every vertex $V\notin\partial\Q$ of some square $\RR_{i,j}\subseteq\G$, calling $u(V)=\u V=(\u V_1, \u V_2)$ and letting $n$ and $m$ be such that $x_n < \u V_1 < x_{n+1}$ and $y_m < \u V_2 < y_{m+1}$, one has
\[
\min \Big\{ |x_n - \u V_1|,\, |x_{n+1} - \u V_1|,\,|y_m - \u V_2|,\, |y_{m+1} - \u V_2| \Big\} > \frac \eta 2\,.
\]
\end{enumerate}
\end{defin}

The existence of good arrival grids is ensured by the following result.
\begin{lemma}\label{existencearrival}
Let $\G$ be a good starting grid, and let $\bar\eta\ll 1/K$ be fixed. Then, there exists a good arrival grid $\widetilde\G$ associated with $\G$ with side-length $\eta<\bar\eta$.
\end{lemma}
\begin{proof}
Let $A\subseteq \Q$ be the set of the vertices of $\G$, and let $\u A^1=\pi_1\big(u(A)\big)$ and $\u A^2=\pi_2\big(u(A)\big)$,  where \(\pi_1\) and \(\pi_2\) are, respectively, the projections on the first and second coordinate. Recall that $u$ is uniquely defined at any point of $A$, hence $\u A^1$ (resp., $\u A^2$) is a finite set of abscissae (resp., ordinates). Let now $\Gamma$ be a horizontal or vertical segment, contained in $\G$ and with endpoints in $\partial\Q$. Since by construction $\Gamma$ is an admissible curve, then the set $B_\Gamma$ of the points of $\Gamma$ which are not Lebesgue points for $D u$ is $\H^1$-negligible, thus so are the sets $\u B^1_\Gamma= \pi_1\big(u(B_\Gamma)\big)$ and $\u B^2_\Gamma= \pi_2\big(u(B_\Gamma)\big)$. Let us now call $C^1_\Gamma\subseteq \Gamma$ (resp., $C^2_\Gamma$) the set of points which are Lebesgue points for $D_\tau u$, but for which the first (resp., second) component of $D_\tau u$ is zero. Let $\delta>0$ be any small positive number. By definition, for any $P\in C^1_\Gamma$ there exists $\eps=\eps(P,\delta)$ such that, calling $I(P,\eps)=\{ Q\in \Gamma:\, |Q-P|<\eps\}$, it is
\[
 \H^1\Big(\pi_1\big(u(I(P,\eps))\big)\Big) < \delta\eps\,.
\]
With a standard covering argument, we can cover $C^1_\Gamma$ with countably many intervals $I(P_i,\eps_i)$ satisfying the above estimate, in such a way that the intervals $I(P_i,\eps_i/4)$ are disjoint. As a consequence, calling $\u C^1_\Gamma= \pi_1\big(u(C^1_\Gamma)\big)\subseteq [0,1]$, we have
\[
\H^1\big(\u C^1_\Gamma\big)\leq 
\sum_{i\in\N} \H^1\Big(\pi_1\big(u(I(P_i,\eps_i))\big)\Big)
\leq \sum_{i\in\N} \delta \eps_i
\leq 4\delta \H^1(\Gamma)\,.
\]
Since $\delta$ was arbitrary we get $\H^1(\u C^1_\Gamma)=0$. The very same argument  ensures also $\H^1(\u C^2_\Gamma)=0$. Calling now $\u B^1$ and $\u C^1$ (resp., $\u B^2$ and $\u C^2$) the union of all the sets $\u B^1_\Gamma$ and $\u C^1_\Gamma$ (resp., $\u B^2_\Gamma$ and $\u C^2_\Gamma$) for all the horizontal and vertical segments $\Gamma$ contained in $\G$, we have $\H^1(\u A^1\cup \u B^1\cup \u C^1)=\H^1(\u A^2\cup \u B^2\cup \u C^2)=0$.\par

Let us now select $s\notin \u A^1\cup \u B^1\cup \u C^1$, let $\Gamma$ be a segment contained in $\G$, and let $P$ be a point in $u^{-1}\big(\{s\}\times[0,1]\big)\cap \Gamma$. Since $s\notin \u A^1$ we have that $P$ is not a vertex of $\G$, and since $s\notin \u B^1\cup \u C^1$ we have that $P$ is a Lebesgue point for $Du$, and $Du(P)$ has a non-zero horizontal component. In particular there are no other points of $u^{-1}\big(\{s\}\times[0,1]\big)$ in a small neighborhood of $P$ in $\Gamma$. As a consequence, the set $u^{-1}\big(\{s\}\times[0,1]\big)\cap \Gamma$ is finite. Repeating the same argument for every segment in $\G$, and arguing in the very same way for $t\notin \u A^2\cup\u B^2\cup \u C^2$, we obtain that for every $s\notin \u A^1\cup \u B^1\cup \u C^1$ (resp., $t\notin \u A^2\cup \u B^2\cup \u C^2$), the set $u^{-1} \big(\{s\}\times [0,1]\big)\cap \G$ (resp., $u^{-1} \big([0,1]\times \{t\}\big)\cap \G$) is a finite set, which does not contain any vertex of $\G$, and containing only Lebesgue points for $Du$, at which $Du$ has a non-zero horizontal (resp., vertical) component.\par

Let now $\eta<\bar\eta$ be any number much smaller than the difference between any two different elements of $\u A^1$ or $\u A^2$. It is then possible to select numbers $0=x_0< x_1 < x_2 < \cdots < x_N = 1$, none of which (except $x_0$ and $x_N$) belonging to $\u A^1\cup\u B^1\cup \u C^1$, in such a way that $\eta<x_{n+1}-x_n<2\eta$ for every $0\leq n<N$, and also that $\min \big\{ |x_n - \u V_1|,\, |x_{n+1} - \u V_1|\big\} > \eta/2$ for every vertex $V\in\G\setminus\partial\Q$, being $\u V=u(V)$, and being $n$ such that $x_n < \u V_1 < x_{n+1}$.\par

By construction, the set $\F=u^{-1}\big(\{x_n,\, 1\leq n< N\}\times [0,1]\big)\cap\G$ is a finite set of points, and  so it  is the set $\u D^2=\pi_2\big(u(\F)\big)$. As a consequence, it is possible to choose numbers $0=y_0 <y_1 < y_2 < \cdots < y_M =1$ not belonging to $\u A^2\cup \u B^2\cup \u C^2\cup \u D^2$ (except $y_0$ and $y_M$), and so that $\eta<y_{m+1}-y_m<2\eta$ for every $0\leq m<M$, and that $\min \big\{ |y_m - \u V_2|,\, |y_{m+1} - \u V_2|\big\} > \eta/2$ for every vertex $V\in\G\setminus\partial\Q$, again calling $\u V=u(V)$ and letting $m$ be such that $y_m < \u V_2 < y_{m+1}$. It is then clear that the set $\widetilde\G$ given by~(\ref{defwidegxnyn}) is a good arrival grid with side-length $\eta'<\eta$.
\end{proof}

\subsection{The piecewise linear function on the grids}

This section is devoted to define the piecewise linear modification of a no-crossing function $u$ on a good starting grid. We will show the following key  result. It is here that the no-crossing condition is used,  see Step II in the proof.

We start with the following useful definition:

\begin{defin}[Generalized segment]\label{gensegm}
Let $\u\RR$ be a closed rectangle of the grid $\widetilde\G$, and let $\u A,\, \u B$ be any two points in $\partial\u\RR$. Given \(\xi\) small, the \emph{generalized segment $[\u{AB}]$ between $\u A$ and $\u B$ in $\u\RR$} is defined as the segment $\u{AB}$ if the two points are not on a same side of $\partial\u\RR$; otherwise, $[\u{AB}]$ is the union of the segments $\u{AM}$ and $\u{MB}$, where $\u M$ is the point in $\u\RR$ whose distance from the side containing $\u A$ and $\u B$ is $\xi$ times the length of $\u{AB}$, and whose projection there is the mid-point of the segment $\u{AB}$.
\end{defin}

Notice that the generalized segment $[\u{AB}]$ is entirely contained in the interior of $\u\RR$ except for $\u A$ and $\u B$. Moreover, taken four distinct ordered points $\u A,\, \u B,\, \u C$ and $\u D$ on a same side of $\partial\u\RR$, the generalized segments $[\u{AB}]$ and $[\u{CD}]$ have empty intersection, as well as $[\u{AD}]$ and $[\u{BC}]$, while $[\u{AC}]$ and $[\u{BD}]$ have exactly an intersection point. 

\begin{prop}\label{injectivelinearization}
Let $v\in\mathcal{NC}_{\id} (\Q)\cap W^{1,p}$, assume that $\G=\G(K)$ is a good starting grid for $v$, let $M_{i,j}$ for $0\leq i,\,j<K$ be given $2\times 2$ matrices, and let $\sigma\ll 1$ be fixed. Then, there exists an injective, piecewise linear function $\varphi:\G\to\u\Q$, coinciding with the identity on $\partial\Q$, such that, for every $0\leq i,\,j<K$, one has
\begin{equation}\label{linearization}
\int_{\partial \RR_{i,j}} |\varphi'-M_{i,j}\tau|^p \leq \frac \sigma K + 2 \int_{\partial \RR_{i,j}} |D v-M_{i,j}|^p\,,
\end{equation}
where $\tau$ is the unit tangent  vector to $\partial\RR_{i,j}$.
\end{prop}
\begin{proof}
We divide the proof in three steps.
\step{I}{Preliminaries.}
First of all, we fix a length $\xi\ll 1/K$ with the property that, for every segment $AB\subseteq \G$ with length $\H^1(AB)\leq \xi$, one has
\begin{equation}\label{shortxi}
\int_{AB} |Dv|^p < \frac{(\sqrt[p]2-1)^p}{16\cdot 7^p K}\,\sigma\,,
\end{equation}
which is clearly possible since $v\in W^{1,p}(\G)$. Then, by Lemma~\ref{existencearrival} we can take a good arrival grid $\widetilde\G$ associated with $\G$ and with side-length equal to some $\eta\ll \xi\ll 1/K$ satisfying
\begin{equation}\label{longxi}
\frac{(6\eta)^p}{\xi^{p-1}} < \frac \sigma{16K}\,(\sqrt[p]2-1)^p\,.
\end{equation}
Let us now call $P_1,\,P_2,\, \dots\, ,\, P_N$ the points of $v^{-1}(\widetilde\G)\cap \G$, and $\u P_j=v(P_j)$ for every $1\leq j\leq N$. Notice that the points $\u P_j$ are all on $\widetilde\G$, but they are not necessarily all distinct. Let us also notice that, by definition of good arrival grid, each $D_\tau v(P_i)$ makes a non-zero angle $\theta_i$ with respect to the horizontal (resp., vertical) direction if $\u P_i$ belongs to a horizontal (resp., vertical) side of $\widetilde\G$; let then $C$ be any constant bigger than $\max_{1\leq i\leq N} 1/\sin\theta_i$.\par

Let now $\eps$ be a constant such that $C\eps$ is much smaller than $\eta$, than the distance between any two $\u P_i$ and $\u P_j$ which do not coincide, and than the distance between any $\u P_i$ and the closest vertex of $\widetilde\G$. By definition of good starting and arrival grid, up to decrease $\eps$ if necessary we have that, for any $P\in \G$, the distance between $v(P)$ and $\widetilde\G$ can be smaller than $\eps$ only if $P$ is close to some of the $P_i$, and in this case $|v(P)-\u P_i|<C\eps$. 

Since any $P_i$ is in the interior of a side of $\G$, again up to decrease $\eps$ if necessary we can find points $P_i^\pm$ on $\G$ for any $1\leq i\leq N$, in such a way that the following holds. First of all, $P_i^+$ and $P_i^-$ are on the same side of $\G$ to which $P_i$ belongs, $P_i$ is in the segment between them, and these segments, varying $1\leq i\leq N$, are all disjoint. Moreover, the curve $v(P_i^-P_i^+)$ lies  within a distance $C\eps$ from $\u P_i$, in particular the $\eps$-neighborhood of this curve is contained in the union of the two squares of $\widetilde\G$ having $\u P_i$ on the boundary, and the points $v(P_i^-)$ and $v(P_i^+)$ do not belong to the same square. Finally, for any $x\in\G$ one has
\begin{gather*}
P\notin \cup_{i=1}^N P_i^-P_i^+\ \Longrightarrow \ {\rm dist}(v(P),\widetilde\G)>2\eps\,, \\
{\rm dist}(v(P),\widetilde\G)<\eps \ \Longrightarrow \ \exists\,!\ 1\leq i\leq N \textrm{ s.t. } P\in P_i^-P_i^+,\, {\rm and}\, |v(x)-\u P_i|<C{\rm dist}(v(x),\widetilde\G)\,.
\end{gather*}

\step{II}{Definition of the function $\varphi:\G\to\u\Q$.}
For every $1\leq i\leq K-1$, let us now call $\gamma_i$ the horizontal segment connecting $(0,i/K)$ with $(1,i/K)$, and $\gamma_{K-1+i}$ the vertical segment connecting $(i/K,0)$ with $(i/K,1)$. By construction, these curves satisfy the assumption of Definition~\ref{defnocross}. As a consequence, there exists an injective function $\psi_1:\G\to\u\Q$, coinciding with the identity on $\G\cap\partial\Q$, such that $\|v-\psi_1\|_{L^\infty(\G)}<\eps/C$.

By construction, $\psi_1(P)$ has distance larger than $\eps$ from $\widetilde\G$ for any $P\notin\cup_{i=1}^N P_i^-P_i^+$. Moreover, for any $1\leq i\leq N$, the whole curve $\psi_1(P_i^-P_i^+)$ is contained in the union of the two squares containing $\u P_i$ in the boundary, and $\psi_1(P_i^+)$ and $\psi_1(P_i^-)$ belong to the two different squares. Let us then call $Q_i^-$ and $Q_i^+$ the first and the last point of the interval $P_i^-P_i^+$ such that $|\psi_1(Q_i^\pm)-\u P_i|=\eps$, which exist since $|\psi_1(P_i)-\u P_i|<\eps/C$ while $|\psi_1(P_i^\pm)-\u P_i|>\eps$. Let us then call $\psi_2:\G\to\u\Q$ the function which coincides with $\psi_1$ outside of the intervals $Q_i^-Q_i^+$, and which is linear on each segment $Q_i^-Q_i^+$: a standard geometrical argument ensures that also $\psi_2$ is injective, and moreover by construction $\|\psi_2-v\|_{L^\infty(\G)}\leq (C+3)\eps$, and the segment connecting $\psi_2(Q_i^-)$ and $\psi_2(Q_i^+)$ intersects $\widetilde\G$ in a single point $\u Q_i$, having distance at most $\eps$ from $\u P_i$.\par

Let us now consider any $1\leq i<j\leq N$: we will say that $P_iP_j$ is a ``principal segment'' if it lies on $\G$ and it does not contain other points $P_l$. Notice that $\G$ is a finite union of principal segments, and the intresection between any two such segments can be only empty, or a common endpoint, or a common internal point, in such a case the intersection must be a vertex $V_{ab}=(a/K,b/K)$ for some $1\leq a,\,b\leq K-1$.


We are now ready to define the function $\varphi:\G\to\u\Q$.  To this end we start noticing that if in the definition of generlaized segment, Defintion~\ref{gensegm}, \(\xi\) is taken sufficiently small, then the following holds: whenever $P_iP_j$ and $P_i'P_j'$ are two principal segments, and the curves $\psi_2(P_iP_j)$ and $\psi_2(P_i'P_j')$ are on a same rectangle $\u\RR$ of $\widetilde\G$, then the two generalised segments $[\u Q_i\u Q_j]$ and $[\u Q_i'\u Q_j']$ have a non-empty intersection if and only if the points $\u Q_i$ and $\u Q_j$ are on the two different parts in which $\partial\u\RR$ is divided by $\u Q_i'$ and $\u Q_j'$.

For every principal segment $P_iP_j\subseteq\G$ we decide that $\varphi(P_iP_j)$ has to be $[\u Q_i\u Q_j]$, with an injective parametrization still to be precised. Notice that, by construction, the function $\varphi$ will be injective regardless of the parametrization on each principal segment $P_iP_j$. Notice also that, for any vertex $V_{ab}=(a/K,b/K)$ of the grid $\G$, the point $\varphi(V_{ab})$ is already defined. Indeed, the point $V_{ab}$ is contained in the interior of exactly two principal segments $P_iP_j$ and $P_i'P_j'$, one horizontal and one vertical, and by construction the generalised segments $[\u Q_i\u Q_j]$ and $[\u Q_i'\u Q_j']$ have exactly one intersection point: this point must obviously be $\varphi(V_{ab})$. The definition of $\varphi$ is then easy: we let $\varphi$ be linear on each segment of $\G$ whose both endpoints, and no interior point, are contained in $\{P_i\}_{1\leq i \leq N} \cup \{V_{ab}\}_{1\leq a,\,b\leq K-1}$.

\step{III}{The validity of~(\ref{linearization}).}
To conclude, we only have to check the validity of the estimate~(\ref{linearization}). Let $\RR=\RR_{i,j}$ be a given square of the grid $\G$, and let $M=M_{i,j}$ be the corresponding given matrix. Notice that $\partial\RR$ is the union of finitely many segments on which $\varphi$ is linear, which are all principal segments except those having at least one vertex as endpoint. Let us consider separately the possible cases. First of all, let $P_iP_j\subseteq \partial\RR$ be a principal segment. We claim that
\begin{equation}\label{priseg}
\int_{P_iP_j} |\varphi'-M \tau|^p \leq 2 \int_{P_iP_j} |Dv-M|^p\,.
\end{equation}
Indeed, since $\varphi$ is linear on $P_iP_j$, this inequality would be given by Lemma~\ref{linearsmall}, with constant $1$ instead of $2$, if $\u Q_i=v(P_i)$ and $\u Q_j=v(P_j)$. But then, since $|\u Q_i-v(P_i)|<\eps$ and $|\u Q_j-v(P_j)|<\eps$, the inequality with constant $2$ follows as soon as $\eps$ is small enough: in fact, keep in mind that the principal segments are finitely many and the constant $\eps$ depends on them, and moreover by construction $\int_{P_iP_j} |Dv-M|^p>0$ even if $v(P_i)=v(P_j)$.\par

Let now $AB\subseteq\partial\RR$ be another of the segments on which $\varphi$ is linear, with the property that at least one between $A$ and $B$ is a vertex of $\RR$. Moreover, let us call  $\beta:AB\to\u\Q$ the linear function such that $\beta(A)=v(A)$ and $\beta(B)=v(B)$. Since for every two positive numbers $a,\,b$ one always has
\[
(a+b)^p \leq 2a^p + \frac 2{(\sqrt[p]2-1)^p}\, b^p\,,
\]
keeping in mind again Lemma~\ref{linearsmall} we have
\begin{equation}\label{bothcases}\begin{split}
\int_{AB} |\varphi'-M\tau|^p &\leq 2\int_{AB} |\beta'-M\tau|^p + \frac 2{(\sqrt[p]2-1)^p}\, \int_{AB} |\varphi'-\beta'|^p\\
&\leq 2\int_{AB} |Dv-M|^p + \frac 2{(\sqrt[p]2-1)^p}\, \int_{AB} |\varphi'-\beta'|^p\,.
\end{split}\end{equation}
We have now to distinguish two possible sub-cases, namely, whether $\H^1(AB)$ is larger or smaller than $\xi$.\par

Suppose first that $\H^1(AB)\leq \xi$. Since $\xi<1/K$, this ensures that exactly one between $A$ and $B$ is a vertex of $\G$, and the other one is a point of the form $P_i$ for some $1\leq i \leq N$. By definition of good arrival grid, this means that $|\beta(A)-\beta(B)|>\eta/2$; on the other hand, of course $|\varphi(A)-\varphi(B)|<2\sqrt 2 \eta$, and since both $\varphi$ and $\beta$ are linear in $AB$ this gives $|\varphi'|<6|\beta'|$, so again by Lemma~\ref{linearsmall}
\[
\int_{AB} |\varphi'-\beta'|^p < 7^p \int_{AB} |\beta'|^p \leq 7^p \int_{AB} |Dv|^p\,,
\]
which inserted in~(\ref{bothcases}) and keeping in mind~(\ref{shortxi}) gives
\begin{equation}\label{alsoboth}
\int_{AB} |\varphi'-M\tau|^p \leq 2\int_{AB} |Dv-M|^p + \frac \sigma{8K}\,.
\end{equation}
Let us now instead assume that $\H^1(AB)>\xi$. In this case, it might be possible  that both $A$ and $B$ are vertices of $\G$ (so $AB$ is a whole side of $\RR$), and that $v$ is constant on the whole segment $AB$, hence we cannot estimate $|\varphi'|<6|\beta'|$ as before. Nevertheless, $\varphi-\beta$ is a linear function on $AB$, and since the four points $\varphi(A),\,\varphi(B),\,\beta(A)$ and $\beta(B)$ lie in a same rectangle of the grid $\widetilde\G$, we have $|(\varphi(A)-\beta(A))-(\varphi(B)-\beta(B))|<6\eta$, hence
\[
\int_{AB} |\varphi'-\beta'|^p< \frac{(6\eta)^p}{\H^1(AB)^{p-1}} < \frac{(6\eta)^p}{\xi^{p-1}}\,,
\]
which inserted in~(\ref{bothcases}) and keeping in mind~(\ref{longxi}) gives again~(\ref{alsoboth}). Summarizing, $\partial\RR$ is a finite union of segments, and estimate~(\ref{priseg}) holds for all segments except those having at least a vertex of $\RR$ as endpoint, for which in turn whe weaker estimate~(\ref{alsoboth}) holds. Since the latter segments are at most $8$, adding over all the segments immediately yields~(\ref{linearization}).
\end{proof}

\subsection{Proof of Theorem~\mref{Caratt} (sufficiency)}\label{ss: suff}

This section is devoted to find a sequence of diffeomorphisms which strongly converge to a given Sobolev function satisfying the no-crossing condition. In fact, we will look for a sequence of finitely piecewise affine homeomorphisms; this is easier to achieve with our construction, and the two things are in fact equivalent. Indeed, as proved in~\cite{MCP}, given any finitely piecewise affine homeomorphism $v:\Q\to\u\Q$, coinciding with the identity on the boundary, there is a sequence of diffeomorphisms $v_j:\Q\to\u\Q$, also coinciding with the identity on the boundary, which converge to $v$ strongly in $W^{1,p}$ and uniformly.

\proofof{Theorem~\mref{Caratt}, sufficiency part}
Let $u:\Q\to\u\Q$ be a $W^{1,p}$ function, coinciding with the identity on $\partial\Q$ and satisfying the no-crossing condition, and let $\eps>0$ be given. To obtain the thesis, we will find a finitely piecewise affine homeomorphism $\omega:\Q\to\u\Q$, coinciding with the identity on $\partial\Q$, and such that
\begin{equation}\label{closeW1p}
\|\omega-u\|_{W^{1,p}(\Q)}< \eps\,;
\end{equation}
moreover, we will have to establish the ``improved convergence'' stated in the claim. The proof will be divided in several steps for the sake of clarity.

\step{I}{Definition of $\delta,\,\xi,\,\sigma$ and $K$ and subdivision in squares.}
Let $\delta=\delta(p,u,\eps)$ be a very small constant, to be specified later, and let $H_1,\,H_2$ and $H_3$ be the constants of Propositions~\ref{extgener}, \ref{extdet0} and~\ref{extdetsigma}. Since $u\in W^{1,p}$, we can select $\xi=\xi(p,u,\eps,\delta)\ll \delta$ such that for every set $E\subseteq\Q$ one has
\begin{equation}\label{defxi}
|E|<\xi \qquad \Longrightarrow \qquad \int_E |Du|^p < \frac{\delta^2}{4100 H_1}\,.
\end{equation}
Let now $\sigma=\sigma(p,u,\eps,\delta,\xi)\ll \xi$ be a constant such that, calling
\[
\A_\sigma=\big\{M\in \R^{2\times 2}:\, \hbox{either } |M|\in (0,\sigma)\cup (1/\sigma,+\infty),\, \hbox{or } \det M\in (0,\sigma)\cup (1/\sigma,+\infty)\big\}\,,
\]
one has
\begin{align}\label{xi2}
\big|\big\{x:\, Du(x)\in \A_\sigma\big\}\big|<\frac \xi 2\,, &&
615\cdot 2^p\sigma \,\frac{H_1+H_2} \delta< \delta\,.
\end{align}
Let now $K=K(p,u,\eps,\delta,\xi,\sigma)\gg 1/\sigma$ be a large integer, to be fixed in a moment. We will subdivide the $K^2$ squares of the grid $\G=\G(K)$ in four groups, as follows. A square $\RR_{i,j}$, with $1\leq i,\,j< K$, will be called a ``Lebesgue square adapted to $\sigma$'' if there is a matrix $M_{i,j}\notin \A_\sigma$ such that $\RR_{i,j}$ is a Lebesgue square with matrix $M_{i,j}$ and constant $\kappa=\sigma^{3p}$. All the other squares will be called ``bad squares'', and we let $\Q_1$ be their union, and set $M_{i,j}=0$ for any bad square $\RR_{i,j}$. Thanks to Lemma~\ref{moltiLeb} and to the first estimate in~(\ref{xi2}), up to take $K=K(p,u,\eps,\delta,\xi,\sigma)\gg 1$ big enough, we have
\begin{equation}\label{Q1small}
\big| \Q_1\big| < \xi\,.
\end{equation}

Let us now take a Lebesgue square $\RR_{i,j}$ adapted to $\sigma$, and let us consider the matrix $M_{i,j}$. If $M_{i,j}=0$, we will say that $\RR_{i,j}$ is a ``Lebesgue zero square'', if $M_{i,j}\neq 0$ but $\det M_{i,j}=0$ we will say that $\RR_{i,j}$ is a ``Lebesgue square with zero determinant'', and otherwise we will say that $\RR_{i,j}$ is a ``general Lebesgue square''. We will call $\Q_2$, $\Q_3$ and $\Q_4$ the union of all the Lebesgue zero squares, the Lebesgue squares with zero determinant, and the general Lebesgue squares respectively.

\step{II}{Definition of the functions $v:\Q\to\u\Q$ and $\varphi:\G\to\u\Q$.}
Let us now apply Lemma~\ref{startinggrid} to the function $u$, with the constants $\delta$ and $K$ and the matrices $M_{i,j}$ introduced in Step~I. 
We obtain then a $(1+4\delta)$-biLipschitz homeomorphism $\Phi:\Q\to\Q$, coinciding with the identity on $\partial\Q$, such that $\|\Phi-{\rm Id}\|_{L^\infty}< 2\delta/K$ and that~(\ref{tocon}) holds for every $1\leq i,\,j<K$. We let then $v=u\circ \Phi:\Q\to\u\Q$. Notice that $v$ is also a $W^{1,p}$ function coinciding with the identity on a neighbourhood of $\partial\Q$ and satisfying the no-crossing condition, and in addition $\G$ is a good starting grid for $v$. As a consequence, we can apply Proposition~\ref{injectivelinearization} to find an injective, piecewise linear function $\varphi:\G\to\u\Q$ satisfying~(\ref{linearization}). We will construct our function $\omega:\Q\to\u\Q$ satisfying~(\ref{closeW1p}) so that $\omega=\varphi$ on $\G$. Notice that, since $\varphi$ is piecewise linear, injective and coincides with the identity on $\partial\Q$, we can define the extension $\omega$ independently on every square of the grid $\G$: if each extension is finitely piecewise affine and injective, then the resulting $\omega$ is a finitely piecewise affine homeomorphism.

\step{III}{Extension on the ``bad squares'' $\RR_{i,j}\subseteq\Q_1$.}
Let us start by defining the extension $\omega$ on a given bad square $\RR=\RR_{i,j}\in\Q_1$. Since by construction $M_{i,j}=0$, the estimates~(\ref{linearization}) of Proposition~\ref{injectivelinearization} and~(\ref{tocon}) of Lemma~\ref{startinggrid} ensure that
\[\begin{split}
\int_{\partial\RR} |\varphi'|^p &\leq \frac \sigma K +2 \int_{\partial\RR} |Dv|^p
= \frac \sigma K + 2 \int_{\partial\RR} |D(u\circ \Phi)|^p\\
&\leq \frac \sigma K + 2(1+4\delta)^{p+1} \int_{\Phi(\partial\RR)} |Du|^p
\leq \frac \sigma K +\frac{410K} \delta \int_{\RR^+} |Du|^p\,.
\end{split}\]
By Proposition~\ref{extgener} we obtain then a finitely piecewise affine homeomorphism $\omega:\RR\to\R^2$, coinciding with $\varphi$ on $\partial\RR$, and such that
\[
\int_\RR |D\omega|^p \leq 410\,\frac{H_1} \delta \int_{\RR^+} |Du|^p+\sigma H_1 |\RR|\,.
\]
Keeping in mind~(\ref{Q1small}) and~(\ref{defxi}), we obtain then that
\[
\int_{\Q_1} |D\omega|^p \leq \frac 9{10}\,\delta + \sigma H_1 \xi\,,
\]
hence again by~(\ref{defxi}) and recalling that $\xi\ll\sigma\ll \delta$
\begin{equation}\label{est1}
\int_{\Q_1} |D\omega-Du|^p\leq 2^p \int_{\Q_1} |D\omega|^p + |Du|^p
\leq 2^p \delta\,.
\end{equation}

\step{IV}{Extension on the ``Lebesgue zero squares'' $\RR_{i,j}\subseteq\Q_2$.}
Let us now consider a Lebesgue zero square $\RR=\RR_{i,j}$. Since $M_{i,j}$ is again $0$, exactly as in Step~III we can define the extension $\omega$ by Proposition~\ref{extgener} and notice that
\begin{equation}\label{zerouno}
\int_\RR |D\omega|^p \leq 410\,\frac{H_1} \delta \int_{\RR^+} |Du|^p+\sigma H_1 |\RR|\,.
\end{equation}
Then, we cannot continue the argument of Step~III, because the area of $\Q_2$ needs not to be small. However, since $\RR$ is a Lebesgue square with matrix $M=0$ and constant $\sigma^{3p}$, by definition
\begin{equation}\label{zerodue}
\int_{\RR^+} |Du|^p \leq \sigma^{3p} |\RR|<\sigma |\RR|\,.
\end{equation}
Hence, by~(\ref{zerouno}) we obtain
\[
\int_\RR |D\omega|^p \leq \bigg(\frac{410}\delta+1\bigg)\, H_1 \sigma |\RR|\leq \frac{411}\delta\, H_1 \sigma |\RR|\,,
\]
which again by~(\ref{zerodue}) and keeping in mind the second estimate in~(\ref{xi2}) gives
\begin{equation}\label{est2}
\int_{\Q_2} |D\omega-Du|^p \leq 2^p \int_{\Q_2} |D\omega|^p + |Du|^p
\leq 2^p \sigma \,\bigg(411\,\frac{H_1} \delta\,+1\bigg)\, |\Q_2|\leq \delta\,.
\end{equation}

\step{V}{Extension on the ``Lebesgue squares with zero determinant'' $\RR_{i,j}\subseteq\Q_3$.}
Let now $\RR=\RR_{i,j}$ be a Lebesgue square with zero determinant, and let us call for brevity $M=M_{i,j}$ the corresponding matrix. First of all, since $\RR$ is a Lebesgue square with area $1/K^2$ and constant $\sigma^{3p}$, and $|M|>\sigma$, we have
\[
\|Du-M\|^p_{L^p(\RR)} < \sigma^{3p} |\RR| <\sigma^{2p} |M|^p |\RR| = (\sigma^2 \|M\|_{L^p(\RR)})^p\,,
\]
hence $\|Du\|_{L^p(\RR)}\geq \|M\|_{L^p(\RR)}(1-\sigma^2)$, thus
\begin{equation}\label{anchedopo}
\int_\RR |Du|^p > (1-\sigma^2)^p |M|^p |\RR| > (1-\sigma) |M|^p |\RR|\,.
\end{equation}
Now, by the estimates~(\ref{linearization}) of Proposition~\ref{injectivelinearization} and~(\ref{tocon}) of Lemma~\ref{startinggrid}, and recalling that for every $a,\,b\in\R$ it is $(a+\delta b)^p\leq a^p + 2p\delta(a^p+b^p)$, we have
\[\begin{split}
\int_{\partial \RR} |\varphi'-M\tau|^p &\leq \frac \sigma K+2 \int_{\partial \RR} |D v-M|^p
=\frac \sigma K+2 \int_{\partial \RR} |D (u\circ\Phi)-M|^p\\
&\leq \frac \sigma K+2 (1+4\delta)^{p+1} \int_{\Phi(\partial\RR)} \big(|Du-M| +5\delta |M|\big)^p\\
&\leq \frac \sigma K+3 \int_{\Phi(\partial\RR)} (1+10p\delta)|Du-M|^p + 10p\delta |M|^p\\
&\leq \frac \sigma K+\frac{121p\delta}K\, |M|^p + \frac{613K} \delta \int_{\RR^+} |Du-M|^p
\leq \frac{121p\delta}K\, |M|^p + \frac{614}{K\delta}\, \sigma\,,
\end{split}\]
where the last inequality comes recalling that $\RR$ is a Lebesgue square with matrix $M$ and constant $\sigma^{3p}<\sigma$. Since $M$ has zero determinant, we can apply Proposition~\ref{extdet0} to get a finitely piecewise affine homeomorphism $\omega$ on $\RR$, coinciding with $\varphi$ on $\partial\RR$, such that
\[
\int_\RR |D\omega -M|^p \leq \frac{H_2}{K^2}\, \bigg(121p\delta\, |M|^p + \frac{614}\delta\, \sigma\bigg)\,.
\]
So, from~(\ref{anchedopo}) we deduce
\[
\int_\RR |D\omega-Du|^p
\leq
2^p \bigg(\int_\RR |D\omega -M|^p + \int_\RR |Du-M|^p\bigg)
\leq 2^{p+7} p \delta H_2 \int_\RR |Du|^p +615\, \frac{2^p H_2}\delta\, \sigma |\RR|\,,
\]
and again by the second estimate in~(\ref{xi2}), adding over all the squares contained in $\Q_3$ we get
\begin{equation}\label{est3}
\int_{\Q_3} |D\omega-Du|^p \leq \delta\bigg( 1+ H_2 2^{p+7} p \int_{\Q_3} |Du|^p \bigg).
\end{equation}

\step{VI}{Extension on the ``general Lebesgue squares'' $\RR_{i,j}\subseteq\Q_4$.}
Let us finally consider a general Lebesgue square $\RR=\RR_{i,j}$, with corresponding matrix $M=M_{i,j}$. The calculation already done in Step~V, together with the fact that $614\sigma<\delta^2$ by~(\ref{xi2}), ensures
\[
\int_{\partial \RR} |\varphi'-M\tau|^p \leq \frac{121p\delta}K\, |M|^p + \frac{614}{K\delta}\, \sigma
<(121p\delta |M|^p + 1)\, \frac \delta K\,.
\]
As a consequence, by Proposition~\ref{extdetsigma} we get a finitely piecewise affine homeomorphism $\omega$ on $\RR$, coinciding with $\varphi$ on $\partial\RR$, such that also by~(\ref{anchedopo}) we have
\[\begin{split}
\int_\RR |D\omega - M|^p &\leq \frac{H_3}{K^{2-\frac 1p}} \bigg(\int_{\partial\RR} |\varphi' - M\tau|^p\bigg)^{\frac 1p}
\leq \frac{H_3}{K^2} \, (121p\delta |M|^p + 1) \delta^{1/p}\\
&\leq H_3 \delta^{1/p} |\RR| + 122 p \delta H_3 \int_\RR |Du|^p\,.
\end{split}\]
Arguing as usual, we get then
\[\begin{split}
\int_\RR |D\omega - Du|^p &\leq 
2^p \bigg(\int_\RR |D\omega-M|^p +\int_\RR |Du - M|^p\bigg)\\
&\leq 2^p |\RR| \big(H_3 \delta^{1/p}+ \sigma \big) + 2^{p+7} p \delta H_3 \int_\RR |Du|^p\,,
\end{split}\]
which adding on the squares contained in $\Q_4$ gives
\begin{equation}\label{est4}
\int_{\Q_4} |D\omega-Du|^p \leq \delta^{1/p} \bigg(2^{p+1} H_3 + 2^{p+7} p H_3 \int_{\Q_4} |Du|^p \bigg)\,.
\end{equation}

\step{VII}{Existence of the approximation.}
In Steps~III--VI we have defined the function $\omega$ on each of the finitely many squares $\RR_{i,j}$. By construction, the resulting function $\omega:\Q\to\u\Q$ is a finitely piecewise affine homeomorphism coinciding with the identity on $\partial\Q$. Moreover, by~(\ref{est1}), (\ref{est2}), (\ref{est3}) and~(\ref{est4}), and since $\omega-u\in W^{1,p}_0$, we get the validity of~(\ref{closeW1p}) as soon as $\delta$ has been chosen small enough, only depending on $p$, $u$ and $\eps$. The proof of the existence of a sequence strongly converging to $u$ is then concluded.

\step{VIII}{Proof of points (i) and (ii).} To conclude the sufficiency part of Theorem~\mref{Caratt}, we need to show the validity of points~(i) and~(ii) in the statement. To this end we first note that, on each square $\RR$ of the starting grid $\G$, the map \(\omega\) constructed in the above steps satisfies 
\[
\|u-\omega\|_{L^\infty (\R)}\le C\Big(\osc_{\mathcal R} u +\osc_{\partial \mathcal R} u+\eta \big)\,,
\]
where \(\eta\) is the size of the grid \(\widetilde \G \) and \(C\) is a purely geometric constant. Hence, to obtain the validity of points~(i) and~(ii), we need to show that given \(\eps\) small and setting
\[
\mathcal S=\bigcup_{\mathcal R: \osc_{\partial\RR} u+\osc_{\RR} u \geq 2\eps} \mathcal R\,,
\]
then, as soon as the starting and arrival grids $\G$ and $\widetilde\G$ are fine enough, we have \(\mathcal S=\emptyset \) if \(u\) is continuous, and \(\H^{2-p}_\infty(\mathcal S)\le \eps\) if \(p\in[1,2]\). 

Being the first part clear by uniform continuity, we will just prove the second one. For, note that by Corollary~\ref{cor:diff} \(u\in INV^+(\Q)\) and thus, by~\eqref{eq:osc2}, for every square $\RR$ of the starting grid $\G$ one has
\[
\big(\osc_{\mathcal R} u \big)^p \le \big(\osc_{\partial \mathcal R} u \big)^p \le C \ell(\mathcal R)^{p-2} \int_{\mathcal R^+} |Du|^p,
\]
where \(\mathcal R^+\) is the union of $\RR$ with the squares of $\G$ which are adjacent to it, and where $\ell(\RR)=1/K$ is the side of $\RR$. In particular, if $\osc_{\partial\RR} u +\osc_{\RR} u\geq 2\eps$, then $\osc_\RR u \geq \eps$ hence
\[
|\RR| = \ell(\RR)^2 \leq \frac{C\ell(\RR)^p}{\eps^p} \int_{\RR^+} |Du|^p=\frac{CK^{-p}}{\eps^p} \int_{\RR^+} |Du|^p\,,
\]
and then, up to multiply $C$ by $9$, we have
\begin{equation}\label{osc3}
|\mathcal S|\le \frac{ C K^{-p}}{\eps^p} \int_{\mathcal S^+} |Du|^p\,,
\end{equation}
being $\SS^+$ the union of all the squares which either belong to $\SS$, or are adjacent to a square of $\SS$. Since $Du\in L^p(\Q)$, as soon as $K$ is big enough depending on $\eps$ we have by the above estimate that $\SS^+$ is as small as we wish, hence by the absolute continuity of the integral we deduce
\[
\int_{\mathcal S^+} |Du|^p \le \frac{\eps^{p+1}}{C}\,,
\]
so by~\eqref{osc3} we get that
\[
\H^{2-p}_{\infty}(\mathcal S)\leq \sum_{\mathcal R: \osc_{\partial R} u+\osc_{ R} u \geq 2\eps} \ell(\mathcal R)^{2-p}
= |\SS| K^p \leq \frac{C}{\eps^p}\int_{\SS^+} |Du|^p\leq \eps\,,
\]
and this concludes the proof of the final claim in Theorem~\mref{Caratt} by suitably choosing \(\eps\). 
\end{proof}

\subsection{Proof of Theorem~\mref{Caratt} (necessity) and of Theorem~\mref{WeakStrong}}\label{sec:DimB}

In this Section we show that the no-crossing condition is necessary, for a function, to be approximable by diffeomorphisms, thus completing the proof of Theorem~\mref{Caratt} and of Theorem~\mref{WeakStrong}. 
\begin{lemma}\label{lem:nec}
Let $\{u_k\} \subseteq \mathcal D_{\rm id} (\Q)\cap W^{1,p}(\Q)$ be a sequence, weakly converging to \(u\) in \(W^{1,p}(\Q)\). Then \(u\) satisfies the no-crossing condition.
\end{lemma}

\begin{proof}
We start noticing that by Corollary~\ref{cor:diff} and Lemma~\ref{lm:continuity} \(u\) is continuous outside a set \(\SS_u\) of zero \(\H^1\)-measure. Moreover, by extending each \(u_k\) as the identity outside \(\Q\) and by performing a small dilation, we can assume that all the maps \(u_k\) are equal to the identity on a neighbourhood of \(\partial \Q\).

Let us now assume that \(p=1\), the other case being analogous. By Dunford-Pettis Theorem there exists a superlinear function \(g\) such that
\[
\sup_k \int_\Q g(|Du_k|)<+\infty.
\]
Let \(\eps>0\) and be \(\Gamma\subseteq \Q\setminus\SS_u\) be as in the Definition~\ref{defnocross}. Note that, by pre-composing all the maps with a Bi-Lipischitz transformation, we can assume without loss of generality that \(\Gamma\) is composed by piecewise linear curves. In particular, by making small translations of all the segments belonging to \(\Gamma\) and by slightly extending them, we can find, for almost every small \(\delta\), a set \(\Gamma_\delta\) with the following properties,
\begin{itemize}
\item[-] \(\Gamma_\delta=\Phi_\delta(\Gamma)\) where \(\Phi_\delta\) is a bi-Lipschitz transformation equal to the identity on \(\partial \Q\) and such that \(\|\Phi_\delta-\id\|_{L^\infty}\le \delta\)\,;

\item[-] there exists a (not relabed) subsequence such that $u_k\deb u$ in $W^{1,1}(\Gamma_\delta)$.
\end{itemize}
By the continuity of \(u\) on $\Gamma$, we can first chose \(\delta\) sufficiently small such that
\[
\|u -u\circ \Phi_{\delta}\|_{L^\infty(\Gamma)} \le \eps/2
\]
and then, by the \(1\)-dimensional Sobolev embedding Theorem, \(\bar k\) sufficiently big such that
\[
\|u_{\bar {k}}-u\|_{L^\infty(\Gamma_{\delta})}\leq \eps/2.
\]
The map \(\psi=u_{\bar k} \circ \Phi_\delta:\Gamma\to \u \Q\) is then the injective modification of \(u\).
\end{proof}

\begin{proof}[Proof of Theorem~\mref{Caratt}]	
The sufficiency part has been showed in Section~\ref{ss: suff} while the necessary one follows immediately from Lemma~\ref{lem:nec}.
\end{proof}

\begin{proof}[Proof of Theorem~\mref{WeakStrong}]

We divide the proof in two short steps.
\step{I}{The set of maps satisfying the no-crossing condition is (sequentially) weakly closed.} Let \(\{u_j\}\) be a sequence of maps satisfying the no-crossing condition and such that \(u_j\deb u\). We claim that \(u\) satisfies the no-crossing condition as well. Indeed by Theorem~\mref{Caratt} for every \(j\) there is a diffeomorphism \(v_j\) such that \(\|v_j-u_j\|_{W^{1,p}}\le 2^{-j}\), hence \(v_j\deb u\) and by Lemma~\ref{lem:nec}, \(u\) satisfies the no-crossing condition.
\step{II}{Conclusion.}
By the previous step, if \(u\in \overline{\D_{\id}(\Omega)}^{\,{\rm ws}-W^{1,p}}\) then \(u\) satisfies the no-crossing condition. Hence, by Theorem~\mref{Caratt}, \(u\in \overline{\D_{\id}(\Omega)}^{\,{\rm s}-W^{1,p}}\), and this concludes the proof.
\end{proof}

\section{Proof of Theorem~\mref{NonSconn}\label{sec:C}}

This section is devoted to prove Theorem~\mref{NonSconn}. Through the section, $u$ will be a given $W^{1,p}$ function, coinciding with the identity on the boundary, such that the counter-image (in the sense of Definition~\ref{def:multifunction}) of any closed subset of $\u\Q$ which does not disconnect $\u\Q$ is a closed, connected subset of $\Q$ which does not disconnect it. In particular, by Lemma~\ref{ThCINV} we know that $u$ satisfies the INV$^+$ condition, thus by Lemma~\ref{lm:continuity} we find a $\H^1$-negligible set $\SS_u\subseteq\Q$ such that $u$ is continuous on $\Q\setminus\SS_u$. The proof is divided in some subsections for clarity.

\subsection{Preliminary geometrical properties}

Let us start by listing some simple geometrical properties of curves and connected sets.
\begin{lemma}\label{sard}
Let $\Gamma\subseteq{\rm Int}(\Q)$ be a closed, connected set which does not disconnect $\Q$ (that is, $\Q\setminus \Gamma$ is connected). Then, for every $\eps>0$ there exists a ${\rm C}^1$, closed, injective curve $\theta:\S^1\to {\rm Int}(\Q)$, such that $\Gamma$ is contained in the internal part of $\theta$, and for every $s\in\S^1$ one has ${\rm dist}(\theta(s),\Gamma)<\eps$.
\end{lemma}
\begin{proof}
Let $\varphi:\Q\to\R^+$ be a regularised distance function from $\Gamma$, that is, $\varphi$ is a regular function and for every $x\in\Q$ one has
\[
{\rm dist}(x,\Gamma) \leq \varphi(x) < 2 {\rm dist}(x,\Gamma)\,.
\]
By Sard's Lemma, there is some $0<\eps'\leq \eps$ such that $\varphi>\eps'$ on $\partial\Q$ and $\varphi^{-1}(\eps')$ is the union of finitely many ${\rm C}^1$ curves. Since the level set $\{\varphi=\eps'\}$ disconnects $\Gamma$ and $\partial\Q$, one of those curves must contain $\Gamma$ in its internal part. 
\end{proof}

\begin{lemma}\label{1or2inter}
Let $\C\subseteq {\rm Int}(\Q)$ be a closed set which does not disconnect $\Q$, and let $\gamma:[0,1]\to{\rm Int}(\Q)$ be an admissible curve in the sense of Definition~\ref{admcur}. If $\gamma$ intersects $\C$ exactly once, then $\C\cup \gamma$ does not disconnect $\Q$; if $\gamma$ intersects $\C$ exactly twice, then $\Q\setminus(\C\cup\gamma)$ has at most two connected components, in particular surely two if $\C$ is connected.
\end{lemma}
\begin{proof}
For any number $\delta>0$, much smaller than the diameter of $\gamma$ and the distance between $\gamma$ and $\partial\Q$, we call $\A_\delta$ the set of the points $x\in\Q\setminus (\C\cup\gamma)$ with distance larger than $\delta$ from $\C\cup\gamma$, and which can be connected with $\partial\Q$ with a path having distance at least $\delta$ from $\C$. Notice that, since $\C$ does not disconnect $\Q$, the sets $\A_\delta$ fill the whole $\Q\setminus (\C\cup\gamma)$ when $\delta\searrow 0$.\par

Let us now assume that $\gamma$ intersects $\C$ only once, and precisely at $\gamma(0)$. For every $\delta\ll {\rm diam}(\gamma)$ we can find $\eta\ll \delta$ such that, if for some $\bar t\in (0,1)$ one has $|\gamma(0)-\gamma(\bar t)|=\eta$, then for all $0<t<\bar t$ one has $|\gamma(0)-\gamma(t)|<\delta$. Applying now Lemma~\ref{sard} with $\eps\ll \eta$ to the image of $\gamma$, which is a closed, connected set which does not disconnect $\Q$, we find a ${\rm C}^1$, closed, injective curve $\theta$, containing $\gamma$ in its internal part, and so that ${\rm dist}(\theta(s),\gamma)<\eps$ for every $s\in\S^1$. If $\eps$ is small enough, depending on $\C$ and $\eta$, then the intersection between $\theta$ and $\C$ contains only points which have distance strictly less than $\eta$ from $\gamma\cap \C=\gamma(0)$. There is a unique subpath $\theta_1$ of $\theta$, with ${\rm diam}(\theta_1)\approx {\rm diam}(\theta)$, whose both endpoints have distance $\eta$ from $\gamma(0)$, and whose internal points have all distance strictly larger than $\eta$ from $\gamma(0)$; by construction, $\theta_1\cap (\C\cup\gamma)=\emptyset$. We can extend $\theta_1$ to an injective, closed curve $\tilde\theta$, in such a way that $\tilde\theta\setminus\theta_1$ is an arc of the circle $\{z\in\Q:\, |z-\gamma(0)|=\eta\}$; it is possible to do this in two distinct ways, we just pick one of them. Notice that $\gamma$ is not necessarily contained in the internal part of $\tilde\theta$; nevertheless, let us call $\bar t\in[0,1]$ the smallest number such that the restriction of $\gamma$ to $[\bar t,1]$ is contained in the internal part of $\tilde\theta$. Since $\gamma$ does not intersect $\theta_1$, then either $\bar t=0$ or $\gamma(\bar t)$ has distance exactly $\eta$ from $\gamma(0)$: in both cases, the curve $\gamma$ in $[0,\bar t]$ only contains points within distance $\eta\ll\delta$ from $\gamma(0)$.\par

Let now $x\in\A_\delta$, and let $\sigma$ be a path connecting $x$ with $\partial\Q$ having always distance at least $\delta$ from $\C$. If $\sigma\cap \tilde\theta=\emptyset$, then $\sigma$ is entirely contained in the external part of $\tilde\theta$ (we can exclude that $\sigma$ is entirely in the internal part because $x$ has distance larger than $\delta$ from $\gamma$). As a consequence, $\sigma$ does not contain any point of the restriction of $\gamma$ to $[\bar t,1]$; on the other hand, $\sigma$ does also not contain points with distance less than $\delta$ from $\gamma(0)\in \C$, hence it cannot intersect the restriction of $\gamma$ to $[0,\bar t]$. As a consequence, the whole path $\sigma$ is contained in $\Q\setminus (\C\cup\gamma)$, so $x$ is in the same connected component of $\Q\setminus (\C\cup\gamma)$ as $\partial\Q$. Assume instead that $\sigma\cap\tilde\theta\neq \emptyset$, and let $x^+$ and $y^-$ be the first and the last point of $\sigma$ which intersect $\tilde\theta$; arguing as before, the curve $\sigma$ cannot intersect $\gamma$ between $x$ and $x^+$, nor between $y^+$ and the last point of $\sigma$, and moreover $x^+$ and $y^-$ belong to $\theta_1$; indeed, points of $\sigma$ have distance at least $\delta\gg \eta$ from $\C$, while points of $\tilde\theta\setminus\theta_1$ have distance $\eta$ from $\gamma(0)\in\C$. Then, we can find a path between $x$ and $\partial\Q$ putting together $\sigma$ from $x$ to $x^+$, then $\theta_1$ between $x^+$ and $y^-$, and finally again $\sigma$ between $y^-$ and $\partial\Q$. By construction, this path does not intersect $\C\cup\gamma$, so again $x$ is in the same connected component of $\Q\setminus (\C\cup \gamma)$ as $\partial\Q$. Summarizing, $\A_\delta$ belongs to a connected component of $\Q\setminus (\C\cup\gamma)$, and since this holds for every $\delta\ll 1$ we deduce that $\C\cup\gamma$ does not disconnect $\Q$. We have then obtained the thesis in the case that $\gamma$ intersect $\C$ only at $\gamma(0)$, and of course the same argument works if the only intersection point is at $\gamma(1)$.\par

Assume now that $\gamma$ intersects $\C$ only once, and precisely at $\gamma(t)$ for some $0<t<1$. Then, we can write $\gamma=\gamma_1\cup\gamma_2$, where $\gamma_1$ and $\gamma_2$ are the restrictions of $\gamma$ to $[0,t]$ and $[t,1]$ respectively. The above argument applied to $\C$ and $\gamma_1$ ensures that $\C\cup \gamma_1$ is a closed set which does not disconnect $\Q$. And then, the above argument again applied to $\C\cup \gamma_1$ and $\gamma_2$ ensures that also $(\C\cup\gamma_1)\cup \gamma_2=\C\cup \gamma$ does not disconnect $\Q$.\par

Let us now pass to consider the case that $\gamma$ intersects $\C$ exactly twice; the very same argument of the last sentence ensures that it is not restrictive to assume that the intersection points are $\gamma(0)$ and $\gamma(1)$.\par

The argument is similar as before: for any $\delta$ much smaller than the diameter of $\gamma$ and the distance between $\gamma$ and $\partial\Q$ we select $\eta\ll \delta$ such that if for some $\bar t\in (0,1)$ one has $|\gamma(0)-\gamma(\bar t)|=\eta$ (resp., $|\gamma(1)-\gamma(\bar t)|=\eta$), then for all $0<t<\bar t$ (resp., all $\bar t < t < 1$) one has $|\gamma(0)-\gamma(t)|<\delta$ (resp., $|\gamma(1)-\gamma(t)|<\delta$). Let then again $\theta$ be given by Lemma~\ref{sard}, with $\eps\ll \eta$ in such a way that $\theta\cap \C$ only contains points with distance strictly less than $\eta$ from $\gamma(0)$ or $\gamma(1)$. This time, we find two distinct subpaths $\theta_1$ and $\theta_2$ of $\theta$, each of them with the property that one endpoint has distance $\eta$ from $\gamma(0)$, the other one has distance $\eta$ from $\gamma(1)$, and all the interior points (that is, all the points which are not endpoints) have distance strictly larger than $\eta$ from both $\gamma(0)$ and $\gamma(1)$. Again, by construction $\theta_1$ and $\theta_2$ do not intersect $\C\cup\gamma$. We define now a closed, injective curve $\tilde\theta$ by putting together $\theta_1$, $\theta_2$ and two arcs of circle with radius $\eta$, one with center at $\gamma(0)$ and the other at $\gamma(1)$. Up to decrease $\delta$ if necessary, we can assume the existence of a path between $\gamma(1/2)$ and $\partial\Q$ having distance larger than $\delta$ from $\C$, so in particular from $\gamma(0)\cup\gamma(1)$; since $\gamma(1/2)$ is in the internal part of $\tilde\theta$, there exists a last point of this path in $\tilde\theta$, and we assume without loss of generality that this point belongs to $\theta_1$. In other words, all the points of $\theta_1$ are in the same connected component as $\partial\Q$ inside $\Q\setminus (\C\cup\gamma)$. Let us now take any point $x\in\A_\delta$, and let us consider a path $\sigma$ between $x$ and $\partial\Q$ having distance larger than $\delta$ from $\C$. If $\sigma$ does not intersect $\tilde\theta$, then exactly as before we deduce that it also does not intersect $\gamma$, so $x$ is in the same connected component of $\Q\setminus(\C\cup\gamma)$ as $\partial\Q$, hence also the same as $\theta_1$. Otherwise, let $x^+$ be the first point of $\sigma$ which intersects $\tilde\gamma$: as before, we know that $x^+$ belongs either to $\theta_1$ or to $\theta_2$. Summarizing, every point of $\A_\delta$ is connected, in $\Q\setminus(\C\cup\gamma)$, either with $\theta_1$ or with $\theta_2$; then, every $\A_\delta$ intersects at most two connected components of $\Q\setminus(\C\cup\gamma)$, hence $\Q\setminus(\C\cup\gamma)$ has at most two connected components.\par

To conclude, we assume that $\C\cap\gamma=\gamma(0)\cup \gamma(1)$ and that $\Q\setminus (\C\cup\gamma)$ is connected, and we have to prove that $\C$ is not connected. Let the paths $\theta_1$ and $\theta_2$ be as before, and let $x\in\theta_1$ and $y\in\theta_2$ be two points very close to $\gamma(1/2)$. Since $\Q\setminus (\C\cup\gamma)$ is connected, there is a path $\sigma$ between $x$ and $y$ which does not intersect $\gamma\cup \C$, and we can easily assume that this path does not intersect the open segment $xy$. As a consequence, the union between $\sigma$ and the segment $xy$ is an injective, closed curve, which does not intersect $\C$; by construction, one between $\gamma(0)$ and $\gamma(1)$ is in the internal part of this curve, and the other one is in the external part, hence $\C$ is not connected and the proof is concluded.
\end{proof}

\begin{lemma}\label{opposite}
Let $\gamma:[0,1]\to\Q\setminus\SS_u$ be an admissible curve with both endpoints in $\partial\Q$, let $\u P\in{\rm Int}(\u\Q)$ and assume that $(u\circ\gamma)^{-1}(\u P)$ consists of exactly two numbers $t_1$ and $t_2$ and that $v_1=D(u\circ\gamma)(t_1)$ and $v_2=D(u\circ\gamma)(t_2)$ are both well-defined and non-zero. Then, the two vectors $v_1$ and $v_2$ are parallel and with opposite directions.
\end{lemma}
\begin{proof}
Since $u={\rm Id}$ in a neighborhood of $\partial\Q$, we can assume without loss of generality that $\gamma$ intersects $\partial\Q$ only at $\gamma(0)$ and $\gamma(1)$, and we can extend $\gamma$ to an injective, closed curve $\tilde\gamma$ with $\tilde\gamma\setminus\gamma\subseteq \partial\Q$. Since $u$ is continuous also on $\tilde\gamma\subseteq\Q\setminus\SS_u$, by Lemma~\ref{ThCINV} we know that $u$ satisfies the INV condition and, in particular, that the points in $\u\Q\setminus u(\tilde\gamma)$ have all degree $0$ or $1$ with respect to $u(\tilde\gamma)$. Since $u^{-1}(\u P)\cap\gamma=\{t_1,\, t_2\}$, $u\circ\gamma(t)$ is close to $\u P$ only if $t$ is close to $t_1$ or $t_2$; as a consequence, the points near $\u P$ which belong to the image of $u\circ\gamma$ must have direction with respect to $\u P$ very close to $v_1$ or $v_2$. Call $\omega_1$ and $\omega_2$ the unit vectors obtained by rotating clockwise of $90^\circ$ the directions of $v_1$ and $v_2$. Assume first that $v_1$ and $v_2$ are not parallel, hence so are neiter $\omega_1$ and $\omega_2$: then, there is a short segment $\sigma$ near $\u P$, with direction $\nu$, which intersects $u\circ\gamma$ exactly twice, with $\nu\cdot \omega_1>0$, $\nu\cdot \omega_2>0$. As a consequence, the degree of the last point of $\sigma$ equals the degree of the first one plus two, hence the two degrees cannot both be in $\{0,\, 1\}$ and we have a contradiction. Assume instead that $v_1$ and $v_2$ are parallel: then, we can find again a segment $\sigma$ near $\u P$, with direction $\nu$, which intersects $u\circ\gamma$ exactly twice, and we can assume that $\nu\cdot \omega_1>0$. If the direction of $v_2$ is the same as that of $v_1$, then also $\nu\cdot \omega_2>0$ and we have a contradiction as before.
\end{proof}

\begin{lemma}\label{inclsegm}
Let $\u P\in {\rm Int}(\u\Q)$, and let the segments $A^-A^+$ and $B^-B^+$ be two admissible curves in ${\rm Int}(\Q)\setminus\SS_u$ which intersect $u^{-1}(\u P)$ exactly in two points $A\in A^-A^+$ and $B\in B^-B^+$. Then, there exists an admissible curve $\gamma:[0,1]\to\Q\setminus\SS_u$, with $\gamma(0)$ and $\gamma(1)$ in $\partial\Q$, which intersects $u^{-1}(\u P)$ exactly at $A$ and $B$, and which contains the segment $A^-A^+$ and either $B^-B^+$ or $B^+B^-$. In particular, only one of these two possibilities can occur if the derivatives of $u$ at $A$ in the direction of $A^-A^+$ and at $B$ in the direction of $B^-B^+$ are both well-defined and non-zero.
\end{lemma}
\begin{proof}
Let us call for brevity $\C=u^{-1}(\u P)$, which is a closed, connected set which does not disconnect $\Q$. By Lemma~\ref{1or2inter}, we know that $\C\cup A^-A^+\cup B^-B^+$ does not disconnect $\Q$, hence we find an injective curve $\gamma_0$ between $\partial\Q$ and $A^-$ whose interior does not intersect $\C\cup A^-A^+\cup B^-B^+$. Since this curve has a strictly positive Hausdorff distance from $\C$, we can assume without loss of generality that it is an admissible curve; moreover, since $\H^1(\SS_u)=0$, we can also assume without loss of generality that $\gamma_0\cap\SS_u=\emptyset$. Again by Lemma~\ref{1or2inter}, $\C\cup\gamma_0\cup A^-A^+\cup B^-B^+$ does not disconnect $\Q$, so we can find an admissible curve $\gamma_1\subseteq\Q\setminus\SS_u$ between $A^+$ and $B^-$, and an admissible curve $\gamma_2\subseteq\Q\setminus\SS_u$ between $B^+$ and $\partial\Q$, whose interiors do not intersect $\C\cup \gamma_0\cup A^-A^+\cup B^-B^+$. Up to a trivial modification of the curves $\gamma_1$ and $\gamma_2$, we can assume that $\gamma_1\cap\gamma_2$ is either empty or consists of exactly a point.\par

If $\gamma_1\cap\gamma_2=\emptyset$, then the curve $\gamma_0$, followed by the segment $A^-A^+$, then by the curve $\gamma_1$, the segment $B^-B^+$, and the curve $\gamma_2$, is an admissible curve in $\Q\setminus\SS_u$ starting and ending at $\partial\Q$, which intersects $\C$ exactly at $A$ and $B$ and which contains the segments $A^-A^+$ and $B^-B^+$. In particular, if the derivatives of $u$ at $A$ in the direction of $A^-A^+$ and at $B$ in the direction of $B^-B^+$ are both well-defined and non-zero, then by Lemma~\ref{opposite} these two derivatives are parallel and with opposite direction.\par

\begin{figure}[thbp]
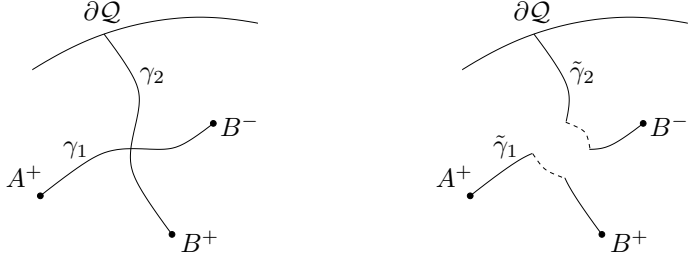
\caption{Definition of $\tilde\gamma_1$ and $\tilde\gamma_2$ starting from $\gamma_1$ and $\gamma_2$ in Lemma~\ref{inclsegm}.}\label{Fig:1}
\end{figure}

Instead, suppose that $\gamma_1\cap\gamma_2$ consist of a point; then, we can trivially build an admissible curve $\tilde\gamma_1$ between $A^+$ and $B^+$, and an admissible curve $\tilde\gamma_2$ between $B^-$ and $\partial\Q$, with empty intersection and with interiors not intersecting $\C\cup\gamma_0\cup A^-A^+\cup B^-B^+$ (see Figure~\ref{Fig:1}). In this case, the searched curve is $\gamma_0$, followed by $A^-A^+$, by $\tilde\gamma_1$, by $B^+B^-$, and by $\tilde\gamma_2$. Notice that, in this second case, if the derivatives of $u$ at $A$ and $B$ in the direction of $A^-A^+$ and $B^-B^+$ are both well-defined and non-zero, then they must be parallel and with the same direction, since $\gamma$ contains the oriented segment $B^+B^-$, and not $B^-B^+$. This concludes the proof.
\end{proof}

\subsection{Left and right connectibility}\label{sec:lr}

In this section we present the definition of the left and right connectibility, and we show their main properties.

\begin{defin}[Compatible curve]
Let $\sigma:[0,1]\to \Q\setminus\SS_u$ be an admissible curve such that $\sigma(0),\, \sigma(1) \in \partial\Q$ while $\sigma(t)\in {\rm Int}(\Q)$ for every $0<t<1$, and let $\u P\in{\rm Int}(\u\Q)$. We say that \emph{$\sigma$ is compatible with $\u P$} if $u^{-1}(\u P)\cap \sigma$ consists of finitely many points $\{\sigma(t_i),\, 1\leq i \leq N\}$, $\sigma$ contains a short segment around each $\sigma(t_i)$, each point $\sigma(t_i)$ is a Lebesgue point for $Du$,  $D(u\circ \sigma)$ is non-zero at each $t_i$ and 
\[
D(u\circ \sigma)(t)=Du(\sigma(t))\sigma'(t).
\]
We will call $\sigma^-$ and $\sigma^+$ the closures of the left and the right connected components of $\Q\setminus \sigma$.
\end{defin}

In the following Lemma we show that if \(\sigma\) is compatible with \(\u P\) and   $\H^{0}( u^{-1}(\u P))>1$, then \(Du\) has rank $1$ at each point of $\sigma\cap u^{-1}(\u P)$. This will be useful in the proof of Lemma~\ref{well-defined}.

\begin{lemma}\label{lm:direction}
Let \(\u P\in \u \Q\) be such that \(u^{-1}(\u P)\) consists of more than one element. Then, for every $Q\in u^{-1}(\u P)$, there exists exactly one direction $\nu\in\S^1$ such that, whenever $\sigma$ is a path compatible with $\u P$, then \(D u(Q)\nu =0\). Moreover, for every sequence \(\{Q_j\} \subseteq u^{-1}(\u P)\) converging to $Q$, one has
\begin{equation}\label{tan}
\nu(Q)=\lim_{j\to \infty} \pm \frac{Q_j-Q}{|Q_j-Q|}\,.
\end{equation}
\end{lemma}

\begin{proof}
To prove the first part of the Lemma, it is enough to assume the existence of a path $\sigma$ compatible with $\u P$ and containing $Q$, and to deduce that ${\rm rank}(Du(Q))=1$. Indeed, by definition $Du(Q)\neq 0$, so we only have to exclude that \(Du(Q)\) is invertible; to do so, we will show that there exists $0<r\ll 1$ such that 
\begin{equation}\label{cont}
u^{-1}(\u P)\cap \partial B(Q,r)=\emptyset\,,
\end{equation}
which gives a contradiction with the fact that \(u^{-1}(\u P)\) is a connected set with more than one point.\par

To prove~\eqref{cont}, by contradicion, we can assume without loss of generality that \(Q=\u P=(0,0)\) and that \(Du(Q)=\id\). Using that \(Q\) is a Lebesgue point of the gradient it is then well known that the maps
\[
u_r(y)=\frac{1}{r} u(ry)
\]
converge in \(W^{1,1}_{\rm loc}(\R^2)\) to \(u_\infty={\rm Id}\). By the one dimensional Sobolev embedding and a simple slicing argument, we can find a radius \(R\in (1/2,1)\) and a subsequence \(r_j\searrow 0\) such that 
\[
\|u_{r_j}-\id\|_{L^\infty(\partial B(0,R))}\to 0
\] 
Choosing \(j\) big enough and scaling back, we deduce~\eqref{cont}. 

Let us now show~\eqref{tan}; we can assume without loss of generality that $Q=\u P=(0,0)$ and that $Du(0)={\rm e}_1\otimes {\rm e}_1$, so that $\nu={\rm e}_2$. Let $\eps>0$; arguing as before, thanks to the $1$-dimensional Sobolev embedding we obtain that, for almost every $r\ll 1$,
\[
|u(x) - (x_1,0)| < \eps |x|
\]
for every $x$ with $|x|=r$. This means that, for almost every $r\ll 1$, one has
\[
\big\{\theta\in \S^1:\, r\theta \in u^{-1}(\u P)\big\} \subseteq \Big(\,\frac \pi 2 - 2\eps,\frac \pi 2+2\eps\,\Big)\cup
\Big(\,-\frac \pi 2 - 2\eps,-\frac \pi 2+2\eps\,\Big)\,.
\]
In the very same way, for almost every $|y|\ll 1$ we obtain that
\[
\big\{\theta\in \S^1:\, (y/\tan\theta,y) \in u^{-1}(\u P)\big\} \subseteq \Big(\,\frac \pi 2 - 2\eps,\frac \pi 2+2\eps\,\Big)\cup
\Big(\,-\frac \pi 2 - 2\eps,-\frac \pi 2+2\eps\,\Big)\,.
\]
The last two inclusions, together with the fact that $u^{-1}(\u P)$ is connected, immediately imply the validity of~(\ref{tan}).
\end{proof}

\begin{defin}[Connectibility]\label{connectibility}
Let $\u P\in{\rm Int}(\u\Q)$, let $\sigma:[0,1]\to\Q$ be compatible with $\u P$, and let $A$ and $B$ be any two points of $u^{-1}(\u P)\cap \sigma$. We say that \emph{$A$ and $B$ are left-connectible} if there is a connected component of $u^{-1}(\u P)\cap \sigma^-$ containing both $A$ and $B$. The definition of the right-connectibility is the same one, replacing $\sigma^-$ with $\sigma^+$.
\end{defin}

The following characterization of the connectibility will be useful in the sequel. Here, and in the rest of the section, whenever two points $A$ and $B$ belong to a same given curve, we write $\arc{AB}$ to denote the part of the curve between them, and we call it ``arc between $A$ and $B$''.

\begin{lemma}\label{daichevale}
Let $\u P\in {\rm Int}(\u \Q)$, let $\sigma:[0,1]\to\Q$ be compatible with $\u P$, and let $A$ and $B$ be any two points of $u^{-1}(\u P)\cap\sigma$. Then, $A$ and $B$ are left-connectible if and only if, for every two points $C,\, D$ in $\sigma\setminus \{A,\,B\}$, exactly one of those in the arc $\arc{AB}\subseteq\sigma$, it is impossible to find a path in $\sigma^-\setminus u^{-1}(\u P)$ connecting $C$ and $D$.
\end{lemma}
\begin{proof}
Let us first assume the existence of two points $C,\, D\in\sigma\setminus\{A,\,B\}$, only one of which in the arc $\arc{AB}$, and of a path $\tau_1\subseteq \sigma^-\setminus u^{-1}(\u P)$ connecting $C$ and $D$. We have to show that $A$ and $B$ are not left-connectible. Let $\tau_2$ be any arc connecting $C$ and $D$ in $\sigma^+$, in particular we can assume that $\tau_2\setminus \{C,\,D\}$ is contained in the internal part of $\sigma^+$. Putting together $\tau_1$ and $\tau_2$ (we can assume without loss of generality that they are both injective), we have a closed curve which does not intersect $u^{-1}(\u P)\cap \sigma^-$, and by construction exactly one between $A$ and $B$ belongs to the internal part of this curve. As a consequence, $A$ and $B$ do not belong to the same connected component of $u^{-1}(\u P)\cap\sigma^-$.\par

Conversely, assume that $A$ and $B$ are not left-connectible: we have to find two points $C$ and $D$ in $\sigma\setminus\{A,\,B\}$, only one of which in $\arc{AB}$, and a path connecting $C$ and $D$ in $\sigma^-\setminus u^{-1}(\u P)$. For the sake of clarity, and without loss of generality, we can think that $\sigma$ is a vertical segment. First of all notice that, since $u^{-1}(\u P)$ is connected, then every connected component of $u^{-1}(\u P)\cap\sigma^-$ must intersect some of the finitely many points of $u^{-1}(\u P)\cap\sigma$. Hence, there are finitely many connected components of the set $u^{-1}(\u P)\cap\sigma$, and since $u^{-1}(\u P)$ is closed then every component is closed, and any two are a strictly positive distance apart. Let then $\Gamma_A$ and $\Gamma_B$ be the connected components of $u^{-1}(\u P)\cap \sigma^-$ containing $A$ and $B$ respectively.
\begin{figure}[thbp]
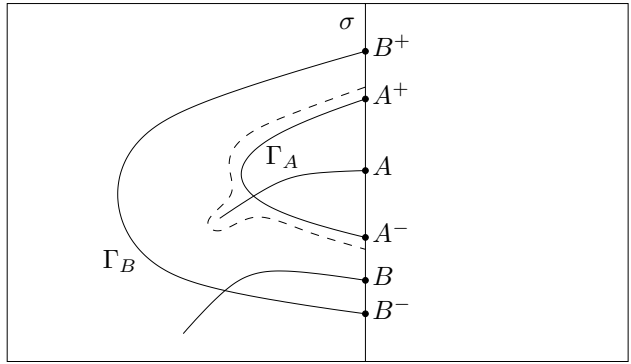
\caption{Construction for the proof of Lemma~\ref{daichevale}; the dashed curve is $\tilde\theta$.}\label{Fig:2}
\end{figure}
Let us call $A^-$ and $A^+$ the first and the last point of $\sigma$ which belong to $\Gamma_A$ (they could both coincide with $A$), and $B^-$ and $B^+$ those in $\Gamma_B$: a possible situation is depicted in Figure~\ref{Fig:2}. Notice that the intersection between the two curves $\arc{A^-A^+}$ and $\arc{B^-B^+}$ must be either one of the two curves (as in the figure), or empty. Indeed, otherwise $\Gamma_A$ and $\Gamma_B$ could clearly not be both connected and with no intersection.\par

As a consequence, without loss of generality we can assume that, as in the figure, the point $B$ is not contained in the arc $\arc{A^-A^+}$. Let us then apply Lemma~\ref{sard} to $\Gamma_A$, with $\eps$ smaller than the distance between $\Gamma_A$ and any other connected component of $u^{-1}(\u P)\cap \sigma^-$, so to find a closed, injective curve $\theta$, which has to intersect both $\sigma^-$ and $\sigma^+$ by construction. Let $P$ be a point of $\theta$, in the interior of $\sigma^-$, with ordinate higher than that of any point of $\Gamma_A$. Then, let $\tilde\theta$ be the shortest subpath of $\theta$ containing $P$ and having both endpoints on $\sigma$, and call $C$ and $D$ its endpoints, being $C$ below $D$; a simple geometric argument, using the fact that $\Gamma_A\subseteq \sigma^-$ is connected, ensures that $C$ is below $A^-$ and $D$ is above $A^+$, however the distance between $C$ and $A^-$, as well as that between $D$ and $A^+$, is less than $\eps$. As a consequence, $C$ belongs to $\arc{AB}$ while $D$ does not, and $\tilde\theta$ is a path connecting $C$ and $D$ in $\sigma^-\setminus u^{-1}(\u P)$, since $\tilde\theta$ does not intersect $\Gamma_A$ by construction and it does not intersect any other component of $\sigma^-\cap u^{-1}(P)$, since they have all distance larger than $\eps$ from $\Gamma_A$, while each point of $\tilde\theta$ has distance less than $\sigma$ from it.
\end{proof}

The first use of the definition of connectibility is to decide which of the two possibilities in Lemma~\ref{inclsegm} holds.

\begin{lemma}\label{onlytwo}
Let $\u P\in {\rm Int}(\u \Q)$, let $\sigma:[0,1]\to\Q$ be compatible with $\u P$, let $A$ and $B$ be two connectible points of $u^{-1}(\u P)\cap\sigma$, and let $A^-A^+$ and $B^-B^+$ be two short segments around $A$ and $B$, contained in $\sigma$ and which do not contain other points of $u^{-1}(\u P)$ apart from $A$ and $B$. Then, there exists a curve $\gamma:[0,1]\to\Q$ compatible with $\u P$, which intersects $u^{-1}(\u P)$ exactly at $A$ and $B$, and which contains the segments $A^-A^+$ and $B^-B^+$. 
\end{lemma}
\begin{proof}
By Lemma~\ref{inclsegm}, we know that the conclusion is true, up to replace the segment $B^-B^+$ with the segment $B^+B^-$. Let us then assume by contradiction the existence of a curve $\gamma$ as in the claim, but with the wrongly oriented segment $B^+B^-$. In particular, the curve $\gamma$ is done by a path $\gamma_0$ between $\partial\Q$ and $A^-$, then the segment $A^-A^+$, then a path $\gamma_1$ between $A^+$ and $B^+$, then the segment $B^+B^-$, and finally a last path $\gamma_2$ between $B^-$ and $\partial\Q$. Let us assume, without loss of generality, that $A$ and $B$ are left-connectible, the analogous argument works if they are right-connectible. Let then $\Gamma$ be the connected component of $u^{-1}(\u P)\cap\sigma^-$ containing both $A$ and $B$. Since it is a closed, connected set which does not disconnect $\Q$, by Lemma~\ref{1or2inter} its union $\Gamma^+$ with the segment $AA^+$, the curve $\gamma_1$ and the segment $B^+B$ disconnects $\Q$ in two connected components $\Q_{\rm int}$ and $\Q_{\rm ext}$, where we call $\Q_{\rm ext}$ the one containing $\partial\Q$. Since the curve $\gamma_0$ does not intersect $\Gamma^+$, we deduce that $A^-\in \Q_{\rm ext}$, and since $\Gamma\subseteq \sigma^-$ we deduce that the right part of the curve $\gamma_1$ (that is, the points arbitrarily close to $\gamma_1$ on the right part with respect to the vector $\gamma_1'$) is done by points in $\Q_{\rm ext}$. In particular, the point $B^-$ belongs to $\Q_{\rm int}$, which is against the existence of the curve $\gamma_2$ connecting $B^-$ with $\partial\Q$ without intersecting $\Gamma^+$.
\end{proof}

The following corollary is immediate.
\begin{corol}\label{connopp}
Let $\u P\in {\rm Int}(\u \Q)$, let $\sigma:[0,1]\to\Q$ be compatible with $\u P$, let $A$ and $B$ be two connectible points of $u^{-1}(\u P)\cap\sigma$. Then, the vectors $D(u\circ\sigma)(A)$ and $D(u\circ\sigma)(B)$ are parallel and with opposite directions.
\end{corol}
\begin{proof}
Let $\gamma$ be a curve given by Lemma~\ref{onlytwo}, and notice that since both $\sigma$ and $\gamma$ contain the segments $A^-A^+$ and $B^-B^+$ then
\begin{align*}
D(u\circ \gamma)\big(\gamma^{-1}(A)\big) = D(u\circ \sigma)\big(\sigma^{-1}(A)\big)\,, &&
D(u\circ \gamma)\big(\gamma^{-1}(B)\big) = D(u\circ \sigma)\big(\sigma^{-1}(B)\big)\,.
\end{align*}
Moreover, both the vectors are non-zero by definition of compatible curves, so the thesis follows by applying Lemma~\ref{opposite} to the curve $\gamma$.
\end{proof}

\begin{corol}\label{conncomp<2}
Let $\u P\in {\rm Int}(\u \Q)$, let $\sigma:[0,1]\to\Q$ be compatible with $\u P$. Then, every connected component of $\sigma^-\cap u^{-1}(\u P)$ contains at most two points in $\sigma$. Hence, any point of $\sigma\cap u^{-1}(\u P)$ is left-connectible with at most one other point.
\end{corol}
\begin{proof}
Assume that $A,\, B$ and $C$ are three distinct points in $u^{-1}(\u P)\cap\sigma$, contained in the same connected component of $\sigma^-\cap u^{-1}(\u P)$; then, any two of these three points are left-connectible. As a consequence, the derivatives of $u$ in the direction of $\sigma$ at the points $A,\, B$ and $C$ are three non-zero vectors, and by Corollary~\ref{connopp} any two of them have opposite orientation, which is clearly absurd.
\end{proof}

\begin{lemma}\label{nonboth}
Let $\u P\in {\rm Int}(\u \Q)$, let $\sigma:[0,1]\to\Q$ be compatible with $\u P$; then, two points in $\sigma\cap u^{-1}(\u P)$ cannot be both left-connectible and right-connectible.
\end{lemma}
\begin{proof}
Assume that $A$ and $B$ are both left-connectible and right-connectible. Let then $C,\, D\in\sigma\setminus u^{-1}(\u P)$ be any two points, exactly one of which in the arc $\arc{AB}$. Since $u^{-1}(\u P)$ does not disconnect $\Q$, there exists a path connecting $C$ and $D$ in $\Q\setminus u^{-1}(\u P)$; without loss of generality, we can assume that this path intersects $\sigma$ in finitely many points $P_1=C,\, P_2 ,\, \dots \,,\, P_N=D$. The subpath between any $P_i$ and $P_{i+1}$ is then either entirely in $\sigma^-$ or entirely in $\sigma^+$. By Lemma~\ref{daichevale}, keeping in mind that $A$ and $B$ are both left- and right-connectible, we deduce that $P_{i+1}$ belongs to the arc $\arc{AB}$ if and only if so does $P_i$. By recursion, $D=P_N$ belongs to $\arc{AB}$ if and only if so does $C=P_1$, which is a contradiction.
\end{proof}

\begin{lemma}\label{order}
Let $\u P\in {\rm Int}(\u \Q)$, let $\sigma:[0,1]\to\Q$ be compatible with $\u P$. Then, we can write $\sigma\cap u^{-1}(\u P)=\{P_1,\, P_2,\, \dots \, , \, P_N\}$, where each $P_i$ is either left- or right-connectible with $P_{i+1}$, and $P_1$ and $P_N$ are not connectible.
\end{lemma}
\begin{proof}
We can assume without loss of generality that $u^{-1}(\u P)\cap\sigma^-$ contains at least two points, otherwise the claim is emptily true. Let us consider the connected components of $u^{-1}(\u P)\cap\sigma^-$ and of $u^{-1}(\u P)\cap\sigma^+$; since $u^{-1}(\u P)$ is connected, each of these connected components must have at least a point in $\sigma$. Moreover, any two of these connected components have strictly positive distance unless they have a common point in $\sigma$. Again since $u^{-1}(\u P)$ is connected, we deduce that any connected component can reach any other one in finitely many ``steps'', where any step means passing from a connected component to another one which has a common point in $\sigma$. In other words, for any two points $A$ and $B$ in $u^{-1}(\u P)\cap\sigma$, there is a finite ``chain'' of points $Q_1=A,\, Q_2,\, \dots\,, \, Q_M=B$, where every $Q_i$ is left- or right-connectible with $Q_{i+1}$.\par

Take now any point in $u^{-1}(\u P)\cap\sigma$: for what we just said, it must be left- or right-connectible with some other point, but by Corollary~\ref{conncomp<2} it can be left-connectible, as well as right-connectible, with at most one other point. Assume for a moment the existence of a point $P_1\in u^{-1}(\u P)\cap\sigma$ which is connectible with only one other point; let us call $P_1$ this point, and $P_2$ the point with which $P_1$ is connected, and let us assume that $P_1$ and $P_2$ are right-connectible (if they are left-connectible the obvious modification of the argument works). If $u^{-1}(\u P)\cap\sigma=\{ P_1,\, P_2\}$ we are done; otherwise, since $P_1$ is not connectible with any other point, $P_2$ must be connectible with some point $P_3$, in particular it must be left-connectible, because it is already right-connectible with $P_1$. Again, if we have taken all the points of $u^{-1}(\u P)\cap\sigma$ we are done, otherwise $P_3$ must be right-connectible with some $P_4$ (since $P_1$ and $P_2$ already found all their connectibility points). With an obvious induction, we are done.\par

To conclude the proof, we must then exclude that all the points have a left-connectible and a right-connectible other point. If that was the case, we could still write $\sigma\cap u^{-1}(\u P)=\{P_1,\, P_2,\, \dots \, , \, P_N\}$, where every $P_i$ is left- (resp., right-)connectible with $P_{i+1}$ if $i$ is even (resp., odd); moreover, $N$ would be surely even and $P_N$ and $P_1$ would be left-connectible. For any even (resp., odd) number $1\leq i<N$, let us call $\Gamma_i$ the connected component of $u^{-1}(\u P)\cap \sigma^-$ (resp., $u^{-1}(\u P)\cap \sigma^+$) which contains $P_i$ and $P_{i+1}$, while $\Gamma_N$ is the connected component of $u^{-1}(\u P)\cap \sigma^-$ containing $P_N$ and $P_1$. Applying Lemma~\ref{onlytwo} to $P_1$ and $P_2$, we find a compatible curve $\gamma$, starting and ending in $\partial\Q$, which contains exactly $P_1$ and $P_2$ in $u^{-1}(\u P)$, and which coincides with $\sigma$ in two short segments around $P_1$ and around $P_2$. Since $\Gamma_1$ is contained in $\sigma^+$ and intersects $\gamma$ only at $P_1$ and $P_2$, we deduce that $\Gamma_1$ is also contained in $\gamma^+$; then, $P_1$ and $P_2$ are right-connectible also with respect to the curve $\gamma$. Similarly, since $\Gamma_2$ and $\Gamma_N$ are contained in $\sigma^-$, we deduce that they are also contained in $\gamma^-$. Let us now consider any $\Gamma_i$ with $3\leq i \leq N-1$: it has no intersection with $\gamma$, hence it is entirely contained in the interior of $\gamma^-$ or in the interior of $\gamma^+$. However, since $\Gamma_3$ has one point in common with $\Gamma_2$ (that is, $P_3$), we deduce that $\Gamma_3$ is entirely in the interior of $\gamma^-$, and by obvious recursion the same holds true for every $\Gamma_i$ with $4\leq i \leq N-1$. The union of all the sets $\Gamma_i$ for $i\neq 1$ is then a connected set contained in $\gamma^-$: this means that $P_1$ and $P_2$ are also left-connectible with respect to $\gamma$. This is a contradiction with Lemma~\ref{nonboth}, so we conclude.
\end{proof}
Notice that the order of the points $P_i$ given by Lemma~\ref{order} is not necessarily the order in which points are met by $\sigma$, see Figure~\ref{Fig:3} left for an example, where again $\sigma$ is depicted as a vertical segment just for the sake of clarity. Notice also that, if $P_1$ and $P_2$ are right-connectible, as in the figure, then every $P_i$ is right-connectible with $P_{i+1}$ when $i$ is odd and left-connectible when $i$ is even. The opposite clearly holds if $P_1$ and $P_2$ are left-connectible.

\begin{figure}[thbp]
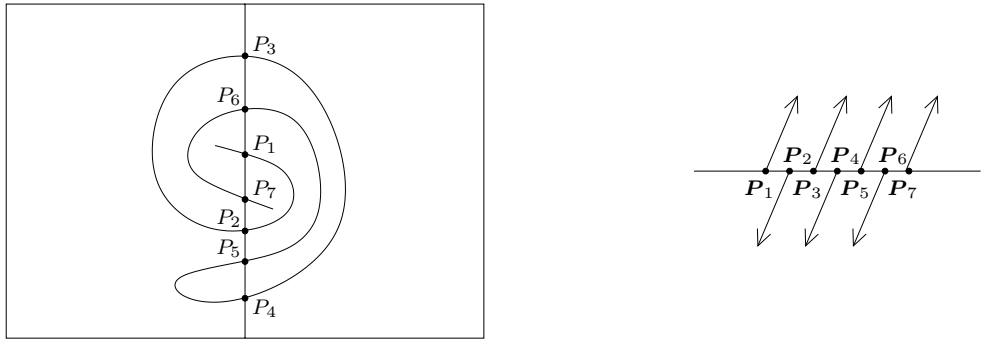
\caption{Left: the points $P_1,\, P_2,\, \dots \, ,\, P_N$ in Lemma~\ref{order}. Right: the corresponding points $\u P_1,\, \u P_2,\, \dots\,,\, \u P_N$ in Definition~\ref{whereupts}.}\label{Fig:3}
\end{figure}

Notice that, by Lemma~\ref{order} and Corollary~\ref{connopp}, the vectors $D(u\circ\sigma)(\sigma^{-1}(P_i))$ are all parallel. We can then conclude the section with a last definition.
\begin{defin}[Characteristic vector]\label{chaver}
Let $\u P\in{\rm Int}(\u\Q)$ and let $\sigma$ be a path compatible with $\u P$. We call \emph{characteristic vector of $\sigma$ at $\u P$} the vector $v\in\S^1$, defined up to a sign, such that $D(u\circ\sigma)(t_i)$ is parallel to $v$ for every $t_i \in (u\circ\sigma)^{-1}(\u P)$.
\end{defin}

\subsection{Separation of points on a curve}

This subsection is devoted to present a rule to ``separate'' the image of the points $u^{-1}(\u P)\cap\sigma$, where $\sigma$ is some given curve compatible with $\u P$. Let us be more precise.

\begin{defin}[Adapted curve]
Let $\widetilde\G\subseteq\u\Q$ be a grid as in~(\ref{defwidegxnyn}), and let $\sigma:[0,1]\to\Q\setminus\SS_u$ be an admissible curve with $\sigma(0),\,\sigma(1)\in\partial\Q$ and $\sigma(t)\in{\rm Int}(\Q)$ for $0<t<1$. We say that \emph{$\sigma$ is adapted to $\widetilde\G$} if $u(\sigma)\cap\widetilde\G$ is a finite set not containing any vertex of $\widetilde\G$, and for any $\u P\in u(\sigma)\cap\widetilde\G$ the curve $\sigma$ is compatible with $\u P$, and the characteristic vector of $\sigma$ at $\u P$, according to Definition~\ref{chaver}, is not parallel to the side of $\widetilde\G$ containing $\u P$. In this case, a number $\xi>0$ will be said \emph{small for $\sigma$} if it is much smaller than the distance between any two elements of $u(\sigma)\cap\widetilde\G$, as well as between any such element and any vertex of $\widetilde\G$.
\end{defin}

\begin{remark}\label{obvrem}
An obvious consequence of the above definition, together with Definitions~\ref{goodstarting} and~\ref{goodarrival}, is the following. If $\widetilde\G$ is a good arrival grid associated with the good starting grid $\G$, then every path contained in $\G$ is adapted to $\widetilde\G$.
\end{remark}

Our aim is to define a modification of $u$ on $\sigma$ which is piecewise linear an injective. We start by defining the images of the finitely many points of $\sigma\cap u^{-1}(\widetilde\G)$.
\begin{defin}\label{whereupts}
Let $\widetilde\G\subseteq\u\Q$ be a grid, let $\sigma$ be a curve adapted to $\widetilde\G$, and let $\xi$ be small for $\sigma$. Consider any point $\u P\in u(\sigma)\cap\widetilde\G$: being $\sigma$ compatible with $\u P$, we write $\sigma\cap u^{-1}(\u P)=\{P_1,\, P_2 ,\, \dots \, ,\, P_N\}$ as in Lemma~\ref{order}. We let then $\u P_1,\, \u P_2,\, \dots \,,\, \u P_N$ be $N$ distinct points on $\widetilde\G$, all contained within a distance $\xi$ from $\u P$ (hence they are all in the same side of $\widetilde\G$ as $\u P$), and ordered as follows. Assume first that $P_1$ and $P_2$ are right-connected. If the side containing $\u P$ is horizontal, then the points $\u P_i$ are ordered from left to right if $D(u\circ\sigma)(\sigma^{-1}(P_1))$ has a positive vertical component (as in Figure~\ref{Fig:3} right), and from right to left otherwise. Similarly, if the side of $\u P$ is vertical, the points $\u P_i$ are ordered from top to bottom if $D(u\circ\sigma)(\sigma^{-1}(P_1))$ has a positive horizontal component, and from bottom to top otherwise. If, instead, $P_1$ and $P_2$ are left-connected, then the orientation of the points $\u P_i$ is done in the opposite way.
\end{defin}

It is important to understand the meaning of the orientation of the points $\u P_i$ just defined, again a look to Figure~\ref{Fig:3} right can also help. Roughly speaking, if the point $P_i$ is right-connected with the point $P_{i+1}$, then an observer sitting at $\u P_i$, looking in the direction of the curve $u\circ \sigma$ at $\sigma^{-1}(P_i)$, will have the point $\u P_{i+1}$ at his right. Notice that, at the same time, also $P_{i+1}$ is right-connected with $P_i$; and in fact, the point $\u P_i$ is also at the right of the point $\u P_{i+1}$, from the point of view of $P_{i+1}$, at which the direction of $u\circ \sigma$ is opposite.

We are now ready to define the modification of $u$ on $\sigma$. Before doing so note that, if \(\xi\) is  taken sufficiently small, then the following simple geometric fact about the generalised segments introduced in Definiton~\ref{gensegm} holds: let \(\u \RR\) be a rectangle of the grid \(\widetilde \G\) and  $\u A,\, \u B,\, \u C$ and $\u D$ be four distinct ordered points on a same side of $\partial\u\RR$, then  $[\u{AB}]$ and $[\u{CD}]$ have empty intersection, as well as $[\u{AD}]$ and $[\u{BC}]$, while $[\u{AC}]$ and $[\u{BD}]$ have exactly an intersection point. 

\begin{defin}[The function $v_\sigma$]\label{defvsigma}
Let $\widetilde\G\subseteq\u\Q$ be a grid, let $\sigma:[0,1]\to\Q$ be adapted to $\widetilde\G$, and let $\xi$ be small for $\sigma$. The function $v_\sigma:\sigma\to\u\Q$ is defined as follows. For any point $\u P\in u(\sigma)\cap\widetilde\G$, we define the points $P_i$ as in Lemma~\ref{order} and the points $\u P_i$ as in Definition~\ref{whereupts}, and we set $v_\sigma(P_i)=\u P_i$ for any $i$; we also set $v_\sigma=u$ at $\sigma(0)$ and $\sigma(1)$. We have then defined $v_\sigma$ at all the points of $\sigma\cap u^{-1}(\widetilde\G)$, and at $\sigma(0)$ and $\sigma(1)$. Let now $A,\,B$ be any two consecutive points on $\sigma$ on which $v_\sigma$ is already defined. Notice that, by definition, the function $u$ in the open curve $\arc{AB}$ is contained in the interior of a square $\u\RR$ of $\widetilde\G$, and $u(A)$ and $u(B)$, as well as $v_\sigma(A)$ and $v_\sigma(B)$, lie on $\partial\u\RR$. We extend then $v_\sigma$ on the curve $\arc{AB}$ as the generalized segment $[\u{AB}]$, being $\u A=v_\sigma(s)$ and $\u B=v_\sigma(t)$, parametrized with constant speed.
\end{defin}

The main goal of this section is to show that the function $v_\sigma$, which is uniformly very close to $u$, is injective.

\begin{prop}\label{injective}
Let $\widetilde\G,\, \sigma$ and $v_\sigma$ be as in Definition~\ref{defvsigma}. Then, the piecewise linear function $v_\sigma:\sigma\to\u\Q$ is injective.
\end{prop}

Since the proof of this result is quite involved, we divide it in several intermediate results. First of all, we show that the order given by Lemma~\ref{order} and Definition~\ref{whereupts} is ``well-defined'', that is, if we take two different adapted curves containing the same two points in $u^{-1}(\u P)$ for some $\u P\in\widetilde\G$, the order between the images of these two points in the two corresponding functions in Definition~\ref{defvsigma} is the same.
\begin{lemma}\label{well-defined}
Let $\widetilde\G\subseteq\u\Q$ be a grid as in~(\ref{defwidegxnyn}), let $\u P\in\widetilde\G$, let $A,\, B\in u^{-1}(\u P)$, and let $\sigma$ and $\tilde\sigma$ be two curves adapted to $\widetilde\G$, both containing $A$ and $B$. Then, the order of the points $v_\sigma(A)$ and $v_\sigma(B)$ in the segment of $\widetilde\G$ containing $\u P$ is the same as that of $v_{\tilde\sigma}(A)$ and $v_{\tilde\sigma}(B)$, where $v_\sigma$ and $v_{\tilde\sigma}$ are as in Definition~\ref{defvsigma}.
\end{lemma}
\begin{proof}
Since $\sigma$ is adapted to $\widetilde\G$, it is compatible with $\u P$ and we can apply Lemma~\ref{order} to $\sigma$ so to write $u^{-1}(\u P)=\{ P_1,\, P_2,\, \dots \,,\, P_N\}$ . Up to exchange $A$ and $B$, there are $1\leq i < j \leq N$ such that $A=P_i$ and $B=P_j$. Just to fix the ideas, we assume that $P_i$ is right connected with $P_{i+1}$ (of course, the symmetric argument works in the opposite case). Then, by Definition~\ref{whereupts} we know that $v_\sigma(B)$ is on the ``right'' of $v_\sigma(A)$, where the right, in the segment of $\widetilde\G$ containing $\u P$, is intended with respect to $\omega=D(u\circ \sigma)(\sigma^{-1}(A))$. For every $1\leq l\leq N-1$ such that $i-l$ is even (resp., odd), let us call $\Gamma_l$ the connected component of $u^{-1}(\u P)\cap \sigma^+$ (resp., $u^{-1}(\u P)\cap\sigma^-$) which contains $P_l$ and $P_{l+1}$.\par

Let now $\gamma:[0,1]\to\Q$ be a curve, compatible with $\u P$, which intersects $u^{-1}(\u P)$ exactly at $A$ and $B$ (such a curve surely exists by Lemma~\ref{inclsegm}). Up to reverse the orientation of $\gamma$, we can assume that $\Gamma_i\subseteq\gamma^+$. But then, the set $\mathcal U=\Gamma_i \cup \Gamma_{i+1} \cup \cdots \cup \Gamma_{j-1}$ is a connected subset of $u^{-1}(\u P)$, which intersects $\gamma$ exactly at $A$ and $B$; since $\Gamma_i$ is contained in $\gamma^+$, then the same is true for the whole $\mathcal U$, which proves that $A$ and $B$ are right-connectible with respect to $\gamma$. Then, defining $v_\gamma$ according to Definition~\ref{defvsigma}, we get that $v_\gamma(B)$ is on the right of $v_\gamma(A)$: notice that, this time, the ``right'' on the segment of $\widetilde\G$ containing $\u P$ must be intended with respect to $\omega'=D(u\circ\gamma)(\gamma^{-1}(A))$, not with respect to $\omega$. Nevertheless, we can readily observe that the two things are equivalent; in fact, we can show that the non-zero vectors $\omega$ and $\omega'$ are parallel and with the same direction. Indeed, by Lemma~\ref{lm:direction}, there is a direction $\nu\in\S^1$ such that $Du(A)\nu=0$, and points of $u^{-1}(\u P)$ close to $A$ have direction, with respect to $A$, converging to $\pm\nu$. As a consequence, calling $\nu^\perp$ a vector orthogonal to $\nu$, and calling $\alpha=Du(A)\nu^\perp\neq 0$, we have that
\begin{align*}
\omega=D(u\circ \sigma)(\sigma^{-1}(A))=Du(A) \sigma'(\sigma^{-1}(A)) = \alpha \sigma'(\sigma^{-1}(A))\cdot \nu^\perp\,, &&
\omega'=\alpha \gamma'(\gamma^{-1}(A))\cdot \nu^\perp\,.
\end{align*}
And in turn, the scalar products of $\sigma'(\sigma^{-1}(A))$ and $\gamma'(\gamma^{-1}(A))$ with $\nu^\perp$ have the same sign, since $\Gamma'_i$ is both in $\sigma^+$ and in $\gamma^+$.\par
Summarizing, we have proved that the order of the images of $A$ and $B$ under $v_\sigma$ is the same as under $v_\gamma$, where $\gamma$ is any curve, compatible with $\u P$, intersecting $u^{-1}(\u P)$ exactly at $A$ and $B$. But then, the very same argument shows that also the order of the images of $A$ and $B$ under $v_{\tilde\sigma}$ is the same as under $v_\gamma$, hence in turn also as under $v_\sigma$. The proof is then concluded.
\end{proof}

We can now show that the orientation of points on the boundary of squares of $\widetilde\G$ is consistent, in the following sense.

\begin{lemma}\label{4distinct}
Let $\widetilde\G\subseteq\u\Q$ be a grid, let $\sigma:[0,1]\to\Q$ be adapted to $\widetilde\G$, let $\u\RR$ be a rectangle of the grid $\widetilde\G$, and let $A,\, B,\, C,\, D$ be four distinct points on $\sigma$ such that the arcs $\arc{AB}$ and $\arc{CD}$ are disjoint, and their images under $u$ are in ${\rm Int}(\u\RR)$ except the four points $\u A=u(A),\, \u B=u(B),\, \u C=u(C),\, \u D=u(D)$ on $\partial\u\RR$. If $\u A\neq \u B$, then the points $\u C$ and $\u D$ cannot belong to the two different open paths in which $\partial\u\RR$ is divided by $\u A$ and $\u B$.
\end{lemma}
\begin{proof}
Since $\sigma$ is a segment near the four points $A,\, B,\, C$ and $D$, then it contains four small disjoint segments $A_{\rm int}A_{\rm ext}$, $B_{\rm int}B_{\rm ext}$, $C_{\rm int}C_{\rm ext}$ and $D_{\rm int}D_{\rm ext}$ around them, and the image of these segments under $u$ crosses $\partial\u\RR$ at $\u A,\, \u B,\, \u C$ and $\u D$ respectively, being $u(X_{\rm int})$ inside $\u\RR$ and $u(X_{\rm ext})$ outside $\u\RR$ for each $X\in \{A,\,B,\,C,\,D\}$. By assumption, the set $\mathcal D=u^{-1}(\u\RR)$ is a closed, connected subset of $\Q$ which does not disconnect $\Q$, hence we can find two injective and continuous paths $\gamma_1,\,\gamma_2:[0,1]\to\Q\setminus \mathcal D$ connecting $B_{\rm ext}$ with $C_{\rm ext}$, and $D_{\rm ext}$ with $A_{\rm ext}$ respectively. Arguing as in Lemma~\ref{inclsegm}, we can assume that $\gamma_1$ and $\gamma_2$ are disjoint, up to exchange $C$ with $D$; we can clearly also assume without loss of generality that $\gamma_1$ and $\gamma_2$ do not intersect the segments $X_{\rm ext}X$ for $X\in\{A,\,B,\,C,\,D\}$, nor the set $\SS_u$ (keep in mind that $\SS_u$ is $\H^1$-negligible). Since the arcs $\arc{AB}$ and $\arc{CD}$ are disjoint, and by assumption they belong to $\mathcal D$, then the closed curve $\gamma:\S^1\to\Q\setminus\SS_u$ obtained joining $\arc{A_{\rm ext}B_{\rm ext}}$ with $\gamma_1$, then $\arc{C_{\rm ext}D_{\rm ext}}$ and then $\gamma_2$ is injective. As a consequence, by Lemma~\ref{ThCINV}, for every point of $\u\Q\setminus u(\gamma)$ the degree with respect to $u(\gamma)$ must be either $0$ or $1$.\par

Notice now that $u(\gamma)\cap\partial\u\RR$ consists precisely of the points $\u A,\,\u B,\, \u C$ and $\u D$. Therefore, since $u\circ\gamma$ is exiting from $\partial\u\RR$ at $\u B$, we have that
\begin{equation}\label{B-+}
\deg(\u B^+,\, u(\gamma))=\deg(\u B^-,\, u(\gamma))-1\,,
\end{equation}
where $\u B^+$ and $\u B^-$ are two points on $\partial\u\RR$, very close to $\u B$, and in such a way that $\u B^+$ and $\u B^-$ are respectively on the right and of the left of $\u B$, with respect to the direction of $u\circ\gamma$ at $\u B$. Similarly, since $u\circ\gamma$ is exiting from $\partial\u\RR$ also at $\u D$, calling $\u D^+$ and $\u D^-$ two points on $\partial\u\RR$, close to $\u D$ and respectively at its right and at its left, we have that
\begin{equation}\label{D-+}
\deg(\u D^+,\, u(\gamma))=\deg(\u D^-,\, u(\gamma))-1\,.
\end{equation}
Let us now assume that $\u C$ and $\u D$ belong to the two different open paths in which $\partial\u\RR$ is divided by $\u A$ and $\u B$. As a consequence, there is a path $\theta$ in $\partial\u\RR$ connecting $\u B$ and $\u D$ without intersecting $\u A$ and $\u D$; then, $\theta$ must contain either both $\u B^+$ and $\u D^-$, or both $\u B^-$ and $\u D^+$. In the first case, the part of $\theta$ between $\u B^+$ and $\u D^-$ is entirely contained in $\u\Q\setminus u(\gamma)$, hence $\deg(\u B^+,\, u(\gamma))=\deg(\u D^-,\, u(\gamma))$, so by~(\ref{B-+}) and~(\ref{D-+}) we find $\deg(\u B^-,\, u(\gamma))=\deg(\u D^+,\, u(\gamma))+2$, which is absurd because the degrees can only be $0$ or $1$ (or only $0$ and $-1$) by the INV condition. In the second case, the analogous argument shows $\deg(\u B^+,\, u(\gamma))=\deg(\u D^-,\, u(\gamma))-2$, so again a contradiction.
\end{proof}

The above lemma says, roughly speaking, that the orientation of the points on the boundaries of the squares of $\widetilde\G$ is ``correct''; nevertheless, the lemma is useless if the images of some of the points coincide, which is in fact the case that we have to deal with more care. The next result can also treat the case when two points have the same image, if they are connectible.

\begin{lemma}\label{correctorientation}
Let $\widetilde\G\subseteq\u\Q$ be a grid, $\sigma:[0,1]\to\Q$ be adapted to $\widetilde\G$, $\u R\in\widetilde\G$, $P,\, Q\in u^{-1}(\u R)\cap\sigma$, and assume that $P$ and $Q$ are left-connectible. Since $D(u\circ\sigma)$ is opposite at $\sigma^{-1}(P)$ and $\sigma^{-1}(Q)$, let $\u\RR$ be the rectangle of $\widetilde\G$ in which $u\circ \sigma$ is entering at $\sigma^{-1}(P)$ and from which it is exiting at $\sigma^{-1}(Q)$, and let $P^+$ (resp., $Q^-$) be the first point on $\sigma$ after $P$ (resp., the last point before $Q$) at which $u\circ\sigma$ is exiting from $\u\RR$ (resp., entering in $\u\RR$). If the three points $\u P^+=u(P^+)$, $\u Q^-=u(Q^-)$ and $\u R$ on $\partial\u\RR$ are distinct, then the path on $\partial\u\RR$ connecting $\u R$ and $\u P^+$ without intersecting $\u Q^-$ starts at $\u R$ towards right, with respect to $D(u\circ \sigma)(\sigma^{-1}(P))$.
\end{lemma}
\begin{proof}
The geometrical meaning of the claim might seem a bit obscure, but it becomes quite evident with the help of a figure. Figure~\ref{Fig:4} (right) shows the rectangle $\u\RR$, with the points $\u R,\, \u P^+$ and $\u Q^-$. Notice that, since $u\circ\sigma$ is entering in $\u\RR$ at $\sigma^{-1}(P)$, then the vector $D(u\circ \sigma)(\sigma^{-1}(P))$ has a positive vertical component if, as in the figure, we assume just to fix the ideas that $\u R$ is on the bottom side of $\u\RR$. In the situation of the figure, the path connecting $\u R$ with $\u P^+$ on $\partial\u\RR$ without intersecting $\u Q^-$ starts at $\u R$ towards \emph{left}: then, the claim of the lemma says that the situation of the figure is impossible, because the order of the points $\u P^+$ and $\u Q^-$ should be the opposite. Notice that, since $P$ and $Q$ are left-connectible, then the function $v_\sigma$ of Definition~\ref{defvsigma} will have $v_\sigma(Q)$ on the left of $v_\sigma(P)$, then in the situation of the figure the map $v_\sigma$ would not be injective: this perfectly clarifies the importance of the present result. For the sake of clarity, we divide the proof in two steps.
\begin{figure}[thbp]
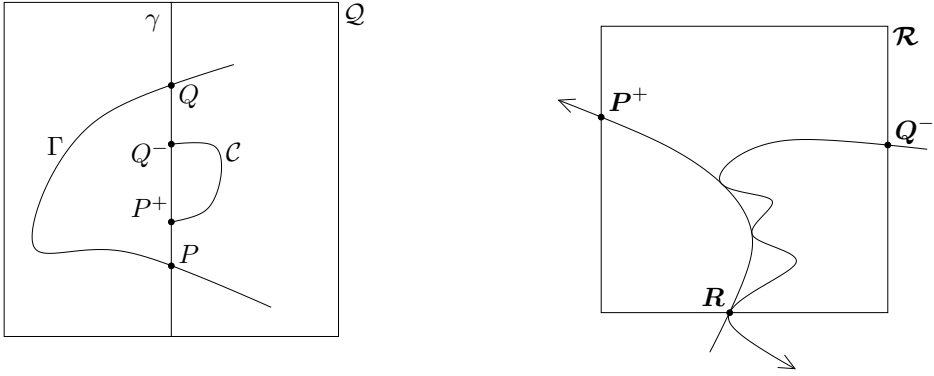
\caption{The situation in Lemma~\ref{correctorientation}.}\label{Fig:4}
\end{figure}
\step{I}{There is an admissible, counterclockwise curve $\gamma:\S^1\to\Q\setminus\SS_u$, containing the arcs $\arc{PP^+}$ and $\arc{Q^-Q}$ of $\sigma$, being a segment around $P,\, P^+,\, Q^-$ and $Q$, and for which $\gamma^{-1}(\u\RR)=\arc{PP^+}\cup\arc{Q^-Q}$.}
Since, by assumption, $u^{-1}(\u\RR)$ is a closed, connected set which does not disconnect $\Q$, as usual we can find two disjoint and admissible paths $\gamma_0$ and $\gamma_2$ which connect $\partial\Q$ with $P$ and $Q$ with $\partial\Q$ respectively, which are contained in ${\rm Int}(\Q)\setminus u^{-1}(\u\RR)$ except at their endpoints, which do not intersect $\SS_u$, and which coincide with the two segments of $\sigma$ near $P$ and $Q$. We claim that there exists also a path $\gamma_1$ connecting $P^+$ and $Q^-$, contained in $\Q\setminus u^{-1}(\u\RR)$ except at $P^+$ and $Q^-$, not intersecting $\gamma_0$ and $\gamma_2$, and coinciding with the segments of $\sigma$ near $P^+$ and $Q^-$. The claim will immediately conclude the step, being $\gamma$ simply the curve obtained putting together $\gamma_0$, $\arc{PP^+}$, $\gamma_1$, $\arc{Q^-Q}$, $\gamma_2$, and a part of $\partial\Q$ connecting the last point of $\gamma_2$ with the first point of $\gamma_0$, since by definition the arcs $\arc{PP^+}$ and $\arc{Q^-Q}$ are contained in $u^{-1}(\u\RR)$. Notice that there are two paths in $\partial\Q$ which connect the last point of $\gamma_2$ with the first one of $\gamma_0$, we have to choose the ``left'' one, so that $\gamma$ is percurred in the counterclockwise sense.\par

Let us now show the claim: if it were false, then arguing again as in Lemma~\ref{inclsegm} we could find two disjoint paths $\tilde\gamma_1$ and $\tilde\gamma_2$, connecting respectively $P^+$ and $Q$, and $Q^-$ and $\partial\Q$, being in $\Q\setminus (u^{-1}(\u\RR)\cup \gamma_0)$ except at $P^+,\, Q^-$ and $Q$. In this case, we could define $\tilde\gamma$ putting together $\gamma_0$, $\arc{PP^+}$, $\tilde\gamma_1$, $\arc{QQ^-}$ and $\tilde\gamma_2$. This path would intersect $u^{-1}(\u R)$ exactly at $P$ and $Q$, which would then be connectible for $\tilde\gamma$, so the vectors $D(u\circ \tilde\gamma)$ at $\tilde\gamma^{-1}(P)$ and $\tilde\gamma^{-1}(Q)$ would be parallel and with opposite directions by Lemma~\ref{connopp}. On the other hand, also the vectors $D(u\circ\sigma)$ at $\sigma^{-1}(P)$ and $\sigma^{-1}(Q)$ must be parallel and with opposite directions for the same reason, and this brings to a contradiction because the directions of $D(u\circ \tilde\gamma)(\tilde\gamma^{-1}(P))$ and $D(u\circ\sigma)(\sigma^{-1}(P))$ are the same, while the directions of $D(u\circ \tilde\gamma)(\tilde\gamma^{-1}(Q))$ and $D(u\circ\sigma)(\sigma^{-1}(Q))$ are opposite because $\tilde\gamma$ contains the curve $\arc{QQ^-}$ instead of $\arc{Q^-Q}$.

\step{II}{Conclusion.}
We will work with the curve $\gamma$ given by Step~I, instead of the original curve $\sigma$; this curve is depicted in Figure~\ref{Fig:4} (left): as usual, just for the sake of clarity, we draw $\gamma\cap{\rm Int}(\Q)$ as a vertical segment. Let us suppose by contradiction that the claim is false, hence the relative position of the points $\u P^+$ and $\u Q^-$ is as in Figure~\ref{Fig:4} (right). Then, we call $\u{\mathcal C}$ the closed arc of $\partial\u\RR$ connecting $\u P^+$ and $\u Q^-$ without containing $\u R$. By assumption, the set $\mathcal C=u^{-1}(\u{\mathcal C})$ is a closed, connected subset of $\Q$ which does not disconnect it. Moreover, by construction $\mathcal C$ intersects $\gamma$ precisely at $P^+$ and $Q^-$. Since $u\circ\gamma$ is exiting from $\u\RR$ at $P^+$ and entering in $\u\RR$ at $Q^-$, the points of $\u{\mathcal C}\setminus \{\u P^+,\, \u Q^-\}$ have all degree $0$ with respect to the curve $\gamma$, hence the set $\mathcal C$ is contained in $\gamma^+$. Summarizing, $\mathcal C$ is a connected set in $\gamma^+$ which connects $P^+$ and $Q^-$, while by assumption there is a connected component $\Gamma$ of $u^{-1}(\u R)\cap\gamma^-$ which connects $P$ and $Q$. Let us now consider $u^{-1}(\u\RR)$: this is a connected subset of $\Q$, which does not disconnect it, and by construction it contains $\Gamma\subseteq \gamma^-$, $\mathcal C\subseteq \gamma^+$, and the two arcs $\arc{PP^+}$ and $\arc{Q^-Q}$. Therefore, it must contain the whole arc $\arc{P^+Q^-}$ of $\gamma$, which is absurd since the curve $\gamma$ is outside $\u\RR$ between $P^+$ and $Q^-$.
\end{proof}

We are now ready to show Proposition~\ref{injective}.
\proofof{Proposition~\ref{injective}}
The function $v_\sigma$ is piecewise linear; in particular, each linear piece connects two points of the grid $\widetilde\G$ (with a small abouse of notation, we refer to each generalised segment as a ``linear piece'', even if it can be actually either one or two linear pieces). Moreover, by construction the points of $v_\sigma$ which are on $\widetilde\G$ are finitely many and disjoint. As a consequence, assuming by contradiction that $v_\sigma$ is not injective, there must exist two disjoint arcs $\arc{AB}$ and $\arc{CD}$ on $\sigma$ such that $v_\sigma(\arc{AB})$ and $v_\sigma(\arc{CD})$ have a non-empty intersection. In particular, there exists a rectangle $\u\RR$ of the grid $\widetilde\G$ which contains the whole linear pieces $v_\sigma(\arc{AB})$ and $v_\sigma(\arc{CD})$. Let us call $\u A=u(A)$, $\u B=u(B)$, $\u C=u(C)$ and $\u D=u(D)$: notice that these points do not necessarily coincide with $v_\sigma(A)$, $v_\sigma(B)$, $v_\sigma(C)$ and $v_\sigma(D)$, but they are very close to them.\par

Since $u^{-1}(\u\RR)$ is a closed, connected set which does not disconnect $\Q$, as already done several times we can find an admissible path $\gamma:[0,1]\to\Q$ which contains the arc $\arc{AB}$ and either the arc $\arc{CD}$ or the arc $\arc{DC}$, which coincides with $\sigma$ on four segments around the points $A,\, B,\, C$ and $D$, and such that $u^{-1}(\u\RR)$ consists precisely of the arcs $\arc{AB}$ and $\arc{CD}$ (or $\arc{DC}$). Notice that it is admissible to assume that $\gamma$ is also adapted to $\widetilde\G$; in fact, we are only interested in the behaviour of $\sigma$, or $\gamma$, around $\u\RR$. Thanks to Lemma~\ref{well-defined}, the two generalised segments $v_\sigma(\arc{AB})$ and $v_\sigma(\arc{CD})$ have a non-empty intersection if and only if the same happens to the two generalised segments $v_\gamma(\arc{AB})$ and $v_\gamma(\arc{CD})$ (and the fact whether $\gamma$ contains the arc $\arc{CD}$ or $\arc{DC}$ has no influence at all). Hence, it is not restrictive to assume directly that $u^{-1}(\u\RR)\cap\sigma$ consists precisely of the arcs $\arc{AB}$ and $\arc{CD}$, and we will do so. We will consider separately the possible cases of the relative positions of the points $\u A,\, \u B,\, \u C$ and $\u D$.

\case{I}{One has $\{ \u A,\, \u B\} \cap \{ \u C,\, \u D\} = \emptyset$.}
We start by considering the case in which each of the points $\u A$ and $\u B$ is different from both $\u C$ and $\u D$. If the four points $\u A,\, \u B,\, \u C$ and $\u D$ are different, then by Lemma~\ref{4distinct} there are two disjoint paths on $\partial\u\RR$ which connect $\u A$ with $\u B$, and $\u C$ and $\u D$ respectively; since the four points are distinct, by construction the same is true for the points $v_\sigma(A),\, v_\sigma(B),\, v_\sigma(C)$ and $v_\sigma(D)$. Keeping in mind Definition~\ref{gensegm}, we obtain that the two generalised segments $v_\sigma(\arc{AB})$ and $v_\sigma(\arc{CD})$ have no intersection, against the assumption. To conclude this case, we assume then that the four points are not distinct. In view of the assumption $\{ \u A,\, \u B\} \cap \{ \u C,\, \u D\} = \emptyset$, this means that either $\u A=\u B$, or $\u C=\u D$, or both. Suppose for instance that $\u A=\u B$; then, the generalised segment $v_\sigma(\arc{AB})$ is extremely close to the point $\u A=\u B$; instead, the points $\u C$ and $\u D$ are both different from $\u A$, hence the generalised segment $v_\sigma(\arc{CD})$ is away from the point $\u A$. This shows again that the two generalised segments are disjoint, against the assumption; hence, this case is concluded.\par

As a consequence, we can assume that $\{ \u A,\, \u B\} \cap \{ \u C,\, \u D\} \neq \emptyset$. Keep in mind that, since $\sigma$ intersects $\u\RR$ only in the two arcs $\arc{AB}$ and $\arc{CD}$, then the curve $u\circ\sigma$ enters in $\u\RR$ at $\u A$ and exits at $\u B$, then enters again at $\u C$ and exits at $\u D$; this shows that necessarily $\u A=\u D$ or $\u B=\u C$ (or both). In fact, the case $\u A=\u C\notin\{\u B,\, \u D\}$ would contradict Lemma~\ref{opposite}, as well as $\u B=\u D\notin \{\u A,\,\u D\}$. Therefore, from now on we will assume that $\u A=\u D$, being the case $\u B=\u C$ completely analogous.

\case{II}{One has $\u A=\u D$, and the points $\u B$ and $\u C$ are different, and different from $\u A$.}
In this case, the points $v_\sigma(A)$ and $v_\sigma(D)$ are very close to $\u A$ and to each other, while $v_\sigma(B)$ and $v_\sigma(C)$ are close to $\u B$ and $\u C$, so away from each other and from $\u A$. Since $u^{-1}(\u A)\cap\sigma$ consists precisely of $A$ and $D$, the two points must be connectible; let us assume, just to fix the ideas, that they are left-connectible. We have then to show that the path in $\u\RR$ connecting $\u A$ and $\u B$ without containing $\u C$ starts at $\u A$ towards right, with respect to the direction of $D(u\circ \sigma)$ at $\sigma^{-1}(A)$. And in fact, this is precisely given by Lemma~\ref{correctorientation}.

\case{III}{One has $\u A=\u D$, $\u C=\u B$ and $\u A\neq \u B$.}
By construction, $u^{-1}(\u A)\cap\sigma$ consists only of $A$ and $D$, while $u^{-1}(\u B)$ only of $B$ and $C$. As a consequence, $A$ and $D$ are connectible in $u^{-1}(\u A)$, as well as $B$ and $C$ in $u^{-1}(\u B)$; just to fix the ideas, assume that $A$ and $D$ are right-connectible. Notice that the non-emptiness of the intersection between the two generalised segments $v_\sigma(\arc{AB})$ and $v_\sigma(\arc{CD})$ implies that $B$ and $C$ are left-connectible. As a consequence, there is a connected set $\Gamma_1\subseteq u^{-1}(\u A)\cap \sigma^+$ which contains both $A$ and $D$, and a connected set $\Gamma_2\subseteq u^{-1}(\u B)\cap \sigma^-$ which contains both $B$ and $C$. Let us now consider the connected set $u^{-1}(\u\RR)\subseteq\Q$, which does not discnnect $\Q$: by construction, it contains $\Gamma_1$ and $\Gamma_2$, as well as both the arcs $\arc{AB}$ and $\arc{CD}$. As a consequence, it must also contain the whole arc $\arc{BC}$, which is absurd because $\sigma$ exits from $\u\RR$ at $\u B$.

\case{IV}{One has $\u A=\u D=\u B\neq \u C$, or $\u A=\u D=\u C\neq \u B$.}
Assume that $\u A=\u D=\u B\neq \u C$, the other case being fully analogous. The non-emptiness of the intersection between the two generalised segments $v_\sigma(\arc{AB})$ and $v_\sigma(\arc{CD})$ means that $v_\sigma(D)$ is between $v_\sigma(A)$ and $v_\sigma(B)$; since $u^{-1}(\u A)\cap\sigma$ consists precisely of $A,\, B$ and $D$, this means that $D$ is left connected, in $u^{-1}(\u A)$, with one between $A$ and $B$, and right connected with the other one. Again, this means that there exists a connected set $\Gamma_1\subseteq u^{-1}(\u A)\cap \sigma^-$ which contains $D$ and one between $A$ and $B$, and a connected set $\Gamma_2\subseteq u^{-1}(\u A)\cap\sigma^+$ which contains $D$ and the other one. Since $u^{-1}(\u\RR)$ does not disconnect $\Q$ and contains both $\Gamma_1$ and $\Gamma_2$, as well as the arc $\arc{AB}$, this means again that the arc $\arc{BC}$ is contained in $u^{-1}(\u\RR)$, a contradiction.

\case{V}{One has $\u A=\u D=\u B=\u C$.}
This final case is very similar to the last ones. The non-emptiness of $v_\sigma(\arc{AB})\cap v_\sigma(\arc{CD})$ implies, this time, that the two segments $v_\sigma(A) v_\sigma(B)$ and $v_\sigma(C)v_\sigma(D)$ are not disjoint neither contained one into the other. Hence, up to exchange the roles of the points, we can assume that $v_\sigma(D)$ is between $v_\sigma(A)$ and $v_\sigma(B)$, while $v_\sigma(C)$ is not between them. By definition, this means that $D$ is left-connectible with one between $A$ and $B$, and right-connectible with the other one: thus, we find a contradiction exactly as in Case~IV.
\end{proof}

\subsection{Proof of Theorem~\mref{NonSconn}}

This last subsection is devoted to show Theorem~\mref{NonSconn}, which now comes as a consequence of Proposition~\ref{injective}.

\proofof{Theorem~\mref{NonSconn}}
Let us call $\SS_u$ the set of the points at which $u$ is not continuous. Since $u$ is INV by Lemma~\ref{ThCINV}, then we know that $\SS_u$ is $\H^1$-negligible. Let then $\gamma_i$ be curves as in Definition~\ref{defnocross}. By a trivial induction argument, up to a diffeomorphisms of $\Q$ onto itself, coinciding with the identity on the boundary, we can assume that the union of the curves $\gamma_i$ is contained in some grid $\G=\G(K)$. As a consequence, to show the result it is enough to show that, if the grid $\G$ is contained in $\Q\setminus\SS_u$, then there exists an injective function $v:\G\to\u\Q$ such that $\|u-v\|_{L^\infty(\G)}<\eps$. Since $u$ is continuous on $\G$, it is not restrictive to assume that $\G$ is a good starting grid.\par

We can now apply Lemma~\ref{existencearrival} to obtain a good arrival grid $\widetilde\G$ with side-length $\eta<\eps/(2\sqrt 2)$; keep in mind that, as observed in Remark~\ref{obvrem}, any path contained in $\G$ is adapted to $\widetilde\G$. Since for any vertex $V_{ij}=(i/K,j/K)$ of the grid $\G$ we have that $u(V_{ij})$ is in the interior of some square of $\widetilde\G$, and since $u$ is continuous at $V_{ij}$, we can define an injective path $\sigma:[0,1]\to\Q$, adapted to $\widetilde\G$, as in Figure~\ref{pathsigma}. More precisely, we select very small balls $B_{ij}$ around the vertices $V_{ij}$, so that $u(B_{ij})$ is contained in the interior of the same square as $u(V_{ij})$, and the image of $\sigma$ consists of the whole grid $\G$ outside the union of the balls $B_{ij}$. Inside each ball, instead, $\sigma$ is the union of two disconnected piecewise linear paths joining the west and the north pole, and the south and the east pole respectively.

\begin{figure}[thbp]
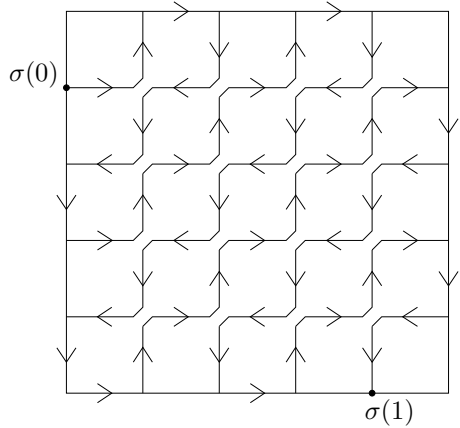
\caption{The function $\sigma$ in the proof of Theorem~\mref{NonSconn}.}\label{pathsigma}
\end{figure}
Let us now call $P_l$, with $1\leq l\leq N$, the finitely many points of $\G$ contained in $u^{-1}(\widetilde\G)$: by construction, $P_l$ are exactly the points of $\sigma\cap u^{-1}(\widetilde\G)$. We apply Proposition~\ref{injective}, finding the points $\u P_l$ on $\widetilde\G$ such that the function $v_\sigma$ of Definition~\ref{defvsigma} is injective. Our function $v:\G\to\u\Q$ will be such that $v(P_l)=\u P_l$ for every $1\leq l\leq N$.\par

Let us be more precise. We introduce the finite set $\A=\bigcup_{l=1}^N P_l \cup \bigcup_{i,\,j=0}^K V_{ij}$. Notice that $\G$ is a finite union of segments, each having both endpoints in $\A$ and no interior point in $\A$. Our aim is to select suitable points $\u V_{ij}\in\u\Q$, for $0\leq i,\,j\leq K$, and to define $v:\G\to\u\Q$ as the function which is linear on each of the segments defined above, and such that $v(P_l)=\u P_l$ and $v(V_{ij})=\u V_{ij}$ for every admissible $l,\,i,\,j$.\par

Of course, we only have to define the image of the points $V_{ij}$ for $1\leq i,\,j\leq K-1$, because for the other points it must necessarily be $\u V_{ij}=(i/K,j/K)$. Let then $V_{ij}$ be any of these points, and notice that $u(V_{ij})$ is in the interior of some rectangle $\u\RR$ of the grid $\widetilde\G$. We will define $\u V_{ij}$ as a suitable point in the interior of the same rectangle: notice that, as a consequence, for every $P\in\G$ the point $v(P)$ will be in the same rectangle of $\widetilde\G$ as $u(P)$: hence, $\|v-u\|_{L^\infty(\G)}\leq 2\sqrt 2 \eta<\eps$. Therefore, to conclude we only have to find points $\u V_{ij}$ so that the resulting function $v$ is injective.

Let us now fix $1\leq i,\,j\leq K-1$. Let $AB$ (resp., $CD$) be the smallest horizontal (resp., vertical) segment containing $V_{ij}$, and such that $A$ and $B$ (resp., $C$ and $D$) are of the form $P_l$. Since the injective function $v_\sigma$ does not intersect $\widetilde\G$ in the union of the open segments $AB$ and $CD$, then $\u A,\, \u B,\, \u C$ and $\u D$ are four distinct points on the boundary of the same rectangle $\u\RR$ of the grid $\widetilde\G$ containing $u(V_{ij})$. We state what follows.\\
{\bf Claim: }{\it The points $\u C$ and $\u D$ belong to the two different parts in which $\partial\u\RR$ is subdivided by $\u A$ and $\u B$.}\\
Indeed, there exists a path $\tilde\sigma$, adapted to $\widetilde\G$, which contains the four points $A,\, B,\, C$ and $D$, and for which $A$ and $C$, as well as $B$ and $D$, are consecutive points of $\tilde\sigma\cap u^{-1}(\widetilde\G)$: notice that we can simply take $\tilde\sigma=\sigma$ if the four points $A,\,B,\, C$ and $D$ belong to the four sides of $\G$ starting at $V_{ij}$, but also otherwise the existence of $\tilde\sigma$ is obvious. As a consequence, we have an injective function $v_{\tilde\sigma}:\tilde\sigma\to\u\Q$ from Proposition~\ref{injective}: if we call $\widetilde{\u A},\,\widetilde{\u B},\,\widetilde{\u C}$ and $\widetilde{\u D}$ the images of $A,\,B,\,C$ and $D$ under $v_{\tilde\sigma}$, by injectivity we know that the generalised segments $[\widetilde{\u A}\widetilde{\u C}]$ and $[\widetilde{\u B}\widetilde{\u D}]$ do not intersect themselves, so by Lemma~\ref{well-defined} we have that also the generalised segments $[\u{AC}]$ and $[\u{BD}]$ do not intersect themselves. Analogously, there exists another path adapted to $\widetilde\G$ containing the four points, and such that $A$ and $D$, as well as $B$ and $C$, are consecutive points on the path in $u^{-1}(\widetilde\G)$. Arguing as before, this implies that also $[\u{AD}]$ and $[\u{BC}]$ have no intersection, and finally, this concludes the claim.\par

Thanks to the above claim, we know that the generalized segments $[\u{AB}]$ and $[\u{CD}]$ have exactly a point of intersection, which is in the interior of $\u\RR$. We let $\u V_{ij}$ be this point, hence the function $v:\G\to\u\Q$ is now defined. Notice that, by construction, if $AB\subseteq\G$ is a segment with both endpoints and no interior point of the form $P_l$, then the image of the open segment $AB$ under $v$ coincides with the open generalised segment $[\u{AB}]$; up to now, we still do not know that $v$ is injective on $AB$. To conclude the proof, we have only to check that $v$ is injective.\par

Keep in mind that $\G$ is a finite union of essentially disjoint segments having both endpoints and no interior point in $\A$, and that $v$ is linear on each of these pieces. We take now any two of these segments which are disjoint, call them $XY$ and $ZW$: we have to show that $v(XY)\cap v(ZW)=\emptyset$. Similarly to what already done above, let us call $AB$ the smallest segment containing $XY$ with both endpoints of the form $P_l$, and $CD$ the one corresponding to $ZW$. Notice that $AB$ (resp., $CD$) can coincide with $XY$ (resp., $ZW$) or be larger, and even if $XY$ and $ZW$ have no intersection, the segments $AB$ and $CD$ could have a non-empty intersection, and even coincide. Let us consider the four different possible cases for $AB$ and $CD$.
\case{I}{The segments $AB$ and $CD$ are disjoint.}
Suppore first that the two segments $AB$ and $CD$ have empty intersection; in this case, we will prove that $v(AB)\cap v(CD)=\emptyset$, which of course implies $v(XY)\cap v(ZW)=\emptyset$. As noticed above, $v(AB)$ and $v(CD)$ are the (generalized) segments connecting $\u A=v(A)$ and $\u B=v(B)$, and $\u C=v(C)$ and $\u D=v(D)$ respectively. We can assume that these two segments are in the same rectangle $\u\RR$ of the grid $\widetilde\G$, since otherwise there is nothing to prove; we have only to check that $\u C$ and $\u D$ belong to the same of the two parts in which $\partial\u\RR$ is divided by $\u A$ and $\u B$. Again by Lemma~\ref{well-defined}, this follows by the existence of a path $\tilde\sigma$, adapted to $\widetilde\G$, which contains the four points $A,\,B,\,C$ and $D$, and for which $A$ and $B$ are two consecutive points on $\sigma$ of $u^{-1}(\widetilde\G)$, as well as $C$ and $D$; and in turn, the existence of such a path is again obvious.

\case{II}{The segments $AB$ and $CD$ coincide.}
Suppose now that $AB=CD$, and assume just to fix the ideas that this is a horizontal segment. Since $XY$ and $ZW$ are two disjoint subsegments of $AB$, it is enough to prove that $v$ is injective on $AB$. By construction, there exist $1\leq i^-,\,i^+,\, j \leq K-1$, with $i^-<i^+$, such that a vertex $V_{im}$ is contained in $AB$ if and only if $i^-\leq i \leq i^+$ and $m=j$. Let then $i^-\leq i\leq i^+$, and let us call $C_iD_i$ the shortest vertical segment containing $V_{ij}$ such that $C_i$ and $D_i$ are of the form $P_l$. By definition, $\u V_{ij}$ is the unique intersection point of the generalised segments $[\u{AB}]$ and $[\u C_i\u D_i]$: all we have to do, then, is to show that the points $\u V_{ij}$ are ordered from $\u A$ to $\u B$, on $[\u{AB}]$, when $i$ ranges from $i^-$ to $i^+$. Once again by Lemma~\ref{well-defined}, this immediately follows by taking the path $\tilde\sigma$ which is the union of the segments $V_{i^-0}V_{i^-K}$, then $V_{i^-K}V_{(i^-+1)K}$, then $V_{(i^-+1)K}V_{(i^-+1)0}$, then $V_{(i^-+1)0}V_{(i^-+2)0}$, and so on up to $V_{i^+0}V_{i^+K}$ or $V_{i^+K}V_{i^+0}$, depending whether the difference $i^+-i^-$ is even or odd.

\case{III}{The intersection between $AB$ and $CD$ is a common endpoint.}
Suppose now that $AB\cap CD$ is done by a single point, which is a common endpoint. Up to a change in the name of the points, we can then assume that $B=C$. Keep again in mind that $v(AB)$ and $v(CD)$ are the two generalised segments between $\u A$ and $\u B$, and between $\u C$ and $\u D$ respectively; since $B=C$, hence $\u B=\u C$, these two generalised segments are contained in the two different rectangles of the grid $\widetilde\G$ having $\u B=\u C$ on the boundary. As a consequence, the only intersection between $v(AB)$ and $v(CD)$ is the point $\u B=\u C$; moreover, as noticed above, the only point of $AB$ (resp., $CD$) having $\u B$ as image is $B$ (resp., $C$) itself. As a consequence, $v(XY)=v(ZW)$ can be non-empty only if $B=C$ belongs both to $XY$ and to $ZW$, which is in turn absurd since $XY\cap ZW=\emptyset$. Also this case is then concluded.

\case{IV}{The intersection between $AB$ and $CD$ is a point, which is not a common endpoint.}
The last possible case is that the two segments $AB$ and $CD$ have a single intersection point, which is internal to at least one of them, hence automatically to both. This means that necessarily $AB\cap CD=V_{ij}$ for a suitable $1\leq i,\,j\leq K-1$, and $AB$ and $CD$ are the shortest horizontal and vertical segment containing $V_{ij}$ and with both endpoints of the form $P_l$. We know then that the intersection $v(AB)\cap v(CD)$ consists of a single point, namely $\u V_{ij}$. Moreover, since by Case~II we know that $v$ is injective both on $AB$ and on $CD$, then $V_{ij}$ is the unique point of $AB$, as well as of $CD$, having $\u V_{ij}$ as image under $v$. Therefore, the only possibility for the intersection $v(XY)\cap v(ZW)$ not to be empty, is that $V_{ij}$ is contained both in $XY$ and in $ZW$, which is in turn impossible since $XY$ and $ZW$ have empty intersection. We have then concluded also this final case.
\end{proof}

\begin{remark}\label{onlyptsgsq}
A quick look to the proof of Theorem~\mref{NonSconn} ensures that what we really use is the fact that $u^{-1}(\u \C)$ is a closed, connected set which does not disconnect $\Q$ whenever $\u\C\subseteq\u\Q$ is a point, or a square, or a connected piece of the boundary of a square. Actually, we only need the assumption about the points in order to introduce the function $v_\sigma$ in Definition~\ref{defvsigma}, while the assumption about squares and parts of the boundary of squares is only used --several times-- in order to prove Proposition~\ref{injective} and its preparatory Lemmas~\ref{4distinct} and~\ref{correctorientation}. An interesting question is then whether given $u\in W^{1,p}$ it suffices, in order to satisfy the no-crossing condition, that the counterimage of every point in $\u\Q$ is a connected set which does not disconnect $\Q$.
\end{remark}

\subsection{Monotone continuous maps satisfy the conclusion of Theorem~\mref{NonSconn}}

In this short final section we show that a continuous, $W^{1,p}$ map is monotone if and only if it satisfies the assumptions of Theorem~\mref{NonSconn}, hence Theorem~\mref{NonSconn} is somehow a reasonable extension to the non-continuous case of the results of~\cite{IO}. In particular, thanks to Theorem~\mref{Caratt}, this extends Theorem~\ref{thm:summ}~(ii) to the case when \(p=1\), see also~\cite{CORT} for further extensions. Notice that, by the results of Young, it is clear that a monotone map satisfies the no-crossing condition, so the fact that it can be approximated by diffeomorphisms follows directly from Theorem~\mref{Caratt}, even without using Theorem~\mref{NonSconn}.

\begin{lemma}\label{condThC}
Let $u:\Q\to\u\Q$ be a continuous, monotone function which is equal to the identity on \(\partial \Q\). Then, for every closed, connected set $\u C\subseteq\u\Q$ such that $\u\Q\setminus\u C$ is connected, the set $C=u^{-1}(\u C)\subseteq\Q$ is also a closed, connected set such that $\Q\setminus C$ is connected.
\end{lemma}
\begin{proof}
Let \(\u C\) be a closed connected subset of \(\u \Q\) such that $\u\Q\setminus\u C$ is connected. Clearly, \(C\) is closed. Assume that there are two disjoint closed sets \(A\) and \(B\) such that \(u^{-1}(\u C)=A\cup B\) and let us define
\[
\u A=\{ \u P\in \u C \textrm{ such that \(u^{-1}(\u P)\subseteq A\)}\}\qquad \textrm{and} \qquad \u B=\{ \u P\in \u C \textrm{ such that \(u^{-1}(\u P)\subseteq B\)}\}.
\]
For every $\u P\in\u\C$, by monotonicity the set \(u^{-1}(\u P)\) is connected, hence it must be entirely containd either in $A$ or in $B$. Thus, $\u A$ and $\u B$ are a closed partition of \(\u C\) and then, by connectedness, either \(\u A\) or \(\u B\) is empty, which in turn implies that either \(A\) or \(B\) is empty and thus that \(u^{-1}(\u C)\) is connected. 

Let now \(P, Q\in \Q\setminus u^{-1}(\u C)$ and let \(\u P=u(P)\) and \(\u Q=u(Q)\). Since \(\u P, \u Q\in \u\Q\setminus\u C\) and the last is set pathwise connected (since \(\u C\) is closed), we can find a continuous path \(\gamma\) in $\u\Q$ connecting \(\u P\) and \(\u Q\) and not intersecting \(\u C\). From the first part \(u^{-1}(\gamma)\) is a connected set containing \(P\) and \(Q\) but not intersecting \(C\), which shows that \(P\) and \(Q\) belong to the same connected component of \(\Q\setminus u^{-1}(\u C)\). Being \(P\) and \(Q\) arbitrary points, this concludes the proof.
\end{proof}

\section{Counterexamples\label{sec:Ex}}

This last section is devoted to show four counterexamples, which help to clarify the contrast between the no-crossing condition, the INV condition, and the sufficient condition that we use in Theorem~\mref{NonSconn}. In particuar, we will obtain the following situations.\par

In the first example, Section~\ref{primo}, we show that the INV condition is very poor without the additional assumption that $\det Du>0$ almost everywhere. Let us be more precise: in the paper~\cite{MS}, which introduced the INV condition, many of the results for INV functions were proved under this additional assumption, which is physically meaningful, since a deformation not satisfying it has infinite energy in most of the models. In particular, in~\cite[Theorem~9.1]{MS} it is proved that the INV condition with the additional assumption $\det Du>0$ a.e. is stable under a regular change of variables; using the notation of this paper, we can equivalently say that the INV condition implies the INV$^+$ one. Our example of Section~\ref{primo} shows an INV function $u$, not satisfying the assumption $\det Du>0$ a.e., which fails to be INV$^+$; our function is actually INV and Lipschitz continuous but not monotone, and the counter-image of some points is disconnected. In particular, it is an INV function which is not approximable by diffeomorphisms.\par

In the second example, Section~\ref{secondo}, we show a function which is not approximable by diffeomorphisms, but which satisfies the INV$^+$ condition, and for which the counter-image of any point is connected. This example shows that even the condition INV$^+$ is too weak to guarantee that a function is limit of diffeomorphisms. In addition, also the assumption that the counter-image of any point is connected is not sufficient to be limit of diffeomorphisms, while it is necessary as noted in Remark~\ref{contconn}. Keep in mind that, in the particular case of continuous functions, the condition that the counter-image of points is connected is the monotonicity condition, hence it is both necessary and sufficient, as discussed in the Introduction.

In the third example, Section~\ref{terzo}, we show that the INV condition is not sufficient even with the additional assumption that $\det Du>0$ almost everywhere. In fact, we present a function $u$ for which $\det Du>0$ a.e., which is INV (hence also INV$^+$), but which is not limit of diffeomorphisms.

In the last example, Section~\ref{quarto}, we show that the sufficient condition of Theorem~\mref{NonSconn} is far from being necessary. In fact, it is very simple to present examples of limits of diffeomorphisms for which the counter-image of some points disconnects $\Q$. However, a quick inspection to the proof of the Theorem clearly shows that the assumption on the counter-image of points is actually needed only for points which are in the arrival grid $\widetilde\G$; as a consequence, if for a function the set of points with counter-image which disconnects $\Q$ is $\H^1$-negligible, the proof works exactly in the same way, up to chose an arrival grid which does not meet any of those points. In our counterexample, instead, we show a function which is limit of diffeomorphisms, but for which the set of points whose counter-image disconnects $\Q$ is a whole segment.

\subsection{The first counterexample\label{primo}}

Let us construct our first example. We consider the rectangles $\RR=[0,4]\times [-10,10]$ and $\RR^-=[1,3]\times [-9,9]$, and the segment $\SS=[1,3]\times\{0\}$. Let $v:\RR\to\RR$ be a smooth function for which $v(x,y)=(x,0)$ on $\RR^-$, and which is a diffeomorphism between $\RR\setminus \RR^-$ and $\RR\setminus \SS$, coinciding with the identity on $\partial\RR$. It is obvious that this function is a limit of diffeomorphisms, and that it satisfies the INV condition. Our ``special'' function is instead $u:\RR\to\RR$ given by
\[
u(x,y)=\left\{
\begin{array}{ll}
v(x,y) & \hbox{if $(x,y)\in\RR\setminus\RR^-$},\\
(\varphi(x,y),0) & \hbox{if $(x,y)\in\RR^-$},
\end{array}\right.
\]
where $\varphi:\RR^-\to [1,3]$ is a continuous function such that $\varphi(x,y)=x$ whenever $(x,y)\in\partial\RR^-$, so that $u$ is continuous. To define $\varphi$, we divide the rectangle $\RR^-$ in the three essentially disjoint parts $\A,\,\B$ and $\C$ as in Figure~\ref{firstcounter}.
\begin{figure}[thbp]
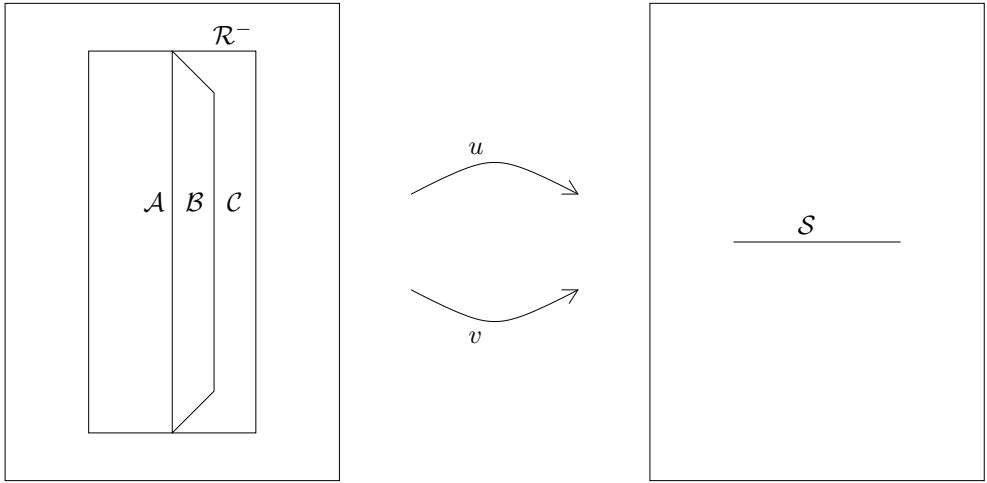
\caption{The function $u$ for the example of Section~\ref{primo}.}\label{firstcounter}
\end{figure}
More precisely, $\A$ is given by all the points $(x,y)\in\RR^-$ with $1\leq x\leq 2$, $\B$ by the points $(x,y)\in\RR^-$ with $2\leq x\leq \min\{11-|y|,2.5\}$, and $\C$ by the other points. The function $\varphi$ is identified by the properties that $\varphi(x,y)=x$ in $\A$, that $\partial\varphi/\partial x=-1$ in $\B$, and that $\partial\varphi/\partial x$ is constant on each horizontal segment contained in $\C$. Hence, on any horizontal segment $[1,3]\times \{y\}\subseteq\RR^-$ one has that $\varphi$ is linear with slope $1$ between $1$ and $2$, then linear with slope $-1$ between $2$ and $\min\{11-|y|,2.5\}$, and then again linear with some strictly positive slope depending on $|y|$ between $\min\{11-|y|,2.5\}$ and $3$.\par

It is clearly impossible to obtain the function $u$, which is Lipschitz continuous then in any $W^{1,p}$ space, as a limit of diffeomorphisms. In addition, there are points in $\RR$ whose counter-image is not even connected, for instance the point $(1.5,0)$. Nevertheless, we will show in a moment that $u$ satisfies the INV condition. It is easy to observe that $u$ does not satisfy the INV$^+$ condition: just take a curve $\gamma$ contained in the interior of the union of the two regions $\B$ and $\C$, and which contains the whole vertical segment between $\B$ and $\C$. Then, any point of this vertical segment is in the interior of the curve $\gamma$, but its image is $(1.5,0)$, which is external to $u(\gamma)$, since $u(\gamma)$ is a segment $[a,b]\times \{0\}$ with $1.5<a<b<3$.

\begin{lemma}
The function $u:\RR\to\RR$ defined above satisfies the INV condition.
\end{lemma}
\begin{proof}
Let $\Gamma$ be a circle contained in $\RR$. Notice that $\Gamma$ cannot intersect both the upper and the lower segment of $\partial\RR^-$; as a consequence, also keeping in mind that $u$ and $v$ are continuous, we obtain that the points which are internal to $u(\Gamma)$ and to $v(\Gamma)$ are the same, and in particular no such point belongs to the segment $\SS$. Let then $P\notin\RR^-$: since $v$ satisfies the INV condition and $u(P)=v(P)$ we have that, if $P$ is internal to $\Gamma$, then $u(P)$ is not external to $u(\Gamma)$, while if $P$ is external to $\Gamma$, then $u(P)$ is not internal to $u(\Gamma)$.\par

To conclude checking the INV condition, let us then take $P=(x,y)\in\RR^-$. Since $u(P)$ belongs to the segment $\SS$, then it cannot be internal to $u(\Gamma)$. We can then assume that $P$ is internal to $\Gamma$, and we have to exclude that $u(P)$ is external to $u(\Gamma)$: since $\varphi$ is continuous, to do so it is enough to find two points $A',\,B' \in \Gamma$ such that $\varphi(A')\leq \varphi(P)\leq \varphi(B')$. Let $AB$ be the horizontal segment containing $P$ with $A,\, B\in\Gamma$.\par

Assume first that $x\leq 2$: in this case, if $A\in\RR^-$ then $\varphi(A)<\varphi(P)$, and if $A\notin\RR^-$ then there is another $A'\in\Gamma$ belonging to the left side of $\RR^-$, so $\varphi(A')<\varphi(P)$. Thus, if $\varphi(B)\geq \varphi(P)$ we are done. On the other hand, if $\varphi(B)<\varphi(P)$, then the segment $AB$ crosses the vertical segment $\partial\A\cap\partial\B$, thus $\Gamma$ must cross the same segment; hence, there is a point $B'\in\Gamma\cap\partial\A\cap\partial\B$, which means $\varphi(B')=2$, and since $x\leq 2$ implies $\varphi(P)\leq 2$ we are done.\par

Assume now that $2< x < \min\{11-|y|,2.5\}$. The existence of a point $B'\in\Gamma$ with $\varphi(B')\geq \varphi(P)$ is clear: indeed, either $\varphi(A)>\varphi(P)$, or the segment $AB$ crosses the vertical segment $\partial\A\cap\partial\B$, so the argument above again applies. If $\varphi(B)<\varphi(P)$ we are done; otherwise, the segment $AB$ crosses $\partial\B\cap\partial\C$ in the point $(\tilde x,y)$, with $\tilde x=\min\{11-|y|,2.5\}$. As a consequence, we find another point $A'=(x',y')\in \Gamma\cap\partial\B\cap\partial\C$, and in particular there is one with $x'\geq \tilde x$: this implies $\varphi(A')<\varphi (P)$, so we are again done.

Let us finally assume that $x\geq\min\{11-|y|,2.5\}$. This time, it is clear that $\varphi(B)>\varphi (P)$ (or, if $B\notin \RR^-$, that there is some $B'\in\Gamma$ belonging to the right side of $\RR^-$, thus with $\varphi(B')>\varphi(P)$); moreover, either $\varphi(A)<\varphi(P)$, or the segment $AB$ crosses $\partial\B\cap\partial\C$, so the existence of a point $A'$ with $\varphi(A')\leq \varphi(P)$ follows by the above argument. The proof is then concluded.
\end{proof}

\subsection{The second counterexample\label{secondo}}

Let us now construct our second example. We consider the domain $\Omega=\{|x|<2\}$, the balls $\B=\{|x|< 1\}$ and $\B^-=\{|x|< 1/2\}$, and the unit segment $\SS=[0,1]\times \{0\}$, and we aim to define a function $u:\Omega\to\Omega$ which equals the identity on $\partial\Omega$. First of all, we define $u$ outside $\B$. For $\eps$ small, we let the image under $u$ of the circle with radius $1+\eps$ be the outer boundary $\Gamma_\eps$ of an $\eps$-neighborhood of the segment $\SS$, parametrized with two different speeds. More precisely, as Figure~\ref{Fig:49} (left) shows, the part of length $\eps$ around the point $(-1-\eps,0)$ is sent, with constant speed, to the whole $\Gamma_\eps$ except an $\eps$-neighborhood of the point $(1+\eps,0)$, and the remaining part of the circle is sent, again with constant speed, to the remaining part of $\Gamma_\eps$. It is clearly possible to extend the definition of $u$ to the rest of $\Omega\setminus\B$ in such a way that $u$ is a $W^{1,p}$ function on $\Omega\setminus\B$ coinciding with the identity on $\partial\Omega$ for every $p<2$. Notice that $u$ is a homeomorphism of $\Omega\setminus\B$ onto $\Omega\setminus\SS$, and that the whole boundary of $\B$ except the ``west pole'' $W=(-1,0)$ is sent on $(1,0)$, while the generalized image of $W$ is the whole segment $\SS$.\par
Let us now define $u$ in the annulus between $\partial\B$ and $\partial\B^-$. First of all, the image of the circle $\partial\B^-$ is the segment $\SS$ done four times with constant speed: more precisely, for every $\theta\in\S^1$ we let $u(\frac 12 \,\cos\theta,\frac 12 \sin\theta)=(\varphi(\theta),0)$, where $\varphi:\S^1\to [0,1]$ is the Lipschitz function such that
\begin{align*}
\varphi(0)=\varphi(\pi)=1\,, && \varphi(\pi/2)=\varphi(-\pi/2)=0\,,
\end{align*}
and $|\varphi'|=2/\pi$. Then, let us consider the circle $\{|x|=1-\eps\}$ for any $0<\eps<1/2$: in this case, for every $\theta\in\S^1$ we let $u\big((1-\eps)(\cos\theta,\sin\theta)\big)=\varphi(\tau_\eps(\theta))$, where $\tau_\eps:\S^1\to\S^1$ is the Lipschitz homeomorphism which maps at constant speed the arc $L_\eps$ centered at $\pi$ of width $2\pi\eps$ on the whole $\S^1$ except the arc $R_\eps$ centered at $0$ of width $2\pi\eps$, while $\S^1\setminus L_\eps$ is mapped, again at constant speed, onto $R_\eps$. Notice that $\tau_{1/2}$ is the identity, so the function $u$ is continuous on $\partial\B^-$, and moreover it is still in $W^{1,p}$ for every $p<2$.\par
\begin{figure}[thbp]
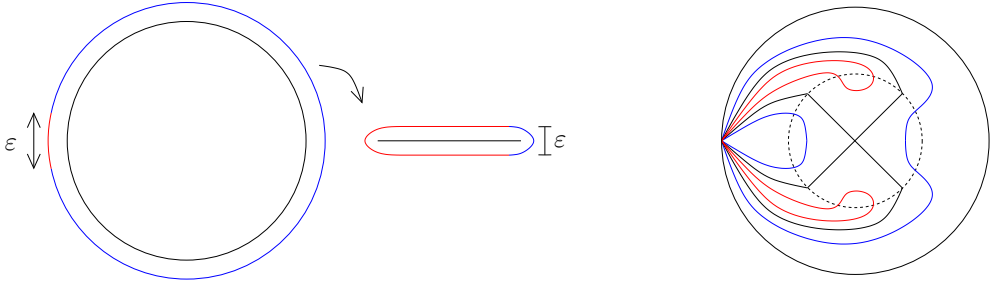
\caption{Definition of $u$ outside the circle $\partial\B$ (left); counter-image of $(x,0)$ for $0<x<1/2$ in red, for $x=1/2$ in black, for $1/2<x<1$ in blue (right). The dashed circle is $\partial\B^-$.}\label{Fig:49}
\end{figure}
To conclude, we have to define $u$ inside $\B^-$, and we simply shrink everything onto the point $(1/2,0)$. More precisely, for every $0<\rho<1/2$ and every $\theta\in\S^1$ we set
\[
u(\rho\cos\theta,\rho\sin\theta)= \frac 12 + \bigg(\varphi(\theta)-\frac 12\bigg) 2\rho\,.
\]
The resulting function $u:\Omega\to\Omega$ is then a $W^{1,p}$ function, coinciding with the identity on the boundary; it is smooth except at the point $W=(-1,0)$. It is easy to see that this function is not limit of diffeomorphisms, and this also follows from Theorem~\mref{Caratt} since it is simple to see that $u$ does not satisfy the no-crossing condition. Nevertheless, we can prove the following.

\begin{lemma}
The function $u$ defined above satisfies the INV$^+$ condition. Moreover, the counter-image of every point is connected.
\end{lemma}
\begin{proof}
Let us first consider the counter-image of a point $\u P=(x,y)\in\Omega$. If $\u P\notin\SS$, then $u^{-1}(\u P)$ is a single point in $\Omega\setminus\B$, hence connected. Assume instead that $\u P\in\SS$, so $\u P=(x,0)$ for some $0\leq x\leq 1$. Let us for a moment consider the case $0<x<1$: in this case, $\u P$ has exactly four counter-images on the circle of radius $1/2$; therefore, for any radius $\rho$ between $1/2$ and $1$, the point $\u P$ has still four counter-images on the circle of radius $\rho$, and these are four points which all converge to $W$ when $\rho\to 1$. When the radius $\rho$ goes from $1/2$ to $0$, instead, the counter-images remain four for a while, in particular until $\rho$ becomes equal to $\rho_{\rm min}=|x-1/2|$; then, on the boundary of the $\rho_{\rm min}$, the four points have become only the two points $(0,\pm\rho_{\rm min})$ if $0<x<1/2$, only the two points $(\pm \rho_{\rm min},0)$ if $1/2<x<1$, and only the origin if $\rho_{\rm min}=0$, that is, $x=1/2$. The counter-image does not have points in the open ball with radius $\rho_{\rm min}$. Figure~\ref{Fig:49} (right) depicts the counter-image of $\u P$ in these three cases. In the limiting case $x=0$, the counter-image is done only by two points on each circle $\{|x|=\rho\}$ for every $1/2\leq \rho<1$, which again converge to $W$ when $\rho\to 1$; in the limiting case $x=1$, instead, the counter-image is the whole circle $\partial\B$ together with the two segment $WA$ and $BC$ with $A=(-1/2,0)$, $B=(1/2,0)$ and $C=(1,0)$.\par

As a consequence of the above characterization, we have that the counter-image of any point of $\Omega$ is a connected set. The validity of the INV$^+$ condition follows then by an obvious geometric argument.
\end{proof}

\subsection{The third counterexample\label{terzo}}

We can now present our third counterexample, which is an INV map $u\in W^{1,p}([-2,2]^2;[-2,2]^2)$ such that $\det Du >0$ almost everywhere, and which does not satisfy the no-crossing condition.\par

First of all, we need to present the auxiliary function $v:[-2,2]^2\to [-2,2]^2$, coinciding with the identity on the boundary. For any $0\leq t\leq 1$, we give four paths $\alpha_t,\, \beta_t,\, \gamma_t$ and $\sigma_t$, see Figure~\ref{Fig:219} for an illustration. The path $\alpha_t$ is simply the segment between $(-2,t)$ and $(0,0)$, while $\sigma_t$ is the segment between $(0,0)$ and $(2,t)$. Instead, the paths $\beta_t$ and $\gamma_t$ are two loops, starting and ending at $(0,0)$, lying on the half-space $\{y\geq 0\}$, percurred in the clockwise sense. In addition, these paths depend smoothly on $t$, each path $\beta_t$ (resp., $\gamma_t$) is contained in the internal part of $\beta_s$ (resp., $\gamma_s$) when $t<s$, the paths $\beta_0$ and $\gamma_0$ consist of the sole point $(0,0)$, and the intersection between any two paths of the form $\alpha_t,\, \beta_t,\, \gamma_t$ and/or $\sigma_t$ is always the sole point $(0,0)$. Finally, each path $\beta_t$ is on the left of each path $\gamma_s$, as in the figure.\par
\begin{figure}[thbp]
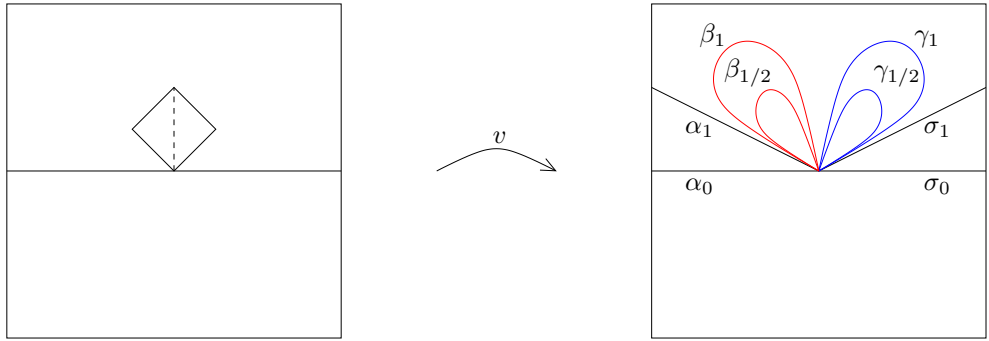
\caption{Definition of $v$ in the example of Section~\ref{terzo}.}\label{Fig:219}
\end{figure}
For any $0\leq t\leq 1$, we call now $x_t=\min\{ t, 1-t\}$, and we introduce the points $A_t=(-2,t)$, $B_t=(-x_t,t)$, $C_t=(0,t)$, $D_t=(x_t,t)$ and $E_t=(2,t)$; notice that these points are all distinct except $B_0=C_0=D_0=(0,0)$ and $B_1=C_1=D_1=(0,1)$. We can then define $v$ on the horizontal segment $A_tE_t$ as follows: on the segment $A_tB_t$ (resp., $B_tC_t$, $C_tD_t$, $D_tE_t$), the image of $v$ is the curve $\alpha_t$ (resp., $\beta_t$, $\gamma_t$, $\sigma_t$), parametrized at constant speed. Notice that, up to now, the function $v$ is smooth except at the point $(0,1)$, where it is discontinuous. Moreover, by construction the function $v$ belongs to $W^{1,p}([-2,2]\times [0,1])$ for any $p<2$. We can now extend $v$ to the whole square $[-2,2]^2$: on the rectangle $[-2,2]\times [-2,0)$ we can simply set $v={\rm Id}$, and it is simple to observe the existence also of a $W^{1,p}$ extension of $v$ on the rectangle $[-2,2]\times (1,2]$ which is a homeomorphism onto the image. The resulting function $v$ belongs to $W^{1,p}([-2,2]^2)$, it is smooth except at the point $(0,1)$, where it is discontinuous, it is injective except at $v^{-1}(0,0)$, and it is clearly a strong $W^{1,p}$ limit of diffeomorphisms, in particular it satisfies the INV and the INV$^+$ conditions, which are equivalent since by construction $\det Dv>0$ almost everywhere.\par

Let us now call $\T^\pm$ the two triangles having vertices in $(0,0)$, $(0,1)$ and $(\pm 1/2,1/2)$, and let $\psi:[-2,2]^2 \to [-2,2]^2$ the function defined as
\[
\psi(s,t) = \left\{\begin{array}{ll}
(s+x_t,t) & \hbox{if $(s,t)\in\T^-$}\,, \\
(s-x_t,t) & \hbox{if $(s,t)\in\T^+$}\,, \\
(s,t) & \hbox{if $(s,t)\notin\T^+\cup\T^-$}\,.
\end{array}\right.
\]
Notice that $\psi$ is a bijection moving $\T^-$ on $\T^+$ and vice-versa. We can then define our function $u:[-2,2]^2\to [-2,2]^2$ simply as $u=v\circ \psi$. Since $v(x)=(0,0)$ for every point $x\in \partial \T^-\cup\partial \T^+$ except at the discontinuity point $(0,1)$, also $u$ belongs to $W^{1,p}([-2,2]^2)$; moreover, $u$ is also continuous except at $(0,1)$, injective except at $\partial\T^-\cup\partial\T^+$, and coincides with the identity on the boundary, and we have that also $\det Du>0$ almost everywhere (so, the INV and the INV$^+$ conditions are equivalent for $u$, as well as for $v$). On the other hand, $u$ does not satisfy the no-crossing condition, so it is not a limit of diffeomorphisms. To observe the failure of the no-crossing condition, it is enough to consider the image of an horizontal segment joining $(-2,t)$ with $(2,t)$ for any $0<t<1$: it is a continuous mappings, namely, the union of the four paths $\alpha_t,\, \gamma_t,\, \beta_t$ and $\sigma_t$, and it is clearly impossible to make it injective with a small uniform variation because $\beta_t$ has been postponed after $\gamma_t$. Nevertheless, we can show the validity of the INV condition for $u$.

\begin{lemma}
The function $u$ defined above satisfies the INV condition.
\end{lemma}
\begin{proof}
Let $\tau:\S^1\to [-2,2]^2\setminus (0,1)$ be a circle: we have to prove that all the points internal (resp., external) to $\tau$ have degree $1$ (resp., $0$) with respect to $u(\tau)$, unless they are contained on $u(\tau)$ itself. We check this in few possible cases.
\case{I}{If $\tau$ does not intersect $\partial \T^-\cup\partial \T^+$.}
In this case, the property comes immediately from the analogous property of $u$. Indeed, any point $x\in [-2,2]^2$ is internal (resp., external) to $\tau$ if and only if $\psi(x)$ is internal (resp., external) to $\psi(\tau)$, which in this case is also an injective closed path. Then, the degree of $u(x)=v(\tau(x))$ with respect to $u(\tau)=v(\psi(\tau))$ behaves correctly because $u$ satisfies the INV condition.
\case{II}{If $\tau$ intersects only $\partial\T^-$, or only $\partial\tau^+$.}
Let us now assume that $\tau$ intersects only $\partial\T^-$, the case when $\tau$ intersects only $\partial\T^+$ is clearly identical. Notice that $u(\tau)=v(\psi(\tau))$. However, this time $\psi(\tau)$ is done by two disconnected open curves, one outside $\T^-\cup\T^+$ and the other one inside $\T^-$; let us call them $\tau_a$ and $\tau_b$. As in Figure~\ref{Fig:229} (above), we denote by $\tau_1$ and $\tau_2$ the two injective, closed paths obtained adding to $\tau_a$ (resp., $\tau_b$) a segment contained in $\partial\T^-$ (resp., $\partial\T^-\cap\partial\T^+$). Notice that $u$ is constantly $(0,0)$ on both segments, hence $u(\tau)=v(\tau_a\cup\tau_b)=v(\tau_1)\cup v(\tau_2)$. Moreover, $v(\tau_1)$ and $v(\tau_2)$ are two injective curves, whose only intersection is the point $(0,0)$. In particular, for any $y\in [-2,2]^2\setminus u(\tau)$ one has that
\begin{equation}\label{sumdeg}
\deg(y,u(\tau))=\deg(y,v(\tau_1))+\deg(y,v(\tau_2))\,.
\end{equation}
Let now $x\in [-2,2]^2$ be any point. If it belongs to $\partial\T^-\cup\partial\T^+$, then $u(x)=(0,0)\in u(\tau)$, hence there is nothing to prove. If $x$ is internal to $\tau$ and not on $\partial\T^-$, then $y=\psi(x)$ is internal to one between $\tau_1$ and $\tau_2$, and external to the other one. Since $v$ satisfies the INV$^+$ condition, we can apply it to $\psi(x)$ with respect to $\tau_1$ and $\tau_2$, so we find that one between $\deg(u(x),v(\tau_1))=\deg(v(y),v(\tau_1))$ and $\deg(u(x),v(\tau_2))=\deg(v(y),v(\tau_2))$ is $0$ and the other one is $1$, thus by~(\ref{sumdeg}) we get $\deg(u(x),u(\tau))=1$, as required since $x$ is internal to $\tau$.\par
\begin{figure}[thbp]
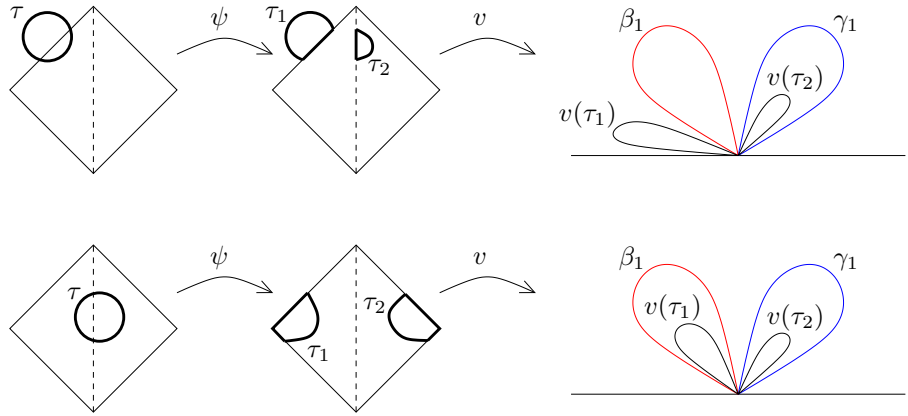
\caption{The paths $\tau,\,\tau_1$ and $\tau_2$ in Case~II (above) and in Case~III (below).}\label{Fig:229}
\end{figure}
Finally, if $x$ is external to $\tau$ and not on $\partial\T^-\cup\partial\T^+$, then $\psi(x)$ is external to both $\tau_1$ and $\tau_2$, so an analogous argument ensures that $\deg(u(x),u(\tau))=0$, so we are done.
\begin{figure}[thbp]
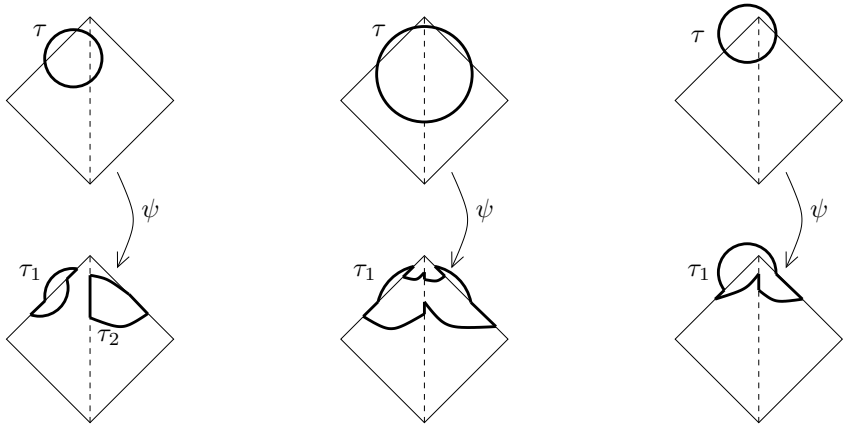
\caption{The paths $\tau,\,\tau_1$ and $\tau_2$ in Case~IV (left) and in Case~V (center and right).}\label{Fig:229b}
\end{figure}
\case{III}{If $\tau$ intersects $\partial\T^-\cap\partial\T^+$ but not $\partial\T^-\setminus \partial\T^+$ nor $\partial\T^+\setminus \partial\T^-$.}
This case is completely analogous to Case~II, as a quick look to Figure~\ref{Fig:229} (below) shows.
\case{IV}{If $\tau$ intersects $\partial\T^-\cap\partial\T^+$ and one between $\partial\T^-\setminus \partial\T^+$ and $\partial\T^+\setminus \partial\T^-$.}
This case is again very similar to that of Case~II, see Figure~\ref{Fig:229b} (left); let us assume, just to fix the ideas, that $\gamma$ intersects $\partial\T^-\setminus\partial\T^+$ and not $\partial\T^+\setminus\partial\T^-$. This time, $\psi(\tau)$ is done by four disconnected open arcs, namely, an arc $\tau_a$ outside $\T^-\cup\T^+$, two arcs $\tau_b$ and $\tau_d$ inside $\T^+$ and an arc $\tau_c$ inside $\T^-$. As in the Figure, we denote by $\tau_1$ the closed, injective curve obtained adding to $\tau_a\cup\tau_c$ two segments on $\partial\T^-\setminus\partial\T^+$. Similarly, we denote by $\tau_2$ the closed, injective curve obtained adding to $\tau_b\cup\tau_d$ a segment on $\partial\T^-\cap\partial\T^+$ and another one on $\partial\T^+\setminus\partial\T^-$. It is then enough to ovserve that $\tau_1$ and $\tau_2$ are disjoint and that~(\ref{sumdeg}) holds as before, and then to argue as in the preceding steps.
\case{V}{If $\tau$ intersects $\partial\T^-\cap\partial\T^+$, $\partial\T^-\setminus \partial\T^+$ and $\partial\T^+\setminus \partial\T^-$.}
This case is now very simple to handle, it is enough to argue more or less exactly as in the last cases. As Figure~\ref{Fig:229b} (center and right) shows, there are two possible subcases, namely, if the point $(0,1)$ is internal or external to $\tau$. In the first case, $\psi(\tau)$ is done by six different pieces, and we can consider a single path $\tau_1$ by adding six intervals to it; in the second case, $\psi(\tau)$ is done by three pieces, and we obtain a single path $\tau_1$ by adding three intervals.
\end{proof}

\subsection{The fourth counterexample\label{quarto}}

We give now our last counterexample, which is the simplest one. It ensures that the sufficient condition in Theorem~\mref{NonSconn} is not necessary, not even up to $\H^1$-negligible sets. More precisely, we exhibit a function $u:[-2,2]^2\to [-2,2]^2$, coinciding with the identity on the boundary, being the $W^{1,p}$ limit of diffeomorphisms for $p<2$, but for which $u^{-1}(\u P)$ is a connected set which disconnects the domain for every $\u P$ in a segment.\par

For every $t\in [0,1/2]$, we consider the $2$-dimensional quadrilateral $\Q_t$ having as vertices the points $(0,0)$, $(2t,0)$ and $(t,\pm t \tan t)$. Notice that, for every $s<t$, $\Q_s$ is contained in $\Q_t$, having the origin as only common point. We start defining $u$ inside the biggest quadrilateral $\Q_{1/2}$ as the function which, for every $0<t<1/2$, maps the whole $\partial\Q_t$ on the point $(2t,0)$; Figure~\ref{Fig:69} depicts the function. Notice that this function is continuous except at $(0,0)$, whose image is the whole segment $[0,1]\times \{0\}$. Exactly as in the example of Section~\ref{secondo}, it is possible extend $u$ as a homeomorphism of $[-2,2]^2\setminus \Q_{1/2}$ onto $[-2,2]^2\setminus [0,1]\times\{0\}$, belonging to the Sobolev space $W^{1,p}([-2,2]^2\setminus \Q_{1/2})$ for $p<2$ and coinciding with the identity on $\partial([-2,2]^2)$. Notice that the whole $u$ is a Borel function, continuous up to the point $(0,0)$, and which is easily seen to satisfy the no-crossing condition; nevertheless, for every $0<t<1$, the counter-image of $(t,0)$ is the polygon $\partial\Q_{t/2}$, which disconnects $[-2,2]^2$. By Theorem~\mref{Caratt}, the function $u$ is then limit of diffeomorhisms if and only if it belongs to $W^{1,p}([-2,2]^2)$, and in turn this property is true, as we prove now.

\begin{figure}[thbp]
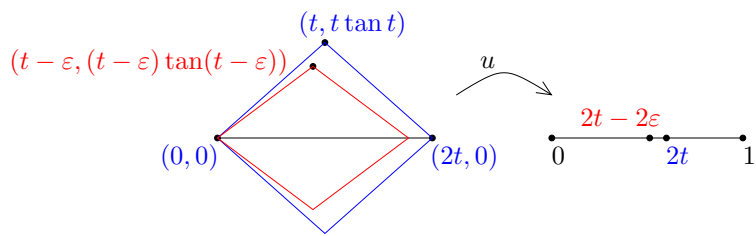
\caption{Definition of $u$ and two quadrilaterals $\Q_t$ and $\Q_{t+\eps}$.}\label{Fig:69}
\end{figure}

\begin{lemma}
The function $u$ defined above belongs to $W^{1,p}([-2,2]^2)$.
\end{lemma}
\begin{proof}
Since $u$ is continuous except at $(0,0)$, and we already know that it is of class $W^{1,p}$ on $[-2,2]^2\setminus \Q_{1/2}$, we have only to show that $u\in W^{1,p}$ inside $\Q_{1/2}$; in fact, it is clear that $u$ is smooth in the interior of $\Q_{1/2}$, so the possible problems are only close to $(0,0)$. Let $P$ be a point on $\partial\Q_t$, having distance $\sigma$ from the origin, and assume for a moment that $P$ belongs to one of the two sides originating at $(0,0)$: then, for every $\eps\ll t$ the point $P$ has distance $\sigma \tan \eps$ from $\partial \Q_{t+\eps}$, hence $|Du(P)|=1/\sigma$; a trivial geometrical argument shows that, if $P$ belongs to one of the two other sides of $\partial\Q_t$, then we have the stronger estimate $|Du(P)|<1/\sigma$, hence in general $|Du(P)|\leq 1/\sigma$ at every $P\in\partial\Q_t$ having distance $\sigma$ from the origin. Observe that this estimate does not depend on $t$. As a consequence, to check that $Du\in L^p(\Q_{1/2})$ it is enough to keep in mind that $p<2$ and to calculate
\[
\int_{\Q_{1/2}} |Du|^p \, dx \leq \int_{\theta=-1/2}^{1/2} \int_{\sigma=0}^1 \frac 1{\sigma^p}\, \sigma\, d\theta\,d\sigma
= \int_{\sigma=0}^1 \frac 1{\sigma^{p-1}}\, d\sigma < +\infty\,.
\]
\end{proof}

\section*{Acknowledgments}
G.D.P. is supported by the MIUR SIR-grant ``Geometric Variational Problems'' (RBSI14RVEZ).

\end{document}